\documentclass[12pt]{article}

\usepackage[a4paper,margin=2.7cm]{geometry}
\usepackage{graphicx}%
\usepackage{multirow}%
\usepackage{amsmath,amssymb,amsfonts}%
\usepackage{amsthm}%
\usepackage{mathrsfs}%
\usepackage[title]{appendix}%
\usepackage{xcolor}%
\usepackage{textcomp}%
\usepackage{booktabs}%
\usepackage{enumitem}
\usepackage{float}
\usepackage[authoryear,round]{natbib}
\usepackage{booktabs}

\theoremstyle{plain}%
\newtheorem{theorem}{Theorem}
\newtheorem{lemma}[theorem]{Lemma}
\newtheorem{cor}[theorem]{Corollary}
\newtheorem{proposition}[theorem]{Proposition}%

\theoremstyle{definition}%
\newtheorem{example}{Example}%
\newtheorem{remark}{Remark}%

\theoremstyle{definition}%
\newtheorem{definition}{Definition}%

\begin{document}

\title{Equilibrium in the Canonical Stackelberg Triopoly via Response Functions and Fixed-Point Theory}


\author{
Anton Badev\thanks{Supervision and Regulation, Board of Governors, 20th Street and Constitution Avenue NW, Washington, DC 20551, USA. Email: anton.badev@gmail.com}
\and
Martin Pavlov\thanks{Faculty of Mathematics and Informatics, University of Plovdiv Paisii Hilendarski, 24 Tsar Assen Str., Plovdiv 4000, Bulgaria. Email: mpavlov@uni-plovdiv.bg}
\and
Boyan Zlatanov\thanks{Corresponding author. Faculty of Mathematics and Informatics, University of Plovdiv Paisii Hilendarski, 24 Tsar Assen Str., Plovdiv 4000, Bulgaria. Email: bobbyz@uni-plovdiv.bg}
}

\date{}
\maketitle


\begin{abstract}
	{We analyze a canonical extension of the Stackelberg duopoly to a
	sequential framework in which each firm anticipates the stationary
	responses of all subsequent players. For a triopoly, we reformulate
	the hierarchical system of first-order conditions as a mixed tripled
	fixed-point problem. Under suitable contractive assumptions, the
	resulting map has a unique stationary Stackelberg triple, and the
	associated fixed-point iteration converges to it from any admissible
	initial point. A separate derivative-generated map is used to derive
	sufficient second-order conditions for strict local optimality.
	Finally, a recursive formulation is employed to analyze the limiting
	behavior as the number of firms increases.}
\end{abstract}




\section{Introduction}\label{sec1}

{Oligopoly markets are characterized by strategic interdependence:
each firm's decision depends on its expectations concerning the
decisions of its competitors. This interdependence becomes
particularly explicit when firms act sequentially rather than
simultaneously, because the decisions of earlier movers determine the
strategic environment faced by subsequent firms. Such interactions
arise naturally in industries involving capacity commitments,
supply-chain ordering, infrastructure investment, and platform
markets, where production or investment decisions are made in
successive stages.

While the classical Stackelberg model focuses primarily on a
leader--follower duopoly, many markets involve three or more firms
acting sequentially. In such settings, equilibrium analysis requires
a nested backward-induction procedure. The last firm chooses its
quantity conditionally on the decisions of all preceding firms; each
intermediate firm (or middle player) anticipates the stationary responses of all
subsequent firms; and the initial leader incorporates the entire
continuation structure into its decision. This nested dependence may
make the first-order system difficult or impossible to solve
explicitly, particularly under nonlinear demand or cost functions.

The purpose of this paper is to develop a fixed-point formulation of
the first-order conditions for a canonical Stackelberg triopoly that
remains applicable when the stationary continuation responses are
defined only implicitly. We distinguish between two types of maps.
The functions \(b_2\) and \(b_3\) represent the stationary
continuation responses obtained through backward induction, whereas
the functions \(F_1,F_2,F_3\) are derivative-generated fixed-point
maps whose fixed-point equations are equivalent to the corresponding
first-order conditions. This distinction is essential: the
stationary response \(b_3\), rather than an arbitrary value of
\(F_3\), enters the reduced payoff function of the intermediate firm,
and both \(b_2\) and \(b_3\) enter the reduced payoff function of the
leader.

We first prove a one-to-one correspondence between stationary
Stackelberg triples and mixed tripled fixed points associated with the
hierarchical backward-induction map
\(\mathcal G_1\) (Theorem~\ref{thm:characterization}). Contractive
conditions imposed on this map yield existence and uniqueness of a
stationary Stackelberg triple and convergence of the corresponding
fixed-point iteration (Corollary~\ref{cor:tripodal}). We then
introduce a derivative-generated map having the same fixed points.
A separate contractive condition for this map implies the
second-order sufficient conditions for strict local maximization by
all three firms
(Proposition~\ref{prop:second-order-conditions}). Global optimality
requires additional assumptions, such as concavity of the relevant
payoff functions on their strategy sets.

The fixed-point formulation also provides a computational procedure
for approximating the stationary solution when the continuation
responses cannot be expressed in closed form. The resulting Picard
iteration should not be interpreted as best-response dynamic:
the derivative-generated maps \(F_i\) are stationary-point maps, not
the firms' best-response functions. Instead, the iteration is a
numerical fixed-point scheme whose convergence is guaranteed under
the stated contractive assumptions. Its fixed point nevertheless
coincides with the backward-induction stationary solution because the
equations defining the derivative-generated maps and the stationary
responses have the same unique solutions.

The paper is organized as follows. Section~2 introduces the Cournot
and canonical Stackelberg triopoly models and presents an illustrative
example. Section~3 develops the general theory of hierarchical mixed
tripled fixed points and establishes existence, uniqueness,
convergence, and error estimates under Hardy--Rogers-type
conditions. Section~4 applies this theory to the canonical
Stackelberg triopoly. It distinguishes stationary continuation
responses from derivative-generated maps, establishes the fixed-point
characterization of the first-order system, and derives sufficient
conditions for strict local optimality. Section~5 compares Cournot
and hierarchical Stackelberg outcomes under linear demand and linear
or quadratic costs and studies the limiting behavior and convergence
rates as the number of firms increases.

The contribution of the paper should therefore be understood as a
fixed-point analysis of the stationary conditions and their
computational properties, rather than as a new general existence
theory for Stackelberg equilibrium. Existence results based on
backward induction and suitable regularity assumptions are already
well established in the economics literature. Our focus is instead
on the construction of a hierarchical fixed-point operator,
sufficient contractive conditions for uniqueness and convergence,
and the relation between derivative-generated contractions and
second-order local optimality.

The analysis is confined to a canonical sequential quantity game with
complete information and single-valued stationary continuation
responses. The single-valuedness assumptions are imposed to obtain a
well-defined hierarchical fixed-point map. Models with multiple
stationary responses would require a set-valued formulation and are
not considered here. Similarly, the convergence of the proposed
fixed-point iteration concerns a computational method for locating
the stationary solution and should not be interpreted as a claim
about the actual adjustment behavior of firms.}

\subsection{Literature}

{

Oligopoly markets occupy a central position in economic theory, as they
describe industries in which a small number of firms dominate production and
strategic decision-making%
\footnote{See the recent contribution of 
\cite{Adhikari2025,Basiri2025258,BCKS,ByrneDeRoos2024,Dianetti,
Geraskin2024627,IIKYZ,Li2025918,Pepall,Rhee2025,Salahmanesh20241203,
Wu-Yu-Yang,Yan2024}.}.
Unlike perfectly competitive markets, where firms act as price takers, or
monopolistic markets, where a single firm controls
the market outcome, oligopolistic markets are characterized by strategic
interdependence: each firm's decision depends explicitly on its expectations
about competitors' behavior. Such structures arise naturally in industries
such as telecommunications, energy, transportation, and digital platforms,
where limited competition, commitment, and informational asymmetries play
important roles (\citealp{DKRZ,Friedman,IIKYZ,Pepall}).

The Stackelberg model introduces an explicit strategic asymmetry into this
framework. In its original formulation \cite{Stackelberg}, one firm enjoys a
commitment advantage over its rival. This advantage is commonly interpreted
through sequential decision-making: the leader first commits to a price or
quantity, and the follower observes that commitment and responds optimally
(\citealp{Pepall,Tirole}). This timing-based interpretation is especially
prominent in price competition and dynamic oligopoly models
(\citealp{BCKS,Cellini-Lambertini,MS}). Its implications, however, depend on the
firms' strategy spaces and the nature of demand. \citet{BoyerMoreaux1987} show
in a differentiated-product duopoly that firms prefer quantity strategies when
products are substitutes and price strategies when products are complements,
while consumers generally prefer price competition. Enlarging the strategy
space further, \citet{YousefimaneshBosVermeulen2023} allow firms to choose
prices and quantities jointly and show that every equilibrium features
strategic rationing by the leader; the relative sizes and profits of the firms
depend critically on the extent to which unmet demand spills over to the
follower.

A closely related literature considers Stackelberg oligopolies with three or
more firms and compares them with Cournot competition. Under linear costs,
\citet{boyer1986perfect} show that the aggregate output of a hierarchical
Stackelberg market converges to the perfectly competitive outcome.
\citet{Robson1990} establishes existence of pure-strategy equilibrium in a
general sequential $N$-firm quantity game and shows that, with free entry,
Stackelberg equilibria converge to the competitive outcome as each firm's
efficient scale becomes small relative to demand. His analysis covers both
U-shaped average costs and natural monopoly and provides an interpretation of
competition as the limit of a hierarchical market game. \citet{Anderson_Engers_1992}
compare an $m$-firm Cournot model with a hierarchy in which firms choose
outputs sequentially, demonstrating that the first-mover advantage familiar
from duopoly need not extend to every position in a larger hierarchy.
\citet{Economides1993} studies simultaneous entry followed by sequential output
choices and finds that sequential and free-entry Cournot markets have the same
aggregate output and price, although sequential competition supports fewer
firms, greater concentration, and higher surplus. Related work considers
multiple leaders and followers \cite{Askar2018,Ohnishi2021,Sherali1984}.
Within this class, \citet{Tesoriere2017Zero,Tesoriere2017Setup} uses
lattice-theoretic methods to establish equilibrium existence with and without
setup costs and characterize entry accommodation, entry deterrence, and the
effects of changing the numbers of leaders and followers.

Other extensions show that the consequences of sequential commitment depend
on information and on whether timing is imposed or chosen. Under private
information about stochastic demand, \citet{Cumbul2021} finds that signaling
can reverse the conventional perfect-information comparison: sequential
competition may produce lower expected output and welfare than Cournot
competition, and the final mover may earn more than the leader.
\citet{EscrihuelaVillar2009} endogenizes the sequence between a cartel and a
competitive fringe and shows that simultaneous play may be preferred when
collusion is readily sustainable. Beyond oligopoly, \citet{KonradLeininger2007}
study a sequential all-pay auction and show that endogenous timing can lead the
strongest contestant to move late. These contributions demonstrate that
leadership, followership, and even the equilibrium ordering of moves are
endogenous to the strategic and informational environment.

Another strand of the literature develops fixed-point methods for equilibrium
analysis. Beginning with the Banach contraction principle \cite{Banach}, this
theory has been extended through Chatterjea, Hardy--Rogers, and Kannan-type
contractions \cite{Chatterjea,Hardy-Rogers,Kannan}, as well as coupled
\cite{BL,GL}, tripled \cite{Borcut-Berinde}, and $n$-tupled fixed points
\cite{Samet-Vetro}. Coupled fixed points have been applied to Cournot,
Bertrand, and unified Cournot--Bertrand duopoly models \cite{DKRZ}, while
\citet{Petrusel} clarifies the relationship between ordinary and coupled fixed
points. Related methods have been used for duopolies with nondifferentiable
response functions \cite{Kabaivanov} and for triopoly markets \cite{IIKYZ}.
The present paper connects this fixed-point literature with the canonical
hierarchical Stackelberg model. For a triopoly, we formulate the nested
stationary conditions as a hierarchical mixed tripled fixed-point problem,
distinguish continuation responses from derivative-generated stationary-point
maps, and provide sufficient conditions for uniqueness, computational
convergence, and strict local optimality. We then extend the model to $n$ firms
and refine the competitive-limit results associated with
\citet{boyer1986perfect,Robson1990} by deriving explicit convergence rates under
linear and quadratic costs. Unlike Robson's free-entry experiment, which
shrinks firm scale relative to demand, our large-market analysis increases the
number of firms while holding the cost specification fixed; the two approaches
therefore provide related but distinct foundations for competition as the limit
of sequential quantity interaction.
}

{
	The monograph \cite{BCKS} provides a comprehensive analysis of
	nonlinear oligopoly models, including Cournot and generalized
	oligopolies under capacity and nonnegativity constraints, and
	develops a broad range of local and global stability results. Within
	the fixed-point literature, \cite{IIKYZ} reformulates the stationary
	conditions of a Cournot triopoly as a generalized tripled fixed-point
	problem. Under appropriate differentiability and optimality
	assumptions, the fixed points obtained in this way coincide with the
	solutions generated by payoff maximization, while suitable
	contractive conditions additionally guarantee uniqueness and
	convergence of the associated numerical fixed-point iteration. This
	computational convergence should be distinguished from the stability
	of an economic best-response dynamics.
	
	Related equilibrium and computational questions also arise in
	stochastic, differential, and hierarchically organized oligopoly
	models. In a stochastic differential setting, \cite{Dianetti}
	establishes the existence of open-loop Nash equilibria by means of
	forward--backward stochastic differential equations and applies the
	results to capacity expansion in oligopolistic markets. In energy and
	electricity markets, \cite{Wu-Yu-Yang} and \cite{Yan2024} develop
	multi-cluster Stackelberg--Nash models incorporating Cournot
	competition at the lower hierarchical level. These studies establish
	equilibrium results under their respective assumptions and propose
	distributed procedures for equilibrium computation.
}

\section{Preliminaries}

{Consider three firms supplying a homogeneous product to the same
market. Let \(x\), \(y\), and \(z\) denote their respective output
levels, and let
\(Z=x+y+z\)
be the aggregate market output. The market price is determined by the
inverse demand function
\(P(Z)=P(x+y+z)\),
which is assumed to be continuous and, whenever derivatives are used,
sufficiently differentiable. Firm \(i\) has a cost function \(c_i\),
\(i=1,2,3\).

The payoff functions of the firms are
\[
\left\{
\begin{array}{l}
	\Pi_1(x,y,z)=xP(x+y+z)-c_1(x),\\[1mm]
	\Pi_2(x,y,z)=yP(x+y+z)-c_2(y),\\[1mm]
	\Pi_3(x,y,z)=zP(x+y+z)-c_3(z).
\end{array}
\right.
\]}

\subsection{The Cournot Triopoly}\label{sec:2.1}

{In the classical Cournot model, the firms act simultaneously and
non-cooperatively. Each firm chooses its output while treating the
outputs of its competitors as fixed
\cite{Cournot2,Friedman,Smith,Varian}. Thus, the firms solve the
individual optimization problems
\[
\left\{
\begin{array}{l}
	\max\{\Pi_1(x,y,z):x\geq0\},\\
	\max\{\Pi_2(x,y,z):y\geq0\},\\
	\max\{\Pi_3(x,y,z):z\geq0\}.
\end{array}
\right.
\]
For an interior solution, the corresponding first-order necessary
conditions are given by}
\begin{equation}\label{eq-foc-cournot}
	\left\{
	\begin{array}{l}
		\displaystyle
		\frac{\partial \Pi_1}{\partial x}(x,y,z)
		= P\!\left(x+y+z\right)
		+ xP^\prime\!\left(x+y+z\right)
		- c_1^\prime(x)=0,\\[2mm]
		\displaystyle
		\frac{\partial \Pi_2}{\partial y}(x,y,z)
		= P\!\left(x+y+z\right)
		+ yP^\prime\!\left(x+y+z\right)
		- c_2^\prime(y)=0,\\[2mm]
		\displaystyle
		\frac{\partial \Pi_3}{\partial z}(x,y,z)
		= P\!\left(x+y+z\right)
		+ zP^\prime\!\left(x+y+z\right)
		- c_3^\prime(z)=0.
	\end{array}
	\right.
\end{equation}

{A Cournot--Nash equilibrium is a vector of output levels such that
each firm maximizes its payoff given the output levels of the other
firms. Every interior Cournot--Nash equilibrium satisfies
\eqref{eq-foc-cournot}. The converse does not hold without additional
assumptions, because a solution of the first-order system may
represent a minimum or another stationary point rather than a
profit-maximizing choice.

If the strategy sets are bounded, one must also determine whether a
maximum is attained in the interior or on the boundary of the
feasible set \cite{Friedman}. Moreover,
\eqref{eq-foc-cournot} may have more than one solution. The strict
second-order inequalities}
\[
	\left\{
	\begin{array}{l}
		\displaystyle
		\frac{\partial^2 \Pi_1}{\partial x^2}(\xi_1,\xi_2,\xi_3)<0,\\[1mm]
		\displaystyle
		\frac{\partial^2 \Pi_2}{\partial y^2}(\xi_1,\xi_2,\xi_3)<0,\\[1mm]
		\displaystyle
		\frac{\partial^2 \Pi_3}{\partial z^2}(\xi_1,\xi_2,\xi_3)<0.
	\end{array}
	\right.
\]
{guarantee strict local maximization with respect to each firm's own
output. Global optimality follows if, in addition, each payoff
function is concave in the firm's own output on its entire strategy
set.}

\begin{remark}
	{Whenever an implicitly defined function is used, we assume that the
	corresponding hypotheses of the implicit function theorem are
	satisfied. In particular, the relevant functions are assumed to have
	the regularity required for implicit differentiation.}
\end{remark}

\begin{remark}
	{In some special cases, the first-order system
	\eqref{eq-foc-cournot} can be solved explicitly in the form
	\[
	x=b_1(y,z),\qquad
	y=b_2(x,z),\qquad
	z=b_3(x,y).
	\]
	If these functions are obtained by maximizing the firms' payoff
	functions, they are best-response functions. A fixed point of the
	ordered triple \((b_1,b_2,b_3)\) satisfies the first-order system.
	Conversely, if the solutions of the individual first-order equations
	are unique, every solution of \eqref{eq-foc-cournot} is a fixed point
	of \((b_1,b_2,b_3)\).
	
	In general, best-response mappings may be set-valued. A systematic
	study of multivalued mappings and their fixed-point structures can be
	found in \cite{TronThera2026}. In the present paper, we restrict the
	analysis to single-valued implicitly defined stationary responses.
	
	For nonlinear demand or cost functions, explicit formulas for the
	stationary responses may not be available. In such cases, the
	first-order equations can instead be represented by
	derivative-generated fixed-point maps
	\cite{DKRZ,Badev_atal_longrun_mobilemarket_math12050724}. The fixed
	points of these maps coincide with the solutions of the corresponding
	first-order system, provided that the required uniqueness assumptions
	hold. These maps are computational stationary-point maps and should
	not, in general, be identified with the firms' best-response
	functions. Accordingly, their Picard iterations represent numerical
	fixed-point procedures rather than best-response dynamics.%
	\footnote{See
		\cite{Badev_atal_longrun_mobilemarket_math12050724} for an
		application to a duopoly market and \cite{IIKYZ} for an application
		to a triopoly market.}}
\end{remark}

\subsection{The Canonical Stackelberg Triopoly}\label{sec:2.3}

In the canonical Stackelberg model \cite{Stackelberg}, the firms
choose their output levels sequentially rather than simultaneously.
Firm~1, called the leader, acts first. Firm~2, called the middle
player, observes the output of Firm~1 before making its decision,
whereas Firm~3, called the follower, acts after observing the
quantities selected by both preceding firms.

The sequential structure is analyzed by backward induction. Although
Firm~3 acts after Firms~1 and~2, its optimization problem is solved
first, with \(x\) and \(y\) treated as given. This determines the
follower's stationary continuation response \(b_3(x,y)\). The
optimization problem of Firm~2 is then solved by taking into account
the quantity \(x\) selected by Firm~1 and the stationary continuation
response \(b_3(x,y)\) of Firm~3. This determines the middle player's
stationary continuation response \(b_2(x)\). Finally, the optimization
problem of Firm~1 is solved by incorporating the stationary
continuation responses of both subsequent firms. The resulting
asymmetry is therefore generated by the order of decisions, whereas
the optimization problems are solved in the reverse order prescribed
by backward induction.

{
	The appropriate equilibrium concept for this sequential game is
	subgame-perfect equilibrium. A strategy profile is a subgame-perfect
	equilibrium if it induces a Nash equilibrium in every subgame. In the
	present three-stage quantity game, subgame perfection requires Firm~3
	to choose an optimal quantity after every admissible history
	\((x,y)\), Firm~2 to choose an optimal quantity after every admissible
	choice \(x\) of Firm~1 while accounting for Firm~3's continuation
	response, and Firm~1 to choose an optimal quantity while anticipating
	the continuation choices of both subsequent firms. Thus, optimality
	is required not only along the equilibrium path but also after every
	admissible history of the game.
}

Let \(\Pi_i(x,y,z)\), \(i=1,2,3\), be the payoff functions introduced
above, where the index indicates the hierarchical position of the
firm. In the backward-induction procedure, the optimization problem
of Firm~3, called the follower, is solved first, with \(x\) and \(y\)
treated as given. For an interior stationary solution, its
first-order condition is
\begin{equation}\label{eq:11}
\frac{\partial\Pi_3}{\partial z}(x,y,z)
=
P(x+y+z)
+
zP'(x+y+z)
-
c_3'(z)
=
0.
\end{equation}

Assume that equation \eqref{eq:11} has a unique solution with respect to \(z\),
denoted by
\(z=b_3(x,y)\).
At this stage, \(b_3(x,y)\) is a stationary response. It is a
best-response function only when the corresponding maximization
conditions are also satisfied.

Firm~2 (the intermediate player) anticipates the stationary response \(b_3(x,y)\) of Firm~3
and therefore considers the reduced payoff function
\[
\Phi_2(x,y)
=
\Pi_2\bigl(x,y,b_3(x,y)\bigr).
\]

Since \(b_3(x,y)\) depends on \(y\), this dependence must be included
when differentiating the reduced payoff. The first-order condition
for an interior stationary solution of Firm~2 is

\begin{equation}\label{eq:12}
	\begin{aligned}
		\frac{\partial \Phi_2}{\partial y}(x,y)
		&=
		P\!\left(x+y+b_3(x,y)\right)\\
		&\quad+
		yP^\prime\!\left(x+y+b_3(x,y)\right)
		\left(
		1+\frac{\partial b_3}{\partial y}(x,y)
		\right)
		-c_2^\prime(y)
		=0.
	\end{aligned}
\end{equation}

Equation~\eqref{eq:12} is the first-order condition for the middle
player's reduced payoff. Assume that, for every \(x\in X_1\), it has
a unique solution with respect to \(y\), denoted by \(y=b_2(x)\).

Firm~1 anticipates the stationary responses of both subsequent firms
and therefore considers the reduced payoff function
\[
\Phi_1(x)
=
\Pi_1\bigl(
x,b_2(x),b_3(x,b_2(x))
\bigr).
\]

Its first-order condition for an interior stationary solution is
\[
	\frac{d\Phi_1}{dx}(x)
	=
	\frac{d}{dx}
	\Pi_1\bigl(
	x,b_2(x),b_3(x,b_2(x))
	\bigr)
	=
	0.
\]

Consequently, every interior Stackelberg equilibrium satisfies the
following system of first-order conditions:

\begin{equation}\label{eq-foc-stackelberg}
	\left\{
	\begin{array}{rcl}
		\displaystyle
		\frac{\partial\Pi_3}{\partial z}(x,y,z)&=&0,\\[8pt]
		\displaystyle
		\frac{\partial\Phi_2}{\partial y}(x,y)&=&0,\\[8pt]
		\displaystyle
		\frac{d\Phi_1}{dx}(x)&=&0.
	\end{array}
	\right.
\end{equation}

If \((\xi,\eta,\theta)\) is an interior Stackelberg equilibrium, then
it satisfies \eqref{eq-foc-stackelberg} together with the hierarchical
consistency relations
\(\eta=b_2(\xi)\), \(\theta=b_3(\xi,\eta)
=
b_3\bigl(\xi,b_2(\xi)\bigr)\).

Conversely, a triple satisfying these equations is initially only a
stationary Stackelberg triple. Additional second-order or global
concavity conditions are required to conclude that it is an
equilibrium.

The strict second-order inequalities
\[	\left\{
	\begin{array}{rcl}
		\displaystyle
		\frac{\partial^2\Pi_3}{\partial z^2}
		(\xi,\eta,\theta)&<&0,\\[8pt]
		\displaystyle
		\frac{\partial^2\Phi_2}{\partial y^2}
		(\xi,\eta)&<&0,\\[8pt]
		\displaystyle
		\frac{d^2\Phi_1}{dx^2}(\xi)&<&0
	\end{array}
	\right.
\]
are sufficient for strict local maximization at the three successive
stages. To obtain global optimality, one may additionally assume that
\(\Pi_3(x,y,\cdot)\), \(\Phi_2(x,\cdot)\), and \(\Phi_1\) are concave
on their respective strategy sets. If the stated concavity conditions hold after every admissible
history, the stationary continuation responses are optimal in all
corresponding subgames. Consequently, the backward-induction solution
constitutes a subgame-perfect equilibrium.

{\subsection{A Simple Illustrative Example}
\label{sec:example}

The following example illustrates the Cournot and canonical
hierarchical Stackelberg constructions under linear demand and
quadratic costs. We use small parameter values so that all
calculations can be verified directly.

\begin{example}\label{ex-new-1}
	Consider a triopoly market with the inverse demand function
	\begin{equation}\label{eq:example-demand}
		P(Q)=12-Q,
		\qquad
		Q=x+y+z,
	\end{equation}
	and quadratic cost functions
	\(C_i(u)=u^2\) for \(i=1,2,3\).
	
	The payoff functions are therefore
	\[
		\begin{aligned}
			\Pi_1(x,y,z)
			&=
			x(12-x-y-z)-x^2,\\
			\Pi_2(x,y,z)
			&=
			y(12-x-y-z)-y^2,\\
			\Pi_3(x,y,z)
			&=
			z(12-x-y-z)-z^2.
		\end{aligned}
	\]
\end{example}

\subsubsection{Cournot Equilibrium}

In the Cournot model, each firm maximizes its payoff while treating
the outputs of the other two firms as fixed. The first-order
conditions are
\[
	\left\{
	\begin{array}{rcl}
		12-4x-y-z&=&0,\\
		12-x-4y-z&=&0,\\
		12-x-y-4z&=&0.
	\end{array}
	\right.
\]

By symmetry, the unique solution is
\(x_C=y_C=z_C=2\) (Figure \eqref{fig:cournot-iteration-firms}).
Hence the aggregate Cournot output and market price are
\(Q_C=6\) and \(P_C=6\).

The second derivatives with respect to the firms' own quantities are
\[
\frac{\partial^2\Pi_1}{\partial x^2}
=
\frac{\partial^2\Pi_2}{\partial y^2}
=
\frac{\partial^2\Pi_3}{\partial z^2}
=
-4<0.
\]
Since the payoff functions are strictly concave in the corresponding
decision variables, \((2,2,2)\) is the unique Cournot equilibrium.

The equilibrium profit of each firm is
\(\Pi_1^C=\Pi_2^C=\Pi_3^C
	=
	2\cdot6-2^2
	=
	8\).
Thus, aggregate producer profit is
\(\Pi_{\mathrm{tot}}^C=24\).

\begin{figure}[!htbp]
	\centering
	\includegraphics[width=0.82\textwidth]{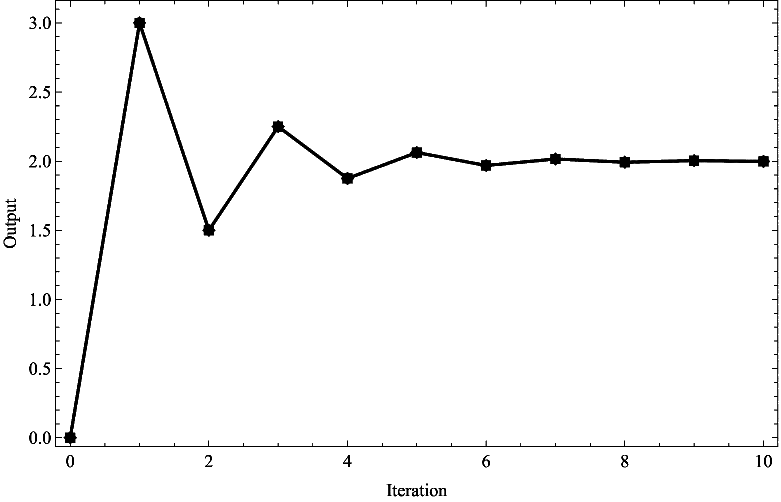}
	\caption{Firm-level output trajectories under the Cournot best-response iteration. Green: Firms~1,~2, and 3. The iteration converges to the symmetric Cournot equilibrium \(x_C=y_C=z_C=2\).}
	\label{fig:cournot-iteration-firms}
\end{figure}

\subsubsection{Stackelberg Equilibrium}

We now consider the canonical hierarchical Stackelberg model in
which Firm~1 is the initial leader, Firm~2 is the middle player,
and Firm~3 is the final follower. All firms are restricted to
nonnegative output levels.

For fixed \(x\geq0\) and \(y\geq0\), Firm~3 solves
\(\max_{z\geq0}\Pi_3(x,y,z)\)
The first-order condition for an interior solution is
\(\frac{\partial\Pi_3}{\partial z}(x,y,z)
=
12-x-y-4z
=
0\).

Since
\(\frac{\partial^2\Pi_3}{\partial z^2}
=
-4<0\),
the payoff of Firm~3 is strictly concave in \(z\). Consequently, its
unique continuation best response is
\begin{equation}\label{eq:example-b3}
	b_3(x,y)
	=
	\max\left\{
	0,
	\frac{12-x-y}{4}
	\right\}.
\end{equation}

Thus, the response of Firm~3 is interior whenever \(x+y<12\) and is
equal to zero whenever \(x+y\geq12\).

Firm~2 anticipates the continuation response
\eqref{eq:example-b3}. Suppose first that \(0\leq x<12\) and consider
the region \(x+y<12\), in which the response of Firm~3 is interior.
On this region, the reduced payoff function of Firm~2 is
\[
	\Phi_2(x,y)=
	\Pi_2\bigl(x,y,b_3(x,y)\bigr)=
	9y-\frac{3}{4}xy-\frac{7}{4}y^2.
\]

The corresponding first-order condition is
\(\frac{\partial\Phi_2}{\partial y}(x,y)
=
9-\frac{3}{4}x-\frac{7}{2}y
=
0\),
and hence the interior stationary response is
\(y
=
\frac{18}{7}-\frac{3}{14}x\).

Moreover,
\(\frac{\partial^2\Phi_2}{\partial y^2}
=
-\frac{7}{2}<0\),
so this stationary response is the unique maximizer of the reduced
payoff on the interior region.

For every \(0\leq x<12\), the value of \(y
=
\frac{18}{7}-\frac{3}{14}x\) is positive and satisfies
\[
\frac{12-x-y}{4}=\frac{12-x-(\frac{18}{7}-\frac{3}{14}x)}{4}>0.
\]

It therefore belongs to the region in which the interior expression
for \(b_3\) is valid. If \(x\geq12\), both continuation firms optimally
choose zero output. Consequently, the unique continuation best
response of Firm~2 can be written as
\begin{equation}\label{eq:example-b2}
	b_2(x)
	=
	\max\left\{
	0,
	\frac{18}{7}-\frac{3}{14}x
	\right\}.
\end{equation}

For \(0\leq x<12\), substituting
\(y=b_2(x)
=
\frac{18}{7}-\frac{3}{14}x\)
into the continuation response of Firm~3 gives
\[
b_3\bigl(x,b_2(x)\bigr)=
	\frac{12-x-b_2(x)}{4}=
	\frac{33}{14}-\frac{11}{56}x.
\]

Therefore, on the relevant interior region \(0\leq x<12\), the
reduced payoff function of Firm~1 is
\[
	\Phi_1(x)=
	\Pi_1\bigl(
	x,b_2(x),b_3(x,b_2(x))
	\bigr)=
	\frac{99}{14}x-\frac{89}{56}x^2.
\]

Its first-order condition is
\(\frac{d\Phi_1}{dx}(x)
=
\frac{99}{14}-\frac{89}{28}x
=
0\).
Hence,
\(x_S
=
\frac{198}{89}
\approx
2.225\).

Furthermore,
\(\frac{d^2\Phi_1}{dx^2}
=
-\frac{89}{28}<0\),
so \(\Phi_1\) is strictly concave on \(0\leq x<12\), and \(x_S\) is
its unique maximizer on this region.

For \(x\geq12\), the continuation responses satisfy
\(b_2(x)=0\), \(b_3\bigl(x,b_2(x)\bigr)=0\),
and the leader's reduced payoff becomes
\[
\Phi_1(x)
=
x(12-x)-x^2
=
12x-2x^2=2x(6-x)<0
\]
for \(x\geq 12\). At the same time, \(\Phi_1(x_S)\approx 7.866>0\). 

Therefore, no value \(x\geq 12\) can be optimal, because it yields a negative profit, whereas 
\(x_S\) yields a positive profit.

Applying backward induction gives (Figure \ref{fig:stackelberg-iteration-firms})
\begin{equation}\label{eq:example-stackelberg-quantities}
	\begin{aligned}
		x_S
		&=
		\frac{198}{89}
		\approx2.225,\\
		y_S
		&=
		b_2(x_S)
		=
		\frac{1305}{623}
		\approx2.095,\\
		z_S
		&=
		b_3(x_S,y_S)
		=
		\frac{4785}{2492}
		\approx1.920.
	\end{aligned}
\end{equation}

The aggregate Stackelberg output is
\(Q_S
	=
	x_S+y_S+z_S
	=
	\frac{15549}{2492}
	\approx
	6.240\),
and the corresponding market price is
\(P_S
	=
	12-Q_S
	=
	\frac{14355}{2492}
	\approx
	5.760\).

Since the relevant payoff functions are strictly concave at all three
stages, the stationary triple in
\eqref{eq:example-stackelberg-quantities} is the unique Stackelberg
equilibrium for this example.

The equilibrium profits are
\[
	\begin{aligned}
		\Pi_1^S
		&=
		\frac{9801}{1246}
		\approx7.866,\\
		\Pi_2^S
		&=
		\frac{1703025}{221788}
		\approx7.679,\\
		\Pi_3^S
		&=
		\frac{22896225}{3105032}
		\approx7.374.
	\end{aligned}
\]

Thus, aggregate producer profit is
\(\Pi_{\mathrm{tot}}^S
	=
	\frac{71162667}{3105032}
	\approx
	22.918\).
	
\begin{figure}[!htbp]
	\centering
	\includegraphics[width=0.82\textwidth]{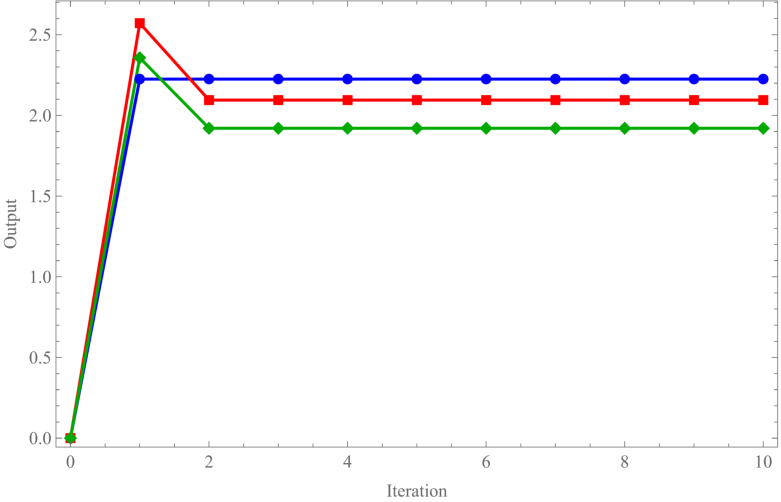}
	\caption{Firm-level outputs generated by the hierarchical Stackelberg backward-induction updates. At each step, the continuation responses \(b_2\) and \(b_3\) are evaluated successively. Blue: Firm~1. Red: Firm~2. Green: Firm~3.}
	\label{fig:stackelberg-iteration-firms}
\end{figure}

\subsubsection{Consumer and Total Surplus}

For the inverse demand function \eqref{eq:example-demand}, consumer
surplus at aggregate output \(Q\) is
\[
	CS(Q)
	=
	\int_0^Q(12-t)\,dt-(12-Q)Q
	=
	\frac{Q^2}{2}.
\]

At the Cournot equilibrium,
\(CS_C
	=
	\frac{6^2}{2}
	=
	18\).
At the Stackelberg equilibrium,
\(CS_S
	=
	\frac12
	\left(
	\frac{15549}{2492}
	\right)^2
	=
	\frac{241771401}{12420128}
	\approx
	19.466\).

Total surplus is the sum of consumer surplus and aggregate producer
profit. Therefore,
\(TS_C
	=
	CS_C+\Pi_{\mathrm{tot}}^C
	=
	18+24
	=
	42\),
whereas
\(TS_S=
		CS_S+\Pi_{\mathrm{tot}}^S=
		\frac{526422069}{12420128}\approx
		42.385\).

Thus, in this example, hierarchical Stackelberg competition produces
a larger aggregate quantity and a lower market price than Cournot
competition. Consumer surplus increases, aggregate producer profit
decreases, and total surplus is slightly higher. The principal equilibrium outcomes under Cournot and hierarchical Stackelberg competition are summarized in
Table~\ref{tab:simple-example-comparison}.

\begin{table}[ht]
	\centering
	\caption{Cournot and Stackelberg equilibrium outcomes for
		\(A=12\), \(B=1\), and \(c=1\).}
	\label{tab:simple-example-comparison}
	\begin{tabular}{l c c}
		\toprule
		Outcome & Cournot & Stackelberg\\
		\midrule
		Firm~1 output & \(2.000\) & \(2.225\)\\
		Firm~2 output & \(2.000\) & \(2.095\)\\
		Firm~3 output & \(2.000\) & \(1.920\)\\
		Aggregate output & \(6.000\) & \(6.240\)\\
		Market price & \(6.000\) & \(5.760\)\\
		Aggregate producer profit & \(24.000\) & \(22.918\)\\
		Consumer surplus & \(18.000\) & \(19.466\)\\
		Total surplus & \(42.000\) & \(42.385\)\\
		\bottomrule
	\end{tabular}
\end{table}

The example also illustrates the asymmetric allocation generated by
sequential play. Relative to the Cournot quantity, the initial leader
and the intermediate firm produce more, whereas the final follower
produces less. These numerical conclusions concern the present
parameter values and are not used as substitutes for the general
comparative results established in Section~
\ref{sec:large-market-comparison}.}

\section{Main Result}

It may be difficult or even impossible to solve the first-order conditions explicitly, both in the Cournot model and in the Stackelberg framework. For instance, if the inverse demand function is given by
$P = 10 - \arctan(x+y+z)$ and the cost functions take the form
$C(t) = t e^{t}$, then explicit closed-form expressions for the response functions generally cannot be obtained. 

In such cases, we can adopt the static approach to
Cournot used by in \cite{Badev_atal_longrun_mobilemarket_math12050724,DKRZ}, in
which the structure of the first-order conditions is used to recast
the optimality conditions as a fixed-point problem.

\subsection{Mixed Tripled Fixed Points}\label{sec:4.1}
{
\begin{definition}\label{def:shorthand-41}
	Let \(X_1\), \(X_2\), and \(X_3\) be nonempty sets, and set
	\(X=X_1\times X_2\times X_3\).
	Let
	\(\mathcal F_i:X\to X_i\) for \(i=1,2,3\),
	be given mappings.
	
	Define the hierarchical map \(S:X\to X\) by
	\[S(x,y,z)
	=
	\bigl(
	S^1(x,y,z),
	S^2(x,y,z),
	S^3(x,y,z)
	\bigr),\]
	where
	\[
	\begin{aligned}
		S^3(x,y,z)
		&=
		\mathcal F_3(x,y,z),\\
		S^2(x,y,z)
		&=
		\mathcal F_2
		\bigl(
		x,y,S^3(x,y,z)
		\bigr),\\
		S^1(x,y,z)
		&=
		\mathcal F_1
		\bigl(
		x,
		S^2(x,y,z),
		S^3(x,y,z)
		\bigr).
	\end{aligned}
	\]
	Equivalently,
	\[
		S(x,y,z)
		=
		\Bigl(
		\mathcal F_1\bigl(
		x,
		\mathcal F_2(x,y,\mathcal F_3(x,y,z)),
		\mathcal F_3(x,y,z)
		\bigr),\mathcal F_2(x,y,\mathcal F_3(x,y,z)),
		\mathcal F_3(x,y,z)
		\Bigr).
	\]
	
	The mappings \(S^i:X\to X_i\), \(i=1,2,3\), are called the
	coordinate maps of \(S\).
\end{definition}

\begin{definition}\label{def:distance-41}
	Let \((X_i,d_i)\), \(i=1,2,3\), be metric spaces, set
	\(X=X_1\times X_2\times X_3\),
	and let \(S:X\to X\) be the hierarchical map from
	Definition~\ref{def:shorthand-41}.
	
	For
	\[
	\overline{x}=(x,y,z),
	\qquad
	\overline{u}=(u,v,w)
	\]
	in \(X\), define
	\[
	M_1(\overline{x},\overline{u})=
	d_1(x,u)+d_2(y,v)+d_3(z,w),
	\]
	\[
		M_2^S(\overline{x},\overline{u})=
		M_1(\overline{x},S(\overline{x}))
		+
		M_1(\overline{u},S(\overline{u})),
	\]
	and
	\[
		M_3^S(\overline{x},\overline{u})=
		M_1(\overline{x},S(\overline{u}))
		+
		M_1(\overline{u},S(\overline{x})).
	\]
\end{definition}

\begin{definition}\label{def:mixed-tripled-fp}
	Let \(X_1,X_2,X_3\) be nonempty sets, and let
	\(S:X_1\times X_2\times X_3
	\to X_1\times X_2\times X_3\)
	be the hierarchical map generated by
	\((\mathcal F_1,\mathcal F_2,\mathcal F_3)\) according to
	Definition~\ref{def:shorthand-41}.
	
	A point
	\(\overline{x}^{\,*}
	=
	(x^*,y^*,z^*)
	\in X_1\times X_2\times X_3\)
	is called a mixed tripled fixed point of
	\((\mathcal F_1,\mathcal F_2,\mathcal F_3)\) with respect to the
	hierarchical architecture if
	\(S(\overline{x}^{\,*})=\overline{x}^{\,*}\).
	
	Equivalently,
	\(S^1(x^*,y^*,z^*)=x^*\), \(S^2(x^*,y^*,z^*)=y^*\), and \(S^3(x^*,y^*,z^*)=z^*\).
\end{definition}

\begin{definition}\label{def:local-cross-symmetry}
	Let \(X_1=X_2=X_3=X\), and let
	\(S:X^3\to X^3\), \(S=(S^1,S^2,S^3)\),
	be the hierarchical map from
	Definition~\ref{def:shorthand-41}.
	
	We say that \(S\) satisfies the local triple cross-symmetry
	property at
	\((x^*,y^*,z^*)\in X^3\) if
	\[
	\begin{aligned}
		S^1(y^*,z^*,x^*)
		&=S^2(x^*,y^*,z^*),\\
		S^2(y^*,z^*,x^*)
		&=S^3(x^*,y^*,z^*),\\
		S^3(y^*,z^*,x^*)
		&=S^1(x^*,y^*,z^*).
	\end{aligned}
	\]
\end{definition}
}
\begin{theorem}\label{thm:unified-41}
	Let $(X_i,d_i)$, $i=1,2,3$, be complete metric spaces, and let
	${\cal{F}}_i:X_1\times X_2\times X_3 \to X_i$ be mappings.
	Let $S:X_1\times X_2\times X_3 \to X_1\times X_2\times X_3$ be the hierarchical map from Definition~\ref{def:shorthand-41}.
	Assume that there exist non-negative constants $k_1,k_2,k_3$ such that
	$k_1 + 2k_2 + 2k_3 < 1$, and for all $(x,y,z),(u,v,w)\in X_1\times X_2\times X_3$ there holds
	\begin{equation}\label{eq:HR-unified-41}
		M_1\left(S(\overline{x}),S(\overline{u})\right)\leq
		k_1\,M_1\big((\overline{x}),(\overline{u})\big)
		+ k_2\,M_2^{S}\big((\overline{x}),(\overline{u})\big)
		+ k_3\,M_3^{S}\big((\overline{x}),(\overline{u})\big),
	\end{equation}
	where $\overline{x}=(x,y,z)$ and $\overline{u}=(u,v,w)$.
	
	Then:
	\begin{enumerate}
		\item There exists a unique $\overline{x}^{\,*}=(x^*,y^*,z^*)\in X_1\times X_2\times X_3$
		such that $\overline{x}^{\,*}=S(\overline{x}^{\,*})$.
		In particular, $\overline{x}^{\,*}$ is the unique mixed tripled fixed point of $({\cal{F}}_1,{\cal{F}}_2,{\cal{F}}_3)$
		in the hierarchical architecture.
		
		\item For any initial point $\overline{x}_0=(x_0,y_0,z_0)\in X_1\times X_2\times X_3$,
		the Picard iteration $\overline{x}_{n+1}:=S(\overline{x}_n)$ converges to $\overline{x}^{\,*}$.
		
		\item Let $k=\frac{k_1 + k_2 + k_3}{1 - k_2 - k_3}$. Then, for all $n\ge 1$:
		\begin{itemize}
			\item A priori estimate:
			$M_1(\overline{x_n},\overline{x^*})
			\le \frac{k^n}{1-k}\,M_1(\overline{x_1},\overline{x_0})$
			\item A posteriori estimate:
			$M_1(\overline{x_n},\overline{x^*})
			\le \frac{k}{1-k}\,M_1(\overline{x_n},\overline{x_{n-1}})$
			\item Rate of convergence:
			$M_1(\overline{x_n},\overline{x^*})
			\le k\,M_1(\overline{x_{n-1}},\overline{x}^{\,*})$.
		\end{itemize}
	\end{enumerate}
	
	If in addition $X_1=X_2=X_3=X$, $d_1=d_2=d_3=d$, and $S$ satisfies the local triple cross symmetry property at $(x^*,y^*,z^*)$, then $x^*=y^*=z^*$.
\end{theorem}

\begin{proof}
	Let $X := X_1 \times X_2 \times X_3$ and define the metric $M_1\big((x,y,z),(u,v,w)\big)$
	for all $(x,y,z),(u,v,w) \in X$. Then $(X,M_1)$ is complete since each $(X_i,d_i)$ is complete.
	
	For notational convenience, let $\overline{x}=(x,y,z)$ and $\overline{u}=(u,v,w)$.
	
	For the admissible map $S$ selected in the theorem,
	the inequality \eqref{eq:HR-unified-41} reads
	\[
	M_1\big(S(\overline{x}), S(\overline{u})\big)
	\le k_1M_1\big(\overline{x},\overline{u}\big)\\
	+ k_2M_2^{S}\big(\overline{x},\overline{u}\big)
	+k_3M_3^{S}\big((\overline{x}),\overline{u}\big).
	\]
	
	Hence, $S$ is a Hardy--Rogers map on $(X,M_1)$ with constants
	$\left(k_1,k_2,k_3\right)$.
	
	Since $k_1 + 2k_2 + 2k_3 < 1$,
	Theorem~\ref{th:1} ensures
	the existence and uniqueness of a fixed point $(x^*,y^*,z^*) \in X$,
	together with convergence of the Picard iteration
	\[
	(x_{n+1},y_{n+1},z_{n+1})=S(x_n,y_n,z_n)
	\]
	for any initial point $(x_0,y_0,z_0) \in X$ and $n=0,1,2,\dots$.
	
	Moreover, the stated error estimates follow directly with
	$k = \frac{k_1 + k_2 + k_3}{1 - k_2 - k_3}$.
	
	Assume $X_1=X_2=X_3=X$ and $d_1=d_2=d_3=d$, and let $(x^*,y^*,z^*)$ be the unique mixed tripled fixed point for the admissible map $S$, i.e., $S(x^*,y^*,z^*)=(x^*,y^*,z^*)$.
	
	Assume that the local triple cross-symmetry property
	of $S$ holds at $(x^*,y^*,z^*)$, i.e.,
	\[
	\begin{aligned}
		S^1(y^*,z^*,x^*) &= S^2(x^*,y^*,z^*),\\
		S^2(y^*,z^*,x^*) &= S^3(x^*,y^*,z^*),\\
		S^3(y^*,z^*,x^*) &= S^1(x^*,y^*,z^*).
	\end{aligned}
	\]

	We will apply the Hardy--Rogers inequality \eqref{eq:HR-unified-41} to the triples
	$(x,y,z)=(x^*,y^*,z^*)=\overline{x^*}$ and $(u,v,w)=(y^*,z^*,x^*)=\overline{y^*}$.
	
	Denote the left-hand side by
	$S_0=\sum_{j=1}^3d\big(S^j(\overline{x^*}),S^j(\overline{y^*})\big)$.
	
	Since $(x^*,y^*,z^*)$ is a tripled fixed point of $S$, i.e., $S^1(x^*,y^*,z^*)=x^*$, $S^2(x^*,y^*,z^*)=y^*$, $S^3(x^*,y^*,z^*)=z^*$, 
	and the admissible map \(S\) satisfies the local
	triple cross-symmetry property at the mixed tripled fixed point $(x^*,y^*,z^*)$, i.e.,
	\begin{align*}
		S^1(y^*,z^*,x^*)
		&=S^2(x^*,y^*,z^*)=y^*,\\
		S^2(y^*,z^*,x^*)
		&=S^3(x^*,y^*,z^*)=z^*,\\
		S^3(y^*,z^*,x^*)
		&=S^1(x^*,y^*,z^*)=x^*,
	\end{align*}
	we obtain
	\[
	\begin{array}{lll}
		S_0&=&\displaystyle\sum_{j=1}^3d\big(S^j(\overline{x^*}),S^j(\overline{y^*})\big)\\[8pt]
		&=&d(x^*,y^*)+d(y^*,z^*)+d(z^*,x^*)=M_1\big((x^*,y^*,z^*),(y^*,z^*,x^*)\big).
	\end{array}
	\]
	
	For the right-hand side, the terms in \eqref{eq:HR-unified-41} evaluate as follows:
	\[
	M_2^{S}\big(\overline{x^*},\overline{y^*}\big)
	=M_1\big(\overline{x^*},S(\overline{x^*})\big)
	+ M_1\big(\overline{y^*},S(\overline{y^*})\big).
	\]
	
	The first addend $M_1\big(\overline{x^*},S(\overline{x^*})\big)$ is zero since $(x^*,y^*,z^*)$ is a tripled fixed point of $S$.
	
	Using the local triple cross-symmetry property, we obtain
	\[
	S(y^*,z^*,x^*)=\big(S^1(y^*,z^*,x^*),S^2(y^*,z^*,x^*),S^3(y^*,z^*,x^*)\big)=(y^*,z^*,x^*),
	\]
	so the second addend $M_1\big(\overline{y^*},S(\overline{y^*})\big)$ is zero, too. Thus
	\[
	M_2^{S}\big((x^*,y^*,z^*),(y^*,z^*,x^*)\big)=0.
	\]
	
	Similarly,
	\[
	\begin{array}{lll}
		M_3^{S}\big(\overline{x^*},\overline{y^*}\big)
		&=& M_1\big(\overline{x^*},S(\overline{y^*})\big)
		+ M_1\big(\overline{y^*},S(\overline{x^*})\big)\\[8pt]
		&=& M_1\big(\overline{x^*},\overline{y^*}\big)
		+ M_1\big(\overline{y^*},\overline{x^*}\big)= 2M_1\big(\overline{x^*},\overline{y^*}\big).
	\end{array}
	\]
	
	Therefore, \eqref{eq:HR-unified-41} gives
	\[
	\begin{array}{lll}
		S_0&=&M_1\big(\overline{x^*},\overline{y^*}\big)\leq k_1M_1\big(\overline{x^*},\overline{y^*}\big)+2k_3M_1\big(\overline{x^*},\overline{y^*}\big)\\[8pt]
		&=& \big(k_1+2k_3\big)M_1\big(\overline{x^*},\overline{y^*}\big).
	\end{array}
	\]
	
	Since $k_1+2k_3<1$ (it follows from $k_1+2k_2+2k_3<1$), we conclude that
	\[
	M_1\big((x^*,y^*,z^*),(y^*,z^*,x^*)\big)=0,
	\]
	i.e., $d(x^*,y^*)=d(y^*,z^*)=d(z^*,x^*)=0$, which yields $x^*=y^*=z^*$.
\end{proof}

\section{Market Equilibrium in the Stackelberg Model}


{
	\subsection{Set-up and a Characterization}
	\label{subsec:characterization}
	
	We reformulate the first-order conditions for the canonical
	Stackelberg triopoly as a hierarchical fixed-point problem. The
	construction follows the order in which the firms make their
	decisions. Firm~3 chooses its quantity after observing the quantities
	of Firms~1 and~2; Firm~2 anticipates the stationary response of
	Firm~3; and Firm~1 anticipates the stationary responses of both
	subsequent firms.
	
	For Firm~3, define the derivative-generated map
	\begin{equation}\label{eq:def-F3}
		F_3(x,y,z)
		=
		z+\frac{\partial\Pi_3}{\partial z}(x,y,z).
	\end{equation}
	Then
	\begin{equation}\label{eq:F3-equivalence}
		F_3(x,y,z)=z
		\quad\Longleftrightarrow\quad
		\frac{\partial\Pi_3}{\partial z}(x,y,z)=0.
	\end{equation}
	
	Assume that, for every \((x,y)\in X_1\times X_2\), the fixed-point
	equation
	\begin{equation}\label{eq:F3-fixed-point}
		F_3(x,y,z)=z
	\end{equation}
	has a unique solution, denoted by
	\(z=b_3(x,y)\).
	The stationary response \(b_3:X_1\times X_2\to X_3\) need not be
	available explicitly; it may be defined implicitly by
	\eqref{eq:F3-fixed-point}. By its definition,
	\begin{equation}\label{eq:b3-identity}
		F_3\bigl(x,y,b_3(x,y)\bigr)=b_3(x,y),
	\end{equation}
	and hence
	\[\frac{\partial\Pi_3}{\partial z}
		\bigl(x,y,b_3(x,y)\bigr)=0.
	\]
	
	It is important to distinguish the stationary response \(b_3(x,y)\)
	from the derivative-generated map \(F_3(x,y,z)\). The former, rather
	than the latter, enters the reduced payoff function of Firm~2:
	\[\Phi_2(x,y)
		=
		\Pi_2\bigl(x,y,b_3(x,y)\bigr).
	\]
	
	Assuming the required differentiability, define
	\begin{equation}\label{eq:def-F2}
		F_2(x,y,z)
		=
		y+\frac{\partial\Phi_2}{\partial y}(x,y).
	\end{equation}
	The third argument is retained so that \(F_2\) has the same domain as
	the maps in the general hierarchical fixed-point framework. Its value
	in the present Stackelberg construction is independent of that
	argument. Therefore,
	\begin{equation}\label{eq:F2-equivalence}
		F_2\bigl(x,y,b_3(x,y)\bigr)=y
		\quad\Longleftrightarrow\quad
		\frac{\partial\Phi_2}{\partial y}(x,y)=0.
	\end{equation}
	
	Assume that, for every \(x\in X_1\), the fixed-point equation
	\begin{equation}\label{eq:F2-fixed-point}
		F_2\bigl(x,y,b_3(x,y)\bigr)=y
	\end{equation}
	has a unique solution, denoted by
	\(y=b_2(x)\).
	Consequently,
	\begin{equation}\label{eq:b2-identity}
		F_2\bigl(
		x,b_2(x),b_3(x,b_2(x))
		\bigr)
		=
		b_2(x),
	\end{equation}
	and hence
	\(\frac{\partial\Phi_2}{\partial y}
		\bigl(x,b_2(x)\bigr)=0\).
	
	Firm~1 anticipates the stationary responses of both subsequent firms.
	Its reduced payoff function is
	\[
		\Phi_1(x)
		=
		\Pi_1\bigl(
		x,b_2(x),b_3(x,b_2(x))
		\bigr).
	\]
	
	Define
	\begin{equation}\label{eq:def-F1}
		F_1(x,y,z)
		=
		x+\frac{d\Phi_1}{dx}(x).
	\end{equation}
	
	The second and third arguments are retained so that \(F_1\) has the
	same domain as the other maps. Since the continuation responses have
	already been incorporated into \(\Phi_1\), the equation
	\[
			F_1\bigl(
		x,b_2(x),b_3(x,b_2(x))
		\bigr)=x
	\]
	is equivalent to
	\begin{equation}\label{eq:F1-equivalence}
		\frac{d\Phi_1}{dx}(x)=0.
	\end{equation}
	
	Thus, the derivative-generated maps \(F_1,F_2,F_3\) determine the
	stationary responses recursively:
	\[
		\begin{aligned}
			F_3(x,y,z)=z
			&\quad\Longrightarrow\quad z=b_3(x,y),\\
			F_2\bigl(x,y,b_3(x,y)\bigr)=y
			&\quad\Longrightarrow\quad y=b_2(x),\\
			F_1\bigl(x,b_2(x),b_3(x,b_2(x))\bigr)=x.
		\end{aligned}
	\]
	
	To place this recursive construction within the general framework of
	Definition~\ref{def:shorthand-41}, define
	\[
	{\cal{F}}_i:X_1\times X_2\times X_3\to X_i,
	\qquad i=1,2,3,
	\]
	by
	\begin{equation}\label{eq:hat-functions}
		\begin{aligned}
			{\cal{F}}_1(x,y,z)
			&=F_1(x,y,z),\\
			{\cal{F}}_2(x,y,z)
			&=b_2(x),\\
			{\cal{F}}_3(x,y,z)
			&=b_3\bigl(x,b_2(x)\bigr).
		\end{aligned}
	\end{equation}
	
	Let \(\mathcal G_1:X_1\times X_2\times X_3
	\to X_1\times X_2\times X_3\) be the hierarchical map from
	Definition~\ref{def:shorthand-41} generated by
	\(({\cal{F}}_1,{\cal{F}}_2,{\cal{F}}_3)\). Thus,
	\[
		\begin{aligned}
			\mathcal G_1(x,y,z)
			=&
			\Bigl(
			{\cal{F}}_1\bigl(
			x,
			{\cal{F}}_2(x,y,{\cal{F}}_3(x,y,z)),
			{\cal{F}}_3(x,y,z)
			\bigr),\\
			&\qquad
			{\cal{F}}_2(x,y,{\cal{F}}_3(x,y,z)),
			{\cal{F}}_3(x,y,z)
			\Bigr).
		\end{aligned}
	\]
	
	Using \eqref{eq:hat-functions}, we obtain
	\begin{equation}\label{eq:G-trip}
		\begin{aligned}
			\mathcal G_1(x,y,z)
			=
			\Bigl(
			&{\cal{F}}_1\bigl(
			x,b_2(x),b_3(x,b_2(x))
			\bigr),\\
			&b_2(x),\,
			b_3(x,b_2(x))
			\Bigr).
		\end{aligned}
	\end{equation}
	
	The map \(\mathcal G_1\) depends only on the quantity chosen by
	Firm~1. This reflects the backward-induction structure of the
	canonical Stackelberg model: once \(x\) is fixed, the quantities of
	Firms~2 and~3 are determined successively by \(b_2(x)\) and
	\(b_3(x,b_2(x))\).
	
	A triple
	\((x^*,y^*,z^*)\in X_1\times X_2\times X_3\) is called a stationary
	Stackelberg triple if
	\begin{equation}\label{eq:stationary-stackelberg-triple}
		\frac{d\Phi_1}{dx}(x^*)=0,\qquad
		\frac{\partial\Phi_2}{\partial y}(x^*,y^*)=0,\qquad
		\frac{\partial\Pi_3}{\partial z}(x^*,y^*,z^*)=0,
	\end{equation}
	and
	\begin{equation}\label{eq:hierarchical-consistency}
		y^*=b_2(x^*),
		\qquad
		z^*=b_3(x^*,y^*)
		=
		b_3\bigl(x^*,b_2(x^*)\bigr).
	\end{equation}

	\begin{theorem}\label{thm:characterization}
		Assume that, for every \((x,y)\in X_1\times X_2\), the equation
		\(F_3(x,y,z)=z\)
		has a unique solution \(z=b_3(x,y)\), and that, for every
		\(x\in X_1\), the equation
		\(F_2\bigl(x,y,b_3(x,y)\bigr)=y\)
		has a unique solution \(y=b_2(x)\).
		
		Let \(({\cal{F}}_1,{\cal{F}}_2,{\cal{F}}_3)\) be defined by
		\eqref{eq:hat-functions}, and let \(\mathcal G_1\) be the
		corresponding hierarchical map defined by \eqref{eq:G-trip}.
		
		Then a triple
		\((x^*,y^*,z^*)\in X_1\times X_2\times X_3\)
		is a stationary Stackelberg triple if and only if it is a mixed
		tripled fixed point of
		\(({\cal{F}}_1,{\cal{F}}_2,{\cal{F}}_3)\) with respect to the
		hierarchical architecture, i.e., \(\mathcal G_1(x^*,y^*,z^*)
		=
		(x^*,y^*,z^*)\).
	\end{theorem}
	
	\begin{proof}
		Suppose first that \((x^*,y^*,z^*)\) is a stationary Stackelberg
		triple. From \eqref{eq:hierarchical-consistency}, we have
		\[
		y^*=b_2(x^*)
		\qquad\text{and}\qquad
		z^*
		=
		b_3(x^*,y^*)
		=
		b_3\bigl(x^*,b_2(x^*)\bigr).
		\]
		
		Moreover,
		\(\frac{d\Phi_1}{dx}(x^*)=0\). 
		By the definition of \(F_1\), this is equivalent to
		\[
		F_1\bigl(
		x^*,b_2(x^*),b_3(x^*,b_2(x^*))
		\bigr)
		=
		x^*.
		\]
		
		Therefore, \eqref{eq:G-trip} gives
		\(\mathcal G_1(x^*,y^*,z^*)
		=
		(x^*,y^*,z^*)\).
		
		Conversely, suppose that
		\(\mathcal G_1(x^*,y^*,z^*)
		=
		(x^*,y^*,z^*)\).
		By \eqref{eq:G-trip}, there holds
		\[
		F_1\bigl(
		x^*,b_2(x^*),b_3(x^*,b_2(x^*))
		\bigr)
		=
		x^*,\qquad
		y^*=b_2(x^*),\qquad
		\text{and}\qquad
		z^*=b_3\bigl(x^*,b_2(x^*)\bigr).
		\]
		
		The first equality and the definition of \(F_1\) imply
		\(\frac{d\Phi_1}{dx}(x^*)=0\).
		
		Since \(y^*=b_2(x^*)\), identity \eqref{eq:b2-identity} gives
		\(F_2(x^*,y^*,z^*)=y^*\),
		and hence
		\(\frac{\partial\Phi_2}{\partial y}(x^*,y^*)=0\).
		
		Furthermore,
		\[
		z^*
		=
		b_3\bigl(x^*,b_2(x^*)\bigr)
		=
		b_3(x^*,y^*).
		\]
		
		Identity \eqref{eq:b3-identity} gives
		\(F_3(x^*,y^*,z^*)=z^*\)
		and therefore
		\(\frac{\partial\Pi_3}{\partial z}(x^*,y^*,z^*)=0\).
		
		Thus, \eqref{eq:stationary-stackelberg-triple} and
		\eqref{eq:hierarchical-consistency} are satisfied, so
		\((x^*,y^*,z^*)\) is a stationary Stackelberg triple.
	\end{proof}
	
	Theorem~\ref{thm:characterization} establishes a one-to-one
	correspondence between stationary Stackelberg triples and mixed
	tripled fixed points of
	\(({\cal{F}}_1,{\cal{F}}_2,{\cal{F}}_3)\). Since
	\(\mathcal G_1\) is the hierarchical map generated by these three
	maps, Theorem~\ref{thm:unified-41} can be applied whenever its metric
	and contractive assumptions are satisfied.
	
	At this stage, the characterization concerns the first-order
	conditions. The second-order conditions are considered separately in
	Subsection~\ref{subsec:second-order-conditions}.
	
	\begin{remark}
		The fixed-point characterization is not intrinsically restricted to
		three firms. For an \(n\)-firm canonical hierarchy, backward
		induction produces stationary continuation responses
		\(b_n,b_{n-1},\ldots,b_2\). These responses define an \(n\)-component
		hierarchical operator whose fixed points are equivalent to the
		recursive system of stationary first-order conditions. The triopoly
		case is presented explicitly because it permits the construction and
		the distinction between the backward-induction and
		derivative-generated maps to be stated without burdensome notation.
		An \(n\)-firm version requires the corresponding \(n\)-tupled
		fixed-point framework recalled in the Appendix.
	\end{remark}
	
	\subsection{Sufficient Condition}
	\label{subsec:sufficient-condition}
	
	\begin{cor}\label{cor:tripodal}
		Consider a triopoly market satisfying the following conditions:
		\begin{enumerate}
			\item Each firm \(i\) produces quantities in a closed, nonempty
			subset \(X_i\) of \((\mathbb{R},|\,\cdot\,|)\).
			
			\item There exists a nonempty closed set
			\[
			D\subseteq X_1\times X_2\times X_3
			\]
			such that
			\[
			\mathcal G_1(D)\subseteq D,
			\]
			where \(\mathcal G_1\) is given by \eqref{eq:G-trip}.
			
			\item There exists \(\alpha\in(0,1)\) such that, for all
			\(\overline{x},\overline{u}\in D\),
			\begin{equation}\label{eq:trip-contr-0}
				M_1\bigl(
				\mathcal G_1(\overline{x}),
				\mathcal G_1(\overline{u})
				\bigr)
				\leq
				\alpha M_1(\overline{x},\overline{u}).
			\end{equation}
		\end{enumerate}
		
		Then:
		\begin{enumerate}
			\item There exists a unique stationary Stackelberg triple
			\((\xi,\eta,\theta)\in D\), equivalently,
			\begin{equation}\label{eq:trip-fixed-point}
				\mathcal G_1(\xi,\eta,\theta)
				=
				(\xi,\eta,\theta).
			\end{equation}
			
			\item For every initial point
			\[
			(x_0,y_0,z_0)\in D,
			\]
			the fixed-point iteration
			\begin{equation}\label{eq:trip-iteration}
				(x_{n+1},y_{n+1},z_{n+1})
				=
				\mathcal G_1(x_n,y_n,z_n),
				\qquad n=0,1,2,\ldots,
			\end{equation}
			converges to \((\xi,\eta,\theta)\).
			
			\item If, in addition, \(X_1=X_2=X_3=X\) and
			\(\mathcal G_1\) satisfies the local triple cross symmetry
			property at \((\xi,\eta,\theta)\), then
			\[
			\xi=\eta=\theta.
			\]
		\end{enumerate}
	\end{cor}
	
	\begin{proof}
		Since each \(X_i\) is a closed subset of
		\((\mathbb{R},|\,\cdot\,|)\), the product
		\(X_1\times X_2\times X_3\), equipped with the summing metric
		\(M_1\) from Definition~\ref{def:distance-41}, is complete.
		Consequently, its closed subset \(D\) is also complete.
		
		By assumption, \(\mathcal G_1\) maps \(D\) into itself and satisfies
		the contractive condition \eqref{eq:trip-contr-0}. Therefore,
		Theorem~\ref{thm:unified-41}, applied to the restriction of
		\(\mathcal G_1\) to \(D\) with
		\[
		k_1=\alpha,
		\qquad
		k_2=k_3=0,
		\]
		guarantees that \(\mathcal G_1\) has a unique mixed tripled fixed
		point in \(D\) and that the iteration \eqref{eq:trip-iteration}
		converges to this point for every initial point in \(D\).
		
		By Theorem~\ref{thm:characterization}, the unique mixed tripled fixed
		point of \(\mathcal G_1\) is precisely the unique stationary
		Stackelberg triple. The final assertion follows from the local triple
		cross symmetry conclusion in Theorem~\ref{thm:unified-41}.
		\end{proof}
		
		Economically, the contractive condition means that the cumulative
		effect of a perturbation in the firms' quantities becomes smaller
		after one complete backward-induction update. Thus, changes in the
		leader's quantity, together with the induced changes in the
		stationary continuation responses of the middle player and the
		follower, cannot be amplified by the hierarchical response system.
		The stationary Stackelberg triple is therefore the unique stationary
		solution in the invariant set \(D\). Moreover, the associated
		numerical fixed-point iteration converges to this solution from every
		initial point in \(D\). This conclusion concerns the convergence of
		the computational procedure and does not assert stability under an economic
		best-response dynamic.
	
	\subsection{The Second-Order Conditions}
	\label{subsec:second-order-conditions}
	
	The map \(\mathcal G_1\) in \eqref{eq:G-trip} represents the
	backward-induction structure through the stationary responses
	\(b_2\) and \(b_3\). To analyze the second-order conditions, define
	the derivative-generated map
	\begin{equation}\label{eq:derivative-map}
		\begin{aligned}
			\mathcal T_1(x,y,z)
			=
			\Bigl(
			&F_1\bigl(x,b_2(x),b_3(x,b_2(x))\bigr),\\
			&F_2\bigl(x,y,b_3(x,y)\bigr),
			F_3(x,y,z)
			\Bigr).
		\end{aligned}
	\end{equation}
	
	The maps \(\mathcal G_1\) and \(\mathcal T_1\) are generally
	different away from their fixed points. Nevertheless, their fixed
	points coincide.
	
	\begin{lemma}\label{lem:equivalent-fixed-points}
		Under the assumptions of
		Theorem~\ref{thm:characterization}, the maps
		\(\mathcal G_1\) and \(\mathcal T_1\), defined by
		\eqref{eq:G-trip} and \eqref{eq:derivative-map}, respectively, have
		exactly the same fixed points.
	\end{lemma}
	
	\begin{proof}
		Let \((x^*,y^*,z^*)\) be a fixed point of \(\mathcal T_1\). Then
		\begin{equation}\label{eq:T-fixed-system}
			\begin{aligned}
				F_1\bigl(
				x^*,b_2(x^*),b_3(x^*,b_2(x^*))
				\bigr)&=x^*,\\
				F_2\bigl(x^*,y^*,b_3(x^*,y^*)\bigr)&=y^*,\\
				F_3(x^*,y^*,z^*)&=z^*.
			\end{aligned}
		\end{equation}
		By the uniqueness assumptions in
		\eqref{eq:F2-fixed-point} and \eqref{eq:F3-fixed-point}, the last two
		equalities imply
		\[
		y^*=b_2(x^*)
		\]
		and
		\[
		z^*=b_3(x^*,y^*)
		=
		b_3\bigl(x^*,b_2(x^*)\bigr).
		\]
		Together with the first equality in \eqref{eq:T-fixed-system}, these
		relations show that
		\[
		\mathcal G_1(x^*,y^*,z^*)
		=
		(x^*,y^*,z^*).
		\]
		
		Conversely, let \((x^*,y^*,z^*)\) be a fixed point of
		\(\mathcal G_1\). By \eqref{eq:G-trip},
		\[
		F_1\bigl(
		x^*,b_2(x^*),b_3(x^*,b_2(x^*))
		\bigr)=x^*,
		\]
		\[
		y^*=b_2(x^*)
		\]
		and
		\[
		z^*=b_3\bigl(x^*,b_2(x^*)\bigr)
		=
		b_3(x^*,y^*).
		\]
		Identities \eqref{eq:b2-identity} and \eqref{eq:b3-identity} yield
		\[
		F_2\bigl(x^*,y^*,b_3(x^*,y^*)\bigr)=y^*
		\]
		and
		\[
		F_3(x^*,y^*,z^*)=z^*.
		\]
		Therefore,
		\[
		\mathcal T_1(x^*,y^*,z^*)
		=
		(x^*,y^*,z^*).
		\]
	\end{proof}
	
	\begin{proposition}\label{prop:second-order-conditions}
		Let \((x^*,y^*,z^*)\) be the stationary Stackelberg triple obtained
		in Corollary~\ref{cor:tripodal}. Assume that it is an interior point
		of \(D\), that the functions involved are twice continuously
		differentiable in a neighborhood of this point, and that
		\(\mathcal T_1(D)\subseteq D\).
		
		If there exists \(\beta\in(0,1)\) such that
		\begin{equation}\label{eq:derivative-contraction}
			M_1\bigl(
			\mathcal T_1(\overline{x}),
			\mathcal T_1(\overline{u})
			\bigr)
			\leq
			\beta M_1(\overline{x},\overline{u})
		\end{equation}
		for all \(\overline{x},\overline{u}\in D\), then
		\begin{equation}\label{eq:second-order-conclusions}
			\frac{\partial^2\Pi_3}{\partial z^2}(x^*,y^*,z^*)<0,
			\qquad
			\frac{\partial^2\Phi_2}{\partial y^2}(x^*,y^*)<0,
			\qquad
			\frac{d^2\Phi_1}{dx^2}(x^*)<0.
		\end{equation}
		Consequently, \((x^*,y^*,z^*)\) satisfies the second-order
		sufficient conditions for strict local maximization for all three
		firms.
	\end{proposition}
	
	\begin{proof}
		By Lemma~\ref{lem:equivalent-fixed-points},
		\((x^*,y^*,z^*)\) is a fixed point of \(\mathcal T_1\). Hence
		\eqref{eq:T-fixed-system} holds. In view of
		\eqref{eq:F1-equivalence}, \eqref{eq:F2-equivalence}, and
		\eqref{eq:F3-equivalence}, the corresponding first-order conditions
		are satisfied.
		
		First, apply \eqref{eq:derivative-contraction} to
		\[
		\overline{x}=(x^*,y^*,z^*),
		\qquad
		\overline{u}=(x^*,y^*,z^*+\Delta z).
		\]
		The first two components of \(\mathcal T_1\) are independent of
		\(z\). Therefore,
		\[
		\left|
		F_3(x^*,y^*,z^*)
		-
		F_3(x^*,y^*,z^*+\Delta z)
		\right|
		\leq
		\beta|\Delta z|.
		\]
		Dividing by \(|\Delta z|\) and passing to the limit as
		\(\Delta z\to0\), we obtain
		\[
		\left|
		\frac{\partial F_3}{\partial z}(x^*,y^*,z^*)
		\right|
		\leq\beta.
		\]
		By \eqref{eq:def-F3},
		\[
				\frac{\partial^2\Pi_3}{\partial z^2}(x^*,y^*,z^*)=
			\frac{\partial F_3}{\partial z}(x^*,y^*,z^*)-1\leq
			\left|
			\frac{\partial F_3}{\partial z}(x^*,y^*,z^*)
			\right|-1\leq
			\beta-1<0.
		\]
		
		Next, apply \eqref{eq:derivative-contraction} to
		\[
		\overline{x}=(x^*,y^*,z^*),
		\qquad
		\overline{u}=(x^*,y^*+\Delta y,z^*).
		\]
		The first component of \(\mathcal T_1\) is independent of \(y\).
		Using also the independence of \(F_2\) from its third argument, we
		obtain
		\[
		\begin{aligned}
			&
			\left|
			F_2\bigl(x^*,y^*,b_3(x^*,y^*)\bigr)
			-
			F_2\bigl(
			x^*,y^*+\Delta y,b_3(x^*,y^*+\Delta y)
			\bigr)
			\right|\\
			&\quad+
			\left|
			F_3(x^*,y^*,z^*)
			-
			F_3(x^*,y^*+\Delta y,z^*)
			\right|
			\leq
			\beta|\Delta y|.
		\end{aligned}
		\]
		Dividing by \(|\Delta y|\) and passing to the limit as
		\(\Delta y\to0\), we find
		\[
		\left|
		\frac{\partial F_2}{\partial y}
		\bigl(x^*,y^*,b_3(x^*,y^*)\bigr)
		\right|
		+
		\left|
		\frac{\partial F_3}{\partial y}(x^*,y^*,z^*)
		\right|
		\leq\beta.
		\]
		In particular,
		\[
		\left|
		\frac{\partial F_2}{\partial y}
		\bigl(x^*,y^*,b_3(x^*,y^*)\bigr)
		\right|
		\leq\beta.
		\]
		By \eqref{eq:def-F2},
		\[
		\begin{aligned}
			\frac{\partial^2\Phi_2}{\partial y^2}(x^*,y^*)
			&=
			\frac{\partial F_2}{\partial y}
			\bigl(x^*,y^*,b_3(x^*,y^*)\bigr)-1\\
			&\leq
			\left|
			\frac{\partial F_2}{\partial y}
			\bigl(x^*,y^*,b_3(x^*,y^*)\bigr)
			\right|-1\\
			&\leq
			\beta-1<0.
		\end{aligned}
		\]
		
		Finally, apply \eqref{eq:derivative-contraction} to
		\[
		\overline{x}=(x^*,y^*,z^*),
		\qquad
		\overline{u}=(x^*+\Delta x,y^*,z^*).
		\]
		We obtain
		\[
		\begin{aligned}
			&
			\left|
			F_1\bigl(
			x^*,b_2(x^*),b_3(x^*,b_2(x^*))
			\bigr)
			-
			F_1\bigl(
			x^*+\Delta x,
			b_2(x^*+\Delta x),
			b_3(x^*+\Delta x,b_2(x^*+\Delta x))
			\bigr)
			\right|\\
			&\quad+
			\left|
			F_2\bigl(x^*,y^*,b_3(x^*,y^*)\bigr)
			-
			F_2\bigl(
			x^*+\Delta x,y^*,b_3(x^*+\Delta x,y^*)
			\bigr)
			\right|\\
			&\quad+
			\left|
			F_3(x^*,y^*,z^*)
			-
			F_3(x^*+\Delta x,y^*,z^*)
			\right|
			\leq
			\beta|\Delta x|.
		\end{aligned}
		\]
		Dividing by \(|\Delta x|\), passing to the limit as
		\(\Delta x\to0\), and retaining the first nonnegative term, we obtain
		\[
		\left|
		\frac{d}{dx}
		F_1\bigl(
		x,b_2(x),b_3(x,b_2(x))
		\bigr)\bigg|_{x=x^*}
		\right|
		\leq\beta.
		\]
		By \eqref{eq:def-F1},
		\[
		F_1\bigl(
		x,b_2(x),b_3(x,b_2(x))
		\bigr)
		=
		x+\frac{d\Phi_1}{dx}(x).
		\]
		Consequently,
		\[
		\begin{aligned}
			\frac{d^2\Phi_1}{dx^2}(x^*)
			&=
			\frac{d}{dx}
			F_1\bigl(
			x,b_2(x),b_3(x,b_2(x))
			\bigr)\bigg|_{x=x^*}-1\\
			&\leq
			\left|
			\frac{d}{dx}
			F_1\bigl(
			x,b_2(x),b_3(x,b_2(x))
			\bigr)\bigg|_{x=x^*}
			\right|-1\\
			&\leq
			\beta-1<0.
		\end{aligned}
		\]
		Thus, all three inequalities in
		\eqref{eq:second-order-conclusions} hold.
	\end{proof}
	
	The conclusion of
	Proposition~\ref{prop:second-order-conditions} is local. Global
	optimality requires additional assumptions, such as concavity of the
	corresponding payoff functions on their entire strategy sets.
	
	When the component functions of \(\mathcal G_1\) are continuously
	differentiable on a convex invariant set \(D\), the contractive
	condition in Corollary~\ref{cor:tripodal} can be verified through
	bounds on the Jacobian of \(\mathcal G_1\). In the metric \(M_1\), it
	is sufficient that the induced column-sum norm of
	\(D\mathcal G_1\) be uniformly smaller than one on \(D\). Through the
	implicit-function formulas for \(b_2\) and \(b_3\), these bounds can
	be expressed in terms of the first and second derivatives of the
	inverse demand and cost functions. In particular, stronger own-output
	curvature reduces the magnitude of the continuation-response
	derivatives, whereas strong cross-output effects may prevent
	contraction. This is related to the familiar role of strategic
	substitutability and contraction conditions in oligopoly models; see
	Vives~\cite[pp.~47--51]{Vives}. The condition used here is a uniform 
	sufficient condition adapted to
	the hierarchical backward-induction operator. It is generally
	stronger than a local condition based only on the Jacobian evaluated
	at the stationary solution. Moreover, the present condition concerns
	the convergence of the numerical fixed-point iteration, whereas the
	stability conditions considered in standard oligopoly dynamics
	typically concern an explicitly specified economic adjustment
	process.}

\subsection{A Contractive Illustrative Example}
\label{subsec:contractive-example}

{\begin{example}\label{ex:contractive-example}
	Let \(x\), \(y\), and \(z\) denote the output levels of the three
	firms, and consider the linear inverse demand function
	\(P(Z)=A-BZ\), \(Z=x+y+z\).
	Assume that all firms have the quadratic cost functions
	\(c_i(u)=Cu^2\) for \(i=1,2,3\). 
	
	Let \(A=12\), \(B=\frac{1}{20}\), and \(C=\frac{1}{40}\).
	Thus, the inverse demand function and the cost functions are
	\(P(Z)=12-\frac{1}{20}Z\) and \(c_i(u)=\frac{1}{40}u^2\) for
	\(i=1,2,3\).
\end{example}

For the parameter values in Example~\ref{ex:contractive-example}, the
derivative-generated map \(\mathcal T_1\), defined in
\eqref{eq:derivative-map}, is given by
\[
	\mathcal T_1(x,y,z)
	=
	\left(
	\frac{40}{7}+\frac{379}{420}x,\,
	8-\frac{1}{30}x+\frac{53}{60}y,\,
	12-\frac{1}{20}x-\frac{1}{20}y+\frac{17}{20}z
	\right).
\]

Let
\(\overline{x}=(x,y,z)
\qquad\text{and}\qquad
\overline{u}=(u,v,w)\).
Using the sum metric \(M_1\), we obtain
\[
\begin{aligned}
	&M_1\bigl(
	\mathcal T_1(\overline{x}),
	\mathcal T_1(\overline{u})
	\bigr)\\
	&\quad\leq
	\left(
	\frac{379}{420}
	+
	\frac{1}{30}
	+
	\frac{1}{20}
	\right)|x-u|
	+
	\left(
	\frac{53}{60}
	+
	\frac{1}{20}
	\right)|y-v|
	+
	\frac{17}{20}|z-w|\\
	&\quad=
	\frac{69}{70}|x-u|
	+
	\frac{14}{15}|y-v|
	+
	\frac{17}{20}|z-w|.
\end{aligned}
\]

Since
\(\frac{14}{15}<\frac{69}{70}\) and \(\frac{17}{20}<\frac{69}{70}\),
it follows that
\(M_1\bigl(
	\mathcal T_1(\overline{x}),
	\mathcal T_1(\overline{u})
	\bigr)
	\leq
	\frac{69}{70}
	M_1(\overline{x},\overline{u})\).

Thus, \(\mathcal T_1\) is a contraction.

Consequently, \(\mathcal T_1\) has a unique fixed point, and the
fixed-point iteration
\begin{equation}\label{eq:example-T-iteration}
	(x_{n+1},y_{n+1},z_{n+1})
	=
	\mathcal T_1(x_n,y_n,z_n),
	\qquad n=0,1,2,\ldots,
\end{equation}
converges to it from every initial point in the invariant domain.

The fixed-point equations are
\[
\left\{
\begin{aligned}
	x&=\frac{40}{7}+\frac{379}{420}x,\\
	y&=8-\frac{1}{30}x+\frac{53}{60}y,\\
	z&=12-\frac{1}{20}x-\frac{1}{20}y+\frac{17}{20}z.
\end{aligned}
\right.
\]

Solving this system gives
\((x^*,y^*,z^*)
	=
	\left(
	\frac{2400}{41},\,
	\frac{14880}{287},\,
	\frac{12400}{287}
	\right)\).

Numerically,
\((x^*,y^*,z^*)
\approx
(58.537,\;51.847,\;43.206)\) (Figure \ref{fig:contractive-iteration-firms}).

By the equivalence of the fixed points of \(\mathcal G_1\) and
\(\mathcal T_1\), established in
Lemma~\ref{lem:equivalent-fixed-points}, this fixed point coincides
with the stationary Stackelberg triple obtained by backward induction:
\(y^*=b_2(x^*)\) and \(z^*=b_3(x^*,y^*)\).

\begin{figure}[!htbp]
	\centering
	\includegraphics[width=0.82\textwidth]{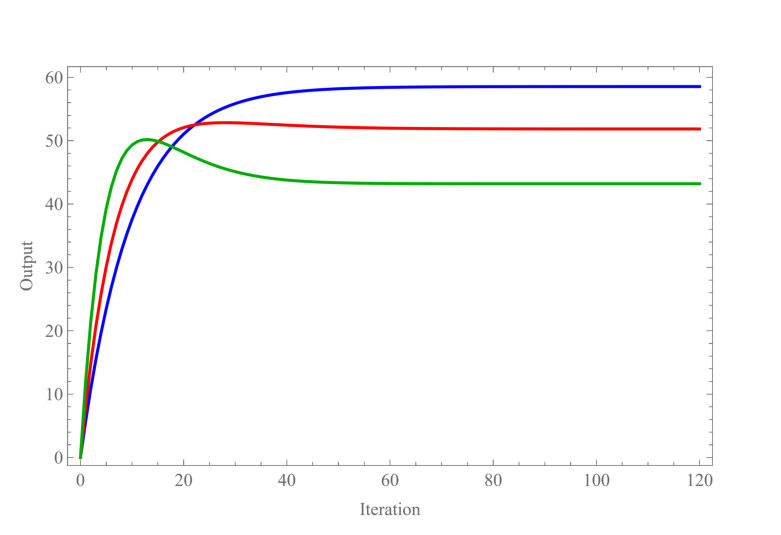}
	\caption{Convergence of the firm-level quantities under the fixed-point iteration \(\mathcal T_1\). Blue: Firm~1. Red: Firm~2. Green: Firm~3. The iterates converge to the unique fixed point \((x^*,y^*,z^*)\approx(58.537,51.847,43.206)\).}
	\label{fig:contractive-iteration-firms}
\end{figure}

Consequently, the aggregate output at the stationary point is
\(Q^*=x^*+y^*+z^*
=\frac{44080}{287}
\approx 153.589\),
and the convergence of the aggregate output to \(Q^*\) is illustrated in
Figure~\ref{fig:contractive-iteration-aggregate}.

\begin{figure}[!htbp]
	\centering
	\includegraphics[width=0.82\textwidth]{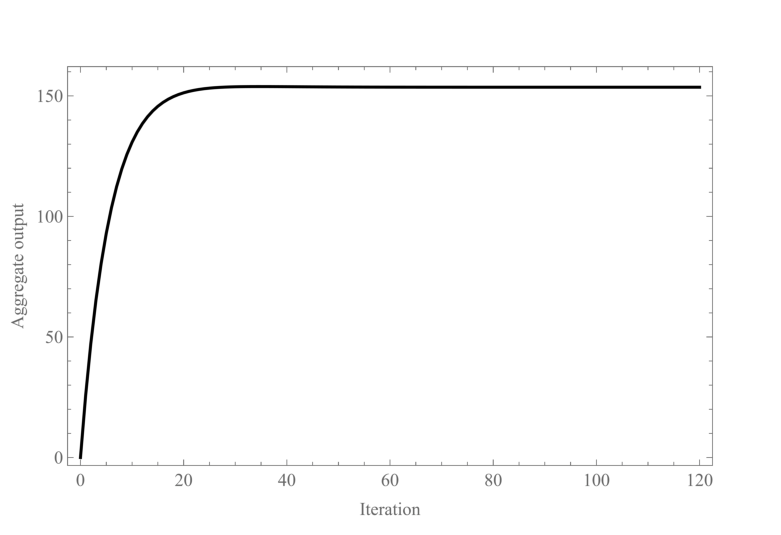}
	\caption{Convergence of aggregate output under the fixed-point iteration \(\mathcal T_1\).}
	\label{fig:contractive-iteration-aggregate}
\end{figure}

We next verify the second-order conditions. From the third component
of \(\mathcal T_1\), we obtain
\[
\frac{\partial^2\Pi_3}{\partial z^2}(x^*,y^*,z^*)
=
\frac{\partial F_3}{\partial z}(x^*,y^*,z^*)-1
=
\frac{17}{20}-1
=
-\frac{3}{20}<0.
\]
Similarly, the second component gives
\[
\frac{\partial^2\Phi_2}{\partial y^2}(x^*,y^*)
=
\frac{\partial F_2}{\partial y}(x^*,y^*)-1
=
\frac{53}{60}-1
=
-\frac{7}{60}<0,
\]
whereas the first component gives
\[
\frac{d^2\Phi_1}{dx^2}(x^*)
=
\frac{dF_1}{dx}(x^*)-1
=
\frac{379}{420}-1
=
-\frac{41}{420}<0.
\]
Therefore, the second-order sufficient conditions are satisfied for
all three firms at the stationary Stackelberg triple.

Since the corresponding reduced payoff functions are quadratic and
strictly concave in their respective decision variables, the
stationary triple is the unique Stackelberg equilibrium for the
parameter values in this example.

It is important to emphasize that
\eqref{eq:example-T-iteration} is a numerical fixed-point iteration.
It is not a best-response dynamic and should not be interpreted as
describing the actual adjustment of the firms' production levels over
time.

We next consider an example in which the first- and second-order
optimality conditions are satisfied, but the corresponding
derivative-generated fixed-point iteration is not convergent from
arbitrary initial points.

\begin{example}\label{ex:noncontractive-iteration}
	considers another model in the same class of linear-demand, quadratic-cost
	models as that in
	Example~\ref{ex:contractive-example}, i.e.,
	\(P(Z)=A-BZ\), \(Z=x+y+z\),
	and assume that all three firms have the quadratic cost functions
	\(c_i(u)=Cu^2\) for \(i=1,2,3\).
	
	In Example~\ref{ex:contractive-example}:
	\(A=12\), \(B=\frac{1}{20}\), and \(C=\frac{1}{40}\). In the present
	illustration, we take
	\(A=30996\), \(B=1\), and \(C=\frac{1}{40}\).
	
	Thus,
	\(P(Z)=30996-Z\)
	and
	\(c_i(u)=\frac{1}{40}u^2\) for \(i=1,2,3\).
	The cost coefficient therefore remains unchanged, whereas the slope of the inverse demand function changes from
	\(1/20\) to \(1\).
	
	For these parameter values, the derivative-generated map is
	\begin{equation}\label{eq:noncontractive-map}
		\mathcal T_1(x,y,z)
		=
		\left(
		\frac{7318836}{881}
		-\frac{2990591}{722420}x,\,
		15876-\frac{21}{41}x-\frac{61}{820}y,\,
		30996-x-y-\frac{21}{20}z
		\right).
	\end{equation}
	
	The corresponding second derivatives are
	\[
	\frac{\partial^2\Pi_3}{\partial z^2}
	=
	-\frac{21}{20}-1
	=
	-\frac{41}{20}<0,
	\]
	\[
	\frac{\partial^2\Phi_2}{\partial y^2}
	=
	-\frac{61}{820}-1
	=
	-\frac{881}{820}<0,
	\]
	and
	\[
	\frac{d^2\Phi_1}{dx^2}
	=
	-\frac{2990591}{722420}-1
	<0.
	\]
	Hence the fixed point satisfies the second-order sufficient
	conditions. Since the corresponding reduced payoff functions are
	quadratic and strictly concave in their respective decision
	variables, the stationary solution determines the unique
	Stackelberg equilibrium for this example.
	
	However, \(\mathcal T_1\) is not a contraction with respect to
	\(M_1\). 
	
	Indeed,
	\(\frac{2990591}{722420}
	+\frac{21}{41}
	+1
	>1\), 
	\(\frac{61}{820}+1>1\), and 
	\(\frac{21}{20}>1\).
	Therefore, the contractive condition
	\eqref{eq:derivative-contraction} is not satisfied.
	
	The fixed-point system associated with
	\eqref{eq:noncontractive-map} is
	\[
		\left\{
		\begin{array}{rcl}
			\displaystyle
			\frac{7318836}{881}
			-\frac{2990591}{722420}x&=&x,\\[8pt]
			\displaystyle
			15876-\frac{21}{41}x-\frac{61}{820}y&=&y,\\[8pt]
			\displaystyle
			30996-x-y-\frac{21}{20}z&=&z.
		\end{array}
		\right.
	\]
	
	The linear part of \(\mathcal T_1\) is represented by the lower
	triangular matrix
	\[
		A=
		\begin{pmatrix}
			-\dfrac{2990591}{722420} & 0 & 0\\[8pt]
			-\dfrac{21}{41} & -\dfrac{61}{820} & 0\\[8pt]
			-1 & -1 & -\dfrac{21}{20}
		\end{pmatrix}.
	\]
	
	Its eigenvalues are its diagonal entries:
	\(-\frac{2990591}{722420}\), \(-\frac{61}{820}\), and \(-\frac{21}{20}\).
	Since
	\(\left|
	-\frac{2990591}{722420}
	\right|>1\) and	\(\left|-\frac{21}{20}\right|>1\), 
	the fixed point is unstable with respect to the numerical iteration
	generated by \(\mathcal T_1\). Consequently, the iteration does not converge from
	arbitrary initial points; in particular, it diverges for generic
	initial points having a nonzero component in an unstable
	eigendirection.
\end{example}

The two examples illustrate the distinction between optimality and
computational convergence. The first- and second-order conditions concern local
optimality and, under appropriate concavity assumptions, global
optimality of the stationary solution. By contrast, the contractive condition concerns
the convergence of the numerical fixed-point iteration. Thus, an
equilibrium may exist and satisfy the required optimality conditions
even when the derivative-generated iteration fails to converge from
arbitrary initial points.}

We next extend our analysis, and the intuition
developed from the numerical examples, to the \(n\)-firm case. This extension serves two purposes. First, it provides a unified analytical framework in which the three-firm outcomes appear as a special case. Second, it allows us to compare the convergence of aggregate output in the two market organizations toward the competitive benchmark \(A/B\), and hence to interpret the limiting behavior of prices, profits, and surplus.

\section{Large-Market Comparison of the Cournot and Hierarchical
	Stackelberg Models}
\label{sec:large-market-comparison}

{We now extend the three-firm analysis to a market with \(n\) firms.
Our purpose is to compare the Cournot and canonical hierarchical
Stackelberg models under a common linear inverse demand function and
two alternative cost specifications. Particular attention is given
to the rates at which aggregate output and market price approach
their respective large-market limits.

The comparison between Cournot and sequential Stackelberg
oligopolies has an established literature. In particular, the
aggregate-output comparison under linear costs is related to the
results of \cite{Anderson_Engers_1992}, while the convergence of sequential
market outcomes toward a competitive limit is related to \cite{Robson1990} and 
\cite{boyer1986perfect}. Accordingly, we do not present the existence
of the large-market limit or the fixed-point representation of an
equilibrium as new contributions. Our analysis instead focuses on
the recursive structure arising under quadratic costs and on an
explicit comparison of the convergence rates under the four
combinations of market organization and cost specification.

The principal results are stated formally in this section. Their
longer algebraic derivations are collected in
Appendix~\ref{app:large-market-proofs}. Numerical values and figures
are used only to illustrate the general results and are presented
separately from the theoretical statements.

\subsection{Model, Notation, and Relation to the Literature}
\label{subsec:large-market-notation}

Consider \(n\) firms producing a homogeneous good. Let \(x_i\geq0\)
denote the output of Firm \(i\), and let
\(Q=\sum_{i=1}^{n}x_i\)
be the aggregate market output. The inverse demand function is
\begin{equation}\label{eq:linear-inverse-demand-n}
	P(Q)=A-BQ,
	\qquad A>0,\quad B>0.
\end{equation}

We consider two common cost specifications. Under linear costs,
\begin{equation}\label{eq:linear-cost-n}
	C_i^{L}(x_i)=cx_i,
	\qquad 0\leq c<A,
\end{equation}
whereas under quadratic costs,
\begin{equation}\label{eq:quadratic-cost-n}
	C_i^{Q}(x_i)=cx_i^2,
	\qquad c>0.
\end{equation}

The corresponding payoff function is
\[
\Pi_i(x_1,\ldots,x_n)
	=
	x_i\left(
	A-B\sum_{j=1}^{n}x_j
	\right)
	-
	C_i(x_i).
\]

The superscript \(C\) denotes the simultaneous Cournot model, and the
superscript \(S\) denotes the canonical hierarchical Stackelberg
model in which the firms choose their quantities sequentially in the
order
\(1,2,\ldots,n\).
Thus, Firm \(1\) is the initial leader and Firm \(n\) is the final
follower. The second superscript identifies the cost specification:
\(L\) for linear costs and \(Q\) for quadratic costs.

We use the notation
\(Q_n^{C,L}\), \(Q_n^{S,L}\), \(Q_n^{C,Q}\), and \(Q_n^{S,Q}\)
for the aggregate equilibrium outputs in the four models. The
corresponding equilibrium prices are denoted by
\begin{equation}\label{eq:price-notation}
	P_n^{X,Y}
	=
	A-BQ_n^{X,Y},
	\qquad
	X\in\{C,S\},\quad
	Y\in\{L,Q\}.
\end{equation}

When individual Stackelberg quantities are required, we write
\(x_{k,n}^{S,Y}\) for \(k=1,\ldots,n,\) and \(Y\in\{L,Q\}\),
where \(k\) denotes the position of the firm in the sequence of
moves. In particular, \(x_{1,n}^{S,Y}\) is the output of the initial
leader and \(x_{n,n}^{S,Y}\) is the output of the final follower.
This notation avoids confusing the position \(k\) of a firm with the
total number \(n\) of firms.

The natural limiting output depends on the cost specification. Under
linear costs, the zero-profit competitive output is
\(M_L
	=
	\frac{A-c}{B}\).
Under quadratic costs, marginal cost at the output of each individual
firm converges to zero as the number of firms increases and individual
output tends to zero. The corresponding large-market output is
\(M_Q
	=
	\frac{A}{B}\).
For brevity, we shall write
\begin{equation}\label{eq:M-notation}
	\widetilde M=M_L=\frac{A-c}{B},
	\qquad
	M=M_Q=\frac{A}{B}.
\end{equation}

Whenever the limits exist, we use
\(Q_\infty^{X,Y}
	=
	\lim_{n\to\infty}Q_n^{X,Y}\) and \(P_\infty^{X,Y}
	=
	\lim_{n\to\infty}P_n^{X,Y}\).

The existence of equilibrium in the finite sequential game follows
from backward induction under the regularity and concavity conditions
imposed on the corresponding continuation problems. The results
below concern the quantities generated by these equilibria, their
comparison with the Cournot quantities, and their rates of convergence
as \(n\to\infty\). Thus, the large-market analysis complements rather
than replaces the standard equilibrium-existence arguments.

For linear costs, the formulas for the sequential quantities and the
comparison with Cournot competition are closely related to
\cite{Anderson_Engers_1992}. The convergence
toward competitive behavior should also be viewed in relation to
\cite{Robson1990}. The results that require separate
derivation here are the recursive representation under quadratic
costs and the explicit comparison of the residual convergence rates.

\subsection{Aggregate Outputs under Quadratic Costs}
\label{subsec:quadratic-aggregate-outputs}

Suppose that all firms have the quadratic cost function
\eqref{eq:quadratic-cost-n}. Introduce the parameter
\(\delta=\frac{c}{B}>0\).
Recall from \eqref{eq:M-notation} that
\(M=\frac{A}{B}\).

We first record the aggregate output in the simultaneous Cournot
model.

\begin{proposition}\label{prop:quadratic-cournot-output}
	Under the inverse demand function
	\eqref{eq:linear-inverse-demand-n} and the quadratic cost function
	\eqref{eq:quadratic-cost-n}, the \(n\)-firm Cournot equilibrium is
	symmetric and satisfies
	\begin{equation}\label{eq:quadratic-cournot-individual-output}
		x_{i,n}^{C,Q}
		=
		\frac{A}{B(n+1)+2c}
		=
		\frac{M}{n+1+2\delta},
		\qquad i=1,\ldots,n.
	\end{equation}
	
	Consequently, the aggregate Cournot output is
	\begin{equation}\label{eq:quadratic-cournot-aggregate-output}
		Q_n^{C,Q}
		=
		\frac{An}{B(n+1)+2c}
		=
		M\frac{n}{n+1+2\delta}.
	\end{equation}
	
	The corresponding market price is
	\begin{equation}\label{eq:quadratic-cournot-price}
		P_n^{C,Q}
		=
		A-BQ_n^{C,Q}
		=
		A\frac{1+2\delta}{n+1+2\delta}.
	\end{equation}
\end{proposition}

The proof follows directly from the symmetric Cournot first-order
conditions and is included in
Appendix~\ref{app:quadratic-outputs} for completeness.

We next consider the canonical hierarchical Stackelberg model. Let
\(Q_n^{S,Q}\) denote the aggregate output in its subgame-perfect
equilibrium. To describe this output, define the normalized sequence
\begin{equation}\label{eq:quadratic-stackelberg-normalization}
	q_n
	=
	\frac{Q_n^{S,Q}}{M},
	\qquad n\geq0.
\end{equation}

\begin{theorem}\label{thm:quadratic-stackelberg-recursion}
	Under the inverse demand function
	\eqref{eq:linear-inverse-demand-n} and the quadratic cost function
	\eqref{eq:quadratic-cost-n}, the normalized aggregate output in the
	canonical hierarchical Stackelberg model satisfies
	\begin{equation}\label{eq:q-initial-value}
		q_0=0
	\end{equation}
	and
	\begin{equation}\label{eq:q-recursion}
		q_n
		=
		q_{n-1}
		+
		\frac{(1-q_{n-1})^2}
		{2(1-q_{n-1}+\delta)},
		\qquad n\geq1.
	\end{equation}
	
	Equivalently,
	\begin{equation}\label{eq:quadratic-stackelberg-aggregate-output}
		Q_n^{S,Q}
		=
		M q_n.
	\end{equation}
	
	In particular,
	\begin{equation}\label{eq:q-first-value}
		q_1
		=
		\frac{1}{2(1+\delta)}
	\end{equation}
	and
	\begin{equation}\label{eq:quadratic-stackelberg-one-firm}
		Q_1^{S,Q}
		=
		\frac{A}{2(B+c)}.
	\end{equation}
\end{theorem}

The detailed backward-induction derivation of
\eqref{eq:q-recursion}, including the recursive continuation outputs
of the individual firms, is given in
Appendix~\ref{app:quadratic-outputs}. The recursion is stated here in
normalized form because it separates the scale parameter
\(M=A/B\) from the cost-to-demand ratio
\(\delta=c/B\).

The next result describes the qualitative behavior of the aggregate
Stackelberg output.

\begin{proposition}\label{prop:quadratic-stackelberg-monotonicity}
	Under the assumptions of
	Theorem~\ref{thm:quadratic-stackelberg-recursion}, the sequence
	\(\{q_n\}_{n\geq0}\) satisfies
	\begin{equation}\label{eq:q-bounds}
		0\leq q_n<1,
		\qquad n\geq0,
	\end{equation}
	and is strictly increasing:
	\begin{equation}\label{eq:q-monotonicity}
		q_n>q_{n-1},
		\qquad n\geq1.
	\end{equation}
	
	Consequently,
	\(\lim_{n\to\infty}q_n=1\)
	and hence
	\begin{equation}\label{eq:quadratic-stackelberg-output-limit}
		\lim_{n\to\infty}Q_n^{S,Q}
		=
		M
		=
		\frac{A}{B}.
	\end{equation}
	
	The equilibrium price therefore satisfies
	\begin{equation}\label{eq:quadratic-stackelberg-price-limit}
		\lim_{n\to\infty}P_n^{S,Q}=0.
	\end{equation}
\end{proposition}

\begin{proof}
	Suppose that \(0\leq q_{n-1}<1\). By
	\eqref{eq:q-recursion},
	\(q_n-q_{n-1}
	=
	\frac{(1-q_{n-1})^2}
	{2(1-q_{n-1}+\delta)}
	>0\).
	Thus, the sequence is strictly increasing as long as
	\(q_{n-1}<1\).
	
	Moreover,
	\[
		1-q_n=
		1-q_{n-1}
		-
		\frac{(1-q_{n-1})^2}
		{2(1-q_{n-1}+\delta)}=
		\frac{
			(1-q_{n-1})(1-q_{n-1}+2\delta)
		}{
			2(1-q_{n-1}+\delta)
		}
		>0.
	\]
	
	Hence \(q_n<1\). Since \(q_0=0\), induction gives
	\eqref{eq:q-bounds} and \eqref{eq:q-monotonicity}.
	
	Therefore, \(\{q_n\}\) is increasing and bounded above, so it
	converges to some \(q_\infty\leq1\). Passing to the limit in
	\eqref{eq:q-recursion} gives
	\(\frac{(1-q_\infty)^2}
	{2(1-q_\infty+\delta)}
	=
	0\),
	and hence \(q_\infty=1\). Equations
	\eqref{eq:quadratic-stackelberg-output-limit} and
	\eqref{eq:quadratic-stackelberg-price-limit} now follow from
	\eqref{eq:quadratic-stackelberg-aggregate-output} and
	\eqref{eq:price-notation}.
\end{proof}

The following proposition compares the simultaneous and sequential
market organizations.

\begin{proposition}\label{prop:quadratic-output-comparison}
	Under the assumptions above,
	\(Q_n^{S,Q}\geq Q_n^{C,Q}\) for \(n\geq1\).
	
	Equality holds for \(n=1\), whereas
	\(Q_n^{S,Q}>Q_n^{C,Q}\) for \(n\geq2\).
		
	Consequently,
	\(P_n^{S,Q}<P_n^{C,Q}\) for \(n\geq2\).
\end{proposition}

The proof of
Proposition~\ref{prop:quadratic-output-comparison} is given in
Appendix~\ref{app:quadratic-outputs}. It is based on the normalized
recursion \eqref{eq:q-recursion} and the Cournot formula
\eqref{eq:quadratic-cournot-aggregate-output}.

Propositions~\ref{prop:quadratic-stackelberg-monotonicity} and
\ref{prop:quadratic-output-comparison} show that sequential
commitment increases aggregate output relative to simultaneous
Cournot competition for every finite \(n\geq2\). Nevertheless, both
market organizations approach the same large-market output
\(M=A/B\). Thus, the principal difference between the two models in
large markets is not their limiting output but the finite-\(n\)
output gap and the rate at which that gap disappears. The convergence
rates are analyzed separately in
Subsection~\ref{subsec:convergence-rates}.

\subsection{Aggregate Outputs under Linear Costs}
\label{subsec:linear-aggregate-outputs}

Suppose that all firms have the linear cost function
\eqref{eq:linear-cost-n}. Recall that
\(\widetilde M
=
\frac{A-c}{B}\)
is the limiting competitive output under linear marginal cost
\(c\).

The equilibrium quantities in the simultaneous Cournot model are
standard.

\begin{proposition}\label{prop:linear-cournot-output}
	Under the inverse demand function
	\eqref{eq:linear-inverse-demand-n} and the linear cost function
	\eqref{eq:linear-cost-n}, the \(n\)-firm Cournot equilibrium is
	symmetric and satisfies
	\begin{equation}\label{eq:linear-cournot-individual-output}
		x_{i,n}^{C,L}
		=
		\frac{A-c}{B(n+1)}
		=
		\frac{\widetilde M}{n+1},
		\qquad i=1,\ldots,n.
	\end{equation}
	
	Consequently, the aggregate Cournot output is
	\begin{equation}\label{eq:linear-cournot-aggregate-output}
		Q_n^{C,L}
		=
		\frac{n(A-c)}{B(n+1)}
		=
		\widetilde M\frac{n}{n+1},
	\end{equation}
	and the corresponding market price is
	\begin{equation}\label{eq:linear-cournot-price}
		P_n^{C,L}
		=
		A-BQ_n^{C,L}
		=
		c+\frac{A-c}{n+1}.
	\end{equation}
\end{proposition}

We next consider the canonical hierarchical Stackelberg model. The
following sequential-output formula is closely related to the
analysis of \cite{Anderson_Engers_1992} and \cite{boyer1986perfect}. 
We state it here to establish the notation
required for the subsequent convergence-rate comparison.

\begin{proposition}\label{prop:linear-stackelberg-output}
	Under the inverse demand function
	\eqref{eq:linear-inverse-demand-n} and the linear cost function
	\eqref{eq:linear-cost-n}, the output of Firm \(k\) in the canonical
	\(n\)-firm hierarchical Stackelberg equilibrium is
	\begin{equation}\label{eq:linear-stackelberg-individual-output}
		x_{k,n}^{S,L}
		=
		\frac{\widetilde M}{2^k},
		\qquad k=1,\ldots,n.
	\end{equation}
	
	Consequently, the aggregate Stackelberg output is
	\begin{equation}\label{eq:linear-stackelberg-aggregate-output}
			Q_n^{S,L}=
			\sum_{k=1}^{n}x_{k,n}^{S,L}=
			\widetilde M
			\sum_{k=1}^{n}2^{-k}=
			\widetilde M\left(1-2^{-n}\right).
\end{equation}

	The corresponding market price is
	\begin{equation}\label{eq:linear-stackelberg-price}
		P_n^{S,L}
		=
		A-BQ_n^{S,L}
		=
		c+(A-c)2^{-n}.
	\end{equation}
\end{proposition}

A backward-induction derivation of
\eqref{eq:linear-stackelberg-individual-output} is provided in
Appendix~\ref{app:linear-outputs}. The formula shows that every firm
produces one half of the quantity produced by its immediate
predecessor:
\(x_{k+1,n}^{S,L}
	=
	\frac12 x_{k,n}^{S,L}\) for \(k=1,\ldots,n-1\).

The next result records the limiting behavior of the two market
organizations.

\begin{proposition}\label{prop:linear-output-limits}
	The aggregate output sequences
	\(\{Q_n^{C,L}\}_{n\geq1}\) and
	\(\{Q_n^{S,L}\}_{n\geq1}\) are strictly increasing and satisfy
	\(\lim_{n\to\infty}Q_n^{C,L}
		=
		\lim_{n\to\infty}Q_n^{S,L}
		=
		\widetilde M
		=
		\frac{A-c}{B}\).
	
	The corresponding equilibrium prices satisfy
	\(\lim_{n\to\infty}P_n^{C,L}
		=
		\lim_{n\to\infty}P_n^{S,L}
		=
		c\).
\end{proposition}

\begin{proof}
	The conclusions follow directly from
	\eqref{eq:linear-cournot-aggregate-output} and
	\eqref{eq:linear-stackelberg-aggregate-output}. Indeed,
	\(\frac{n}{n+1}\)
	and
	\(1-2^{-n}\)
	are strictly increasing sequences converging to \(1\). The price
	limits then follow from
	\eqref{eq:linear-cournot-price} and
	\eqref{eq:linear-stackelberg-price}.
\end{proof}

We can now compare the finite-market outcomes.

\begin{proposition}\label{prop:linear-output-comparison}
	For \(n=1\), the Cournot and Stackelberg models coincide:
	\[Q_1^{S,L}
		=
		Q_1^{C,L}
		=
		\frac{\widetilde M}{2}.\]
	
	For every \(n\geq2\),
	\begin{equation}\label{eq:linear-output-strict-comparison}
		Q_n^{S,L}>Q_n^{C,L}
	\end{equation}
	and, consequently,
	\(P_n^{S,L}<P_n^{C,L}\).
	
	More precisely,
	\(Q_n^{S,L}-Q_n^{C,L}
		=
		\widetilde M
		\left(
		\frac{1}{n+1}-2^{-n}
		\right)\).
\end{proposition}

\begin{proof}
	Using \eqref{eq:linear-cournot-aggregate-output} and
	\eqref{eq:linear-stackelberg-aggregate-output}, we obtain
	\[
		Q_n^{S,L}-Q_n^{C,L}=
		\widetilde M
		\left(
		1-2^{-n}-\frac{n}{n+1}
		\right)=
		\widetilde M
		\left(
		\frac{1}{n+1}-2^{-n}
		\right).
	\]
	
	For \(n=1\), the right-hand side is zero. For every \(n\geq2\) there holds
	\(2^n>n+1\),
	and therefore
	\(2^{-n}<\frac{1}{n+1}\).
	This proves
	\eqref{eq:linear-output-strict-comparison}. The price comparison
	follows from the strict monotonicity of the inverse demand function.
\end{proof}

The hierarchical structure also generates an ordering of the
individual output levels.

\begin{cor}\label{cor:linear-individual-output-ordering}
	For every \(n\geq2\),
	\(x_{n,n}^{S,L}
		<
		x_{i,n}^{C,L}
		<
		x_{1,n}^{S,L}\) for \(i=1,\ldots,n\).
		Equivalently,
	\[\frac{\widetilde M}{2^n}
		<
		\frac{\widetilde M}{n+1}
		<
		\frac{\widetilde M}{2}.
		\]
\end{cor}

\begin{proof}
	The first inequality follows from
	\[
	2^n>n+1,
	\qquad n\geq2,
	\]
	whereas the second follows from
	\[
	n+1>2.
	\]
	The explicit formulas are given by
	\eqref{eq:linear-cournot-individual-output} and
	\eqref{eq:linear-stackelberg-individual-output}.
\end{proof}

Proposition~\ref{prop:linear-output-comparison} is consistent with the
established comparison between simultaneous and sequential
quantity-setting models: sequential commitment produces a larger
aggregate output and a lower price for every finite \(n\geq2\).
Accordingly, this finite-\(n\) ordering is not presented as a new
contribution. Its role here is to provide the common basis for the
comparison of convergence rates.

The formulas also reveal an important difference between aggregate
and individual outcomes. Aggregate Stackelberg output exceeds
aggregate Cournot output, but the sequential allocation is
asymmetric. The initial leader produces more than a Cournot firm,
whereas the final follower produces less. As \(n\) increases, the
outputs of firms appearing later in the sequence decrease
geometrically.

Although the Cournot and Stackelberg aggregate outputs converge to
the same limit \(\widetilde M\), their residual gaps have different
orders:
\[
\widetilde M-Q_n^{C,L}
=
\frac{\widetilde M}{n+1},
\qquad
\widetilde M-Q_n^{S,L}
=
\widetilde M\,2^{-n}.
\]
These exact identities will be incorporated into the unified
comparison in
Subsection~\ref{subsec:convergence-rates}.

\subsection{Comparison of Convergence Rates}
\label{subsec:convergence-rates}

The preceding results show that the Cournot and hierarchical
Stackelberg aggregate outputs have the same limiting value under a
fixed cost specification. We now compare the rates at which these
limits are approached.

Define the normalized output residuals by
\[
	\begin{aligned}
		R_n^{C,L}
		&=
		\frac{\widetilde M-Q_n^{C,L}}{\widetilde M},
		&
		R_n^{S,L}
		&=
		\frac{\widetilde M-Q_n^{S,L}}{\widetilde M},\\
		R_n^{C,Q}
		&=
		\frac{M-Q_n^{C,Q}}{M},
		&
		R_n^{S,Q}
		&=
		\frac{M-Q_n^{S,Q}}{M}.
	\end{aligned}
\]

The residuals measure the relative distance between the finite-market
aggregate output and the corresponding large-market limit. They also
determine the price gaps because the inverse demand function is
linear.

\begin{theorem}\label{thm:convergence-rate-comparison}
	Under the assumptions of
	Subsections~\ref{subsec:quadratic-aggregate-outputs} and
	\ref{subsec:linear-aggregate-outputs}, the normalized output residuals
	satisfy the following properties.
	
	\begin{enumerate}
		\item Under linear costs,
		\begin{equation}\label{eq:linear-residuals}
			R_n^{C,L}
			=
			\frac{1}{n+1},
			\qquad
			R_n^{S,L}
			=
			2^{-n}.
		\end{equation}
		Consequently,
		\begin{equation}\label{eq:linear-residual-ratio}
			\lim_{n\to\infty}
			\frac{R_n^{S,L}}{R_n^{C,L}}
			=
			\lim_{n\to\infty}
			\frac{n+1}{2^n}
			=
			0.
		\end{equation}
		
		\item Under quadratic costs,
		\begin{equation}\label{eq:quadratic-cournot-residual}
			R_n^{C,Q}
			=
			\frac{1+2\delta}{n+1+2\delta},
		\end{equation}
		and hence
		\begin{equation}\label{eq:quadratic-cournot-residual-asymptotic}
			R_n^{C,Q}
			\sim
			\frac{1+2\delta}{n}.
		\end{equation}
		The Stackelberg residual satisfies
		\(R_0^{S,Q}=1\)
		and
		\begin{equation}\label{eq:quadratic-stackelberg-residual-recursion}
			R_n^{S,Q}
			=
			\frac{
				R_{n-1}^{S,Q}
				\bigl(R_{n-1}^{S,Q}+2\delta\bigr)
			}{
				2\bigl(R_{n-1}^{S,Q}+\delta\bigr)
			},
			\qquad n\geq1.
		\end{equation}
		Moreover,
		\begin{equation}\label{eq:quadratic-stackelberg-residual-asymptotic}
			R_n^{S,Q}
			\sim
			\frac{2\delta}{n}.
		\end{equation}
		Consequently,
		\begin{equation}\label{eq:quadratic-residual-ratio}
			\lim_{n\to\infty}
			\frac{R_n^{S,Q}}{R_n^{C,Q}}
			=
			\frac{2\delta}{1+2\delta}
			<1.
		\end{equation}
	\end{enumerate}
\end{theorem}

\begin{proof}
	Under linear costs, equations
	\eqref{eq:linear-cournot-aggregate-output} and
	\eqref{eq:linear-stackelberg-aggregate-output} give
	\(\widetilde M-Q_n^{C,L}
	=
	\frac{\widetilde M}{n+1}\)
	and
	\(\widetilde M-Q_n^{S,L}
	=
	\widetilde M\,2^{-n}\).
	Dividing by \(\widetilde M\) yields
	\eqref{eq:linear-residuals}. 
	
	Since exponential growth dominates
	linear growth,
	\(\lim_{n\to\infty}\frac{n+1}{2^n}=0\),
	which proves \eqref{eq:linear-residual-ratio}.
	
	Under quadratic costs,
	\eqref{eq:quadratic-cournot-aggregate-output} gives
	\[
		R_n^{C,Q}=
		1-\frac{Q_n^{C,Q}}{M}=
		1-\frac{n}{n+1+2\delta}=
		\frac{1+2\delta}{n+1+2\delta}.
	\]
	
	This proves \eqref{eq:quadratic-cournot-residual} and
	\eqref{eq:quadratic-cournot-residual-asymptotic}.
	
	By \eqref{eq:quadratic-stackelberg-normalization}, it follows
	\(R_n^{S,Q}=1-q_n\).
	Using the recursion \eqref{eq:q-recursion}, we obtain
	\[
		R_n^{S,Q}
		=
		1-q_{n-1}
		-
		\frac{(1-q_{n-1})^2}
		{2(1-q_{n-1}+\delta)}=
		R_{n-1}^{S,Q}
		-
		\frac{(R_{n-1}^{S,Q})^2}
		{2(R_{n-1}^{S,Q}+\delta)}=
		\frac{
			R_{n-1}^{S,Q}
			\bigl(R_{n-1}^{S,Q}+2\delta\bigr)
		}{
			2\bigl(R_{n-1}^{S,Q}+\delta\bigr)
		}.
	\]
	
	This establishes
	\eqref{eq:quadratic-stackelberg-residual-recursion}.
	
	For brevity, write
	\(r_n=R_n^{S,Q}\).
	From \eqref{eq:quadratic-stackelberg-output-limit}, we know that
	\(r_n\to0\). The residual recursion gives
	\(
		\frac{1}{r_n}-\frac{1}{r_{n-1}}=
		\frac{1}{r_{n-1}+2\delta}\).
		
	Therefore,
	\(\lim_{n\to\infty}
	\left(
	\frac{1}{r_n}-\frac{1}{r_{n-1}}
	\right)
	=
	\frac{1}{2\delta}\).
	According to the Stolz--Ces\'{a}ro theorem, we get
	\(\lim_{n\to\infty}
	\frac{1/r_n}{n}
	=
	\frac{1}{2\delta}\), i.e.,
	\(\lim_{n\to\infty}nr_n=2\delta\),
	which proves
	\eqref{eq:quadratic-stackelberg-residual-asymptotic}.
	
	Finally,
	\(\frac{R_n^{S,Q}}{R_n^{C,Q}}
	\sim
	\frac{2\delta/n}{(1+2\delta)/n}
	=
	\frac{2\delta}{1+2\delta}\),
	and \eqref{eq:quadratic-residual-ratio} follows.
\end{proof}

The corresponding non-normalized output gaps are summarized in the
following corollary.

\begin{cor}\label{cor:absolute-output-gaps}
	Under linear costs,
	\(\widetilde M-Q_n^{C,L}
		=
		\frac{\widetilde M}{n+1}\) and \(\widetilde M-Q_n^{S,L}
		=
		\widetilde M\,2^{-n}\).
	
	Under quadratic costs,
	\(M-Q_n^{C,Q}
		\sim
		\frac{(1+2\delta)M}{n}\),
	whereas
	\(M-Q_n^{S,Q}
		\sim
		\frac{2\delta M}{n}\).
\end{cor}

Because
\(P_n^{X,Y}=A-BQ_n^{X,Y}\),
the output residuals also yield the price convergence rates.

\begin{cor}\label{cor:price-convergence-rates}
	Under linear costs,
	\(P_n^{C,L}-c
		=
		\frac{A-c}{n+1}\) and \(P_n^{S,L}-c
		=
		(A-c)2^{-n}\).
	
	Under quadratic costs,
	\(P_n^{C,Q}
		\sim
		\frac{(1+2\delta)A}{n}\), whereas
	\(P_n^{S,Q}
		\sim
		\frac{2\delta A}{n}\).
\end{cor}

\begin{proof}
	The conclusions follow by multiplying the corresponding output gaps
	by \(B\), using
	\(B\widetilde M=A-c\) and \(BM=A\).
\end{proof}

Theorem~\ref{thm:convergence-rate-comparison} distinguishes two
different effects of sequential commitment. Under linear costs, the
Stackelberg output gap decreases geometrically, whereas the Cournot
gap decreases algebraically at order \(1/n\). Thus, the Stackelberg
aggregate output approaches its limiting value substantially faster.

Under quadratic costs, both output gaps are of order \(1/n\).
Nevertheless, their leading constants differ. The Stackelberg
constant is \(2\delta\), whereas the Cournot constant is
\(1+2\delta\). Since
\(2\delta<1+2\delta\),
the two models have the same \(1/n\) convergence rate,
but the Stackelberg residual has a strictly smaller leading constant.

These results refine the established observation that both market
organizations approach a competitive large-market outcome. The
existence of the common limit is related to the earlier literature,
including \cite{Robson1990}. The contribution of the present
comparison lies in the exact residual formulas under linear costs and
in the explicit asymptotic constants derived from the quadratic-cost
recursion.

\subsection{Economic Interpretation}
\label{subsec:economic-interpretation}

The preceding results distinguish three aspects of the comparison
between Cournot and hierarchical Stackelberg competition: the
finite-market output level, the limiting market outcome, and the rate
at which that limit is approached. This subsection discusses their
economic implications.

\subsubsection{Aggregate Output and Market Price}

Propositions~\ref{prop:quadratic-output-comparison} and
\ref{prop:linear-output-comparison} show that, under both cost
specifications,
\begin{equation}\label{eq:general-output-comparison}
	Q_n^{S,Y}>Q_n^{C,Y},
	\qquad
	Y\in\{L,Q\},\quad n\geq2.
\end{equation}
Since the inverse demand function is strictly decreasing,
\eqref{eq:general-output-comparison} implies
\(P_n^{S,Y}<P_n^{C,Y}\) for \(Y\in\{L,Q\}\), \(n\geq2\).

Thus, for every finite market with at least two firms, sequential
commitment generates a larger aggregate quantity and a lower market
price than simultaneous Cournot competition. This comparison is
consistent with the established literature on sequential
quantity-setting games, particularly \cite{Anderson_Engers_1992}. 
It is included here as the economic
basis for interpreting the convergence-rate results.

The output ordering should not be interpreted as showing that every
Stackelberg firm produces more than a Cournot firm. Under sequential
play, the distribution of output is asymmetric. In the linear-cost
model, Corollary~\ref{cor:linear-individual-output-ordering} gives
\[
x_{n,n}^{S,L}
<
x_{i,n}^{C,L}
<
x_{1,n}^{S,L}.
\]
The initial leader produces more than a Cournot firm, whereas the
final follower produces less. The larger Stackelberg aggregate output
is therefore generated by the sequential distribution of production,
not by a uniform increase in every firm's output.

\subsubsection{Consumer Surplus}

For the linear inverse demand function
\(P(Q)=A-BQ\),
consumer surplus at aggregate output \(Q\) is
\begin{equation}\label{eq:consumer-surplus-definition}
		CS(Q)
		=
		\int_0^Q P(t)\,dt-P(Q)Q=
		\int_0^Q(A-Bt)\,dt-(A-BQ)Q=
		\frac{B}{2}Q^2.
\end{equation}

The following conclusion follows immediately from the aggregate-output
comparisons.

\begin{proposition}\label{prop:consumer-surplus-comparison}
	Under either linear or quadratic costs,
	\(CS\bigl(Q_n^{S,Y}\bigr)
		>
		CS\bigl(Q_n^{C,Y}\bigr)\) for \(Y\in\{L,Q\}\) and \(n\geq2\).
\end{proposition}

\begin{proof}
	By \eqref{eq:consumer-surplus-definition}, consumer surplus is
	strictly increasing in \(Q\) for \(Q>0\). The conclusion follows from
	\eqref{eq:general-output-comparison}.
\end{proof}

Thus, consumers benefit from the larger aggregate quantity and lower
price generated by sequential competition. This conclusion is
independent of how total production is distributed among the firms.

\subsubsection{Total Surplus under Linear Costs}

Total surplus is not the same as consumer surplus. It is defined as
the gross benefit from consumption minus total production costs:
\[
TS(x_1,\ldots,x_n)
	=
	\int_0^Q P(t)\,dt
	-
	\sum_{i=1}^{n}C_i(x_i),
\]
where 	\(Q=\sum_{i=1}^{n}x_i\). 
Equivalently,
\(TS
=
CS+\sum_{i=1}^{n}\Pi_i\).

Under linear costs, total production cost depends only on aggregate
output:
\[\sum_{i=1}^{n}C_i^L(x_i)
=
cQ.
\]

Consequently, total surplus can be written solely as a function of
\(Q\):
\begin{equation}\label{eq:linear-total-surplus}
		TS_L(Q)=
		\int_0^Q(A-Bt)\,dt-cQ=
		(A-c)Q-\frac{B}{2}Q^2.
\end{equation}

The function \(TS_L\) is strictly increasing on
\([0,\widetilde M)\),
because
\begin{equation}\label{eq:linear-total-surplus-derivative}
	TS_L'(Q)
	=
	A-c-BQ
	>
	0
\end{equation}
on this interval.

\begin{proposition}\label{prop:linear-total-surplus-comparison}
	Under linear costs,
	\begin{equation}\label{eq:linear-total-surplus-ordering}
		TS_L\bigl(Q_n^{S,L}\bigr)
		>
		TS_L\bigl(Q_n^{C,L}\bigr),
		\qquad n\geq2.
	\end{equation}
	
	Moreover,
	\[
		\lim_{n\to\infty}
		TS_L\bigl(Q_n^{C,L}\bigr)
		=
		\lim_{n\to\infty}
		TS_L\bigl(Q_n^{S,L}\bigr)
		=
		\frac{(A-c)^2}{2B}.
	\]
\end{proposition}

\begin{proof}
	For every finite \(n\),
	\(0<Q_n^{C,L}<Q_n^{S,L}<\widetilde M\).
	The strict inequality
	\eqref{eq:linear-total-surplus-ordering} therefore follows from
	\eqref{eq:linear-total-surplus-derivative}. The common limit follows
	from Proposition~\ref{prop:linear-output-limits} and the continuity of
	\(TS_L\).
\end{proof}

Hence, under linear costs, the larger Stackelberg aggregate output
leads not only to higher consumer surplus but also to higher total
surplus. The difference disappears in the large-market limit because
both market organizations approach the same output
\(\widetilde M\).

\subsubsection{Total Surplus under Quadratic Costs}

Under quadratic costs, total production cost is
\(\sum_{i=1}^{n}C_i^Q(x_i)
=
c\sum_{i=1}^{n}x_i^2\).
Total surplus is therefore
\begin{equation}\label{eq:quadratic-total-surplus}
	TS_Q(x_1,\ldots,x_n)
	=
	AQ-\frac{B}{2}Q^2
	-
	c\sum_{i=1}^{n}x_i^2.
\end{equation}

Unlike \eqref{eq:linear-total-surplus}, expression
\eqref{eq:quadratic-total-surplus} is not determined by aggregate
output alone. It also depends on the distribution of production
among the firms. In particular, for a fixed aggregate quantity \(Q\),
the quadratic production cost is lower when output is distributed
more evenly:
\[
\sum_{i=1}^{n}x_i^2
\geq
\frac{Q^2}{n},
\]
with equality if and only if
\(x_1=\cdots=x_n=\frac{Q}{n}\).

Cournot production is symmetric, whereas hierarchical Stackelberg
production is generally asymmetric. Consequently, the comparison
\(Q_n^{S,Q}>Q_n^{C,Q}\)
is sufficient to establish higher consumer surplus in the
Stackelberg model, but it is not by itself sufficient to establish a
general ordering of total surplus. The larger Stackelberg quantity
increases gross consumer benefit, while the asymmetric allocation of
production may increase aggregate quadratic cost.

Accordingly, we do not claim a general total-surplus ordering under
quadratic costs solely from the aggregate-output comparison. Such an
ordering requires either an explicit comparison of
\(c\sum_{i=1}^{n}\bigl(x_{i,n}^{S,Q}\bigr)^2\) and \(c\sum_{i=1}^{n}\bigl(x_{i,n}^{C,Q}\bigr)^2\)
or additional restrictions on the parameters and the distribution of
Stackelberg production.

\subsubsection{Large-Market Interpretation}

Under linear costs, both market organizations satisfy
\[
\lim_{n\to\infty}Q_n^{C,L}=\lim_{n\to\infty}Q_n^{S,L}
=\widetilde M=\frac{A-c}{B}
\]
and
\[
\lim_{n\to\infty}P_n^{C,L}=\lim_{n\to\infty}P_n^{S,L}=c.
\]

Thus, the market price approaches the common marginal cost. The
limiting outcome has the standard competitive interpretation.

Under quadratic costs,
\[
\lim_{n\to\infty}Q_n^{C,Q}=\lim_{n\to\infty}Q_n^{S,Q}=
M=\frac{A}{B}
\]
and
\[
\lim_{n\to\infty}P_n^{C,Q}=\lim_{n\to\infty}P_n^{S,Q}=0.
\]

As the number of firms increases, individual quantities become small,
and their marginal quadratic costs approach zero. The limiting price
therefore approaches the corresponding competitive benchmark.

The common limits are consistent with the earlier literature on
large sequential markets, including \cite{Robson1990}. The
additional information provided by
Theorem~\ref{thm:convergence-rate-comparison} concerns how rapidly
the two market organizations approach these limits. Under linear
costs, the Stackelberg residual decreases geometrically, while the
Cournot residual is of order \(1/n\). Under quadratic costs, both
residuals are of order \(1/n\), but the Stackelberg residual has the
smaller leading constant.

The economic implication is that the effect of sequential commitment
does not disappear merely because the two models have the same
large-market limit. For every finite \(n\), the hierarchical
Stackelberg model produces more output and a lower price, and the
difference remains visible in the rate at which the competitive
benchmark is approached.

\subsection{Numerical Illustration}
\label{subsec:numerical-illustration}

We conclude the large-market comparison with a numerical illustration
of the general results established above. To keep the calculations
transparent, consider the parameter values
\(A=12\), \(B=1\), and \(c=1\).
Thus, the inverse demand function is
\(P(Q)=12-Q\).

The limiting aggregate outputs are
\(\widetilde M
	=
	\frac{A-c}{B}
	=
	11\)
under linear costs and
\(M
	=
	\frac{A}{B}
	=
	12\)
under quadratic costs. Moreover,
\(\delta=\frac{c}{B}=1\).

The aggregate Cournot and Stackelberg outputs under linear costs are
obtained from
\eqref{eq:linear-cournot-aggregate-output} and
\eqref{eq:linear-stackelberg-aggregate-output}:
\(Q_n^{C,L}
	=
	11\frac{n}{n+1}\) and \(Q_n^{S,L}
	=
	11\left(1-2^{-n}\right)\).

Under quadratic costs,
\eqref{eq:quadratic-cournot-aggregate-output} gives
\(Q_n^{C,Q}
	=
	12\frac{n}{n+3}\).
The aggregate Stackelberg output is
\(Q_n^{S,Q}
	=
	12q_n\),
where
\[
	q_0=0,
	\qquad
	q_n
	=
	q_{n-1}
	+
	\frac{(1-q_{n-1})^2}
	{2(2-q_{n-1})}.
\]

The aggregate equilibrium outputs under Cournot and hierarchical
Stackelberg competition, for both linear and quadratic costs and for
selected values of \(n\), are reported in
Table~\ref{tab:numerical-output-comparison}.

\begin{table}[ht]
	\centering
	\caption{Aggregate equilibrium outputs for
		\(A=12\), \(B=1\), and \(c=1\).}
	\label{tab:numerical-output-comparison}
	\begin{tabular}{c c c c c}
		\toprule
		\(n\) &
		\(Q_n^{C,L}\) &
		\(Q_n^{S,L}\) &
		\(Q_n^{C,Q}\) &
		\(Q_n^{S,Q}\)\\
		\midrule
		1  & 5.500 & 5.500  & 3.000 & 3.000\\
		2  & 7.333 & 8.250  & 4.800 & 4.929\\
		3  & 8.250 & 9.625  & 6.000 & 6.240\\
		4  & 8.800 & 10.313 & 6.857 & 7.174\\
		5  & 9.167 & 10.656 & 7.500 & 7.866\\
		6  & 9.429 & 10.828 & 8.000 & 8.396\\
		8  & 9.778 & 10.957 & 8.727 & 9.146\\
		10 & 10.000 & 10.989 & 9.231 & 9.649\\
		\bottomrule
	\end{tabular}
\end{table}

\begin{figure}[!htbp]
	\centering
	\includegraphics[width=0.82\textwidth]{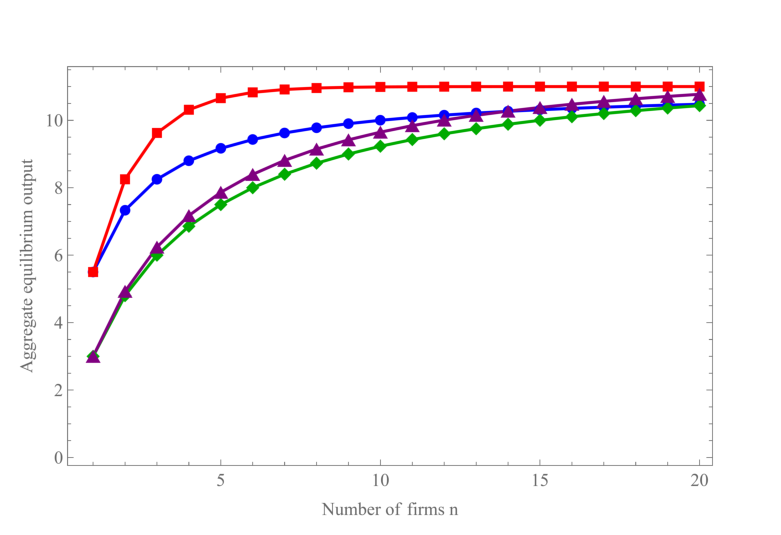}
	\caption{Aggregate equilibrium outputs as functions of the number of firms \(n\). Blue: Cournot with linear costs. Red: hierarchical Stackelberg with linear costs. Green: Cournot with quadratic costs. Purple: hierarchical Stackelberg with quadratic costs.}
	\label{fig:numerical-output-comparison}
\end{figure}

The numerical values (Figure \ref{fig:numerical-output-comparison}) illustrate the general finite-market
comparisons:
\(Q_n^{S,L}>Q_n^{C,L}\)
and
\(Q_n^{S,Q}>Q_n^{C,Q}\)
for every displayed \(n\geq2\). Since \(P(Q)=12-Q\), the corresponding
Stackelberg prices are lower than the Cournot prices.

The different convergence patterns are more clearly visible from the
output gaps. Under linear costs,
\(\widetilde M-Q_n^{C,L}
	=
	\frac{11}{n+1}\) and \(\widetilde M-Q_n^{S,L}
	=
	11\,2^{-n}\).

Under quadratic costs,
\(M-Q_n^{C,Q}
	=
	\frac{36}{n+3}\),
whereas the Stackelberg gap is calculated recursively from
\eqref{eq:q-recursion}  with \(\delta=1\) (Figure \eqref{fig:numerical-output-gaps}).

The output gaps reported in
Table~\ref{tab:numerical-output-gaps} illustrate the different rates
at which the Cournot and hierarchical Stackelberg outputs approach
their respective large-market limits under linear and quadratic
costs.

\begin{table}[ht]
	\centering
	\caption{Distances from the corresponding large-market output.}
	\label{tab:numerical-output-gaps}
	\begin{tabular}{c c c c c}
		\toprule
		\(n\) &
		\(\widetilde M-Q_n^{C,L}\) &
		\(\widetilde M-Q_n^{S,L}\) &
		\(M-Q_n^{C,Q}\) &
		\(M-Q_n^{S,Q}\)\\
		\midrule
		1  & 5.500 & 5.500 & 9.000 & 9.000\\
		2  & 3.667 & 2.750 & 7.200 & 7.071\\
		3  & 2.750 & 1.375 & 6.000 & 5.760\\
		4  & 2.200 & 0.688 & 5.143 & 4.826\\
		5  & 1.833 & 0.344 & 4.500 & 4.134\\
		6  & 1.571 & 0.172 & 4.000 & 3.604\\
		8  & 1.222 & 0.043 & 3.273 & 2.854\\
		10 & 1.000 & 0.011 & 2.769 & 2.351\\
		\bottomrule
	\end{tabular}
\end{table}

\begin{figure}[!htbp]
	\centering
	\includegraphics[width=0.82\textwidth]{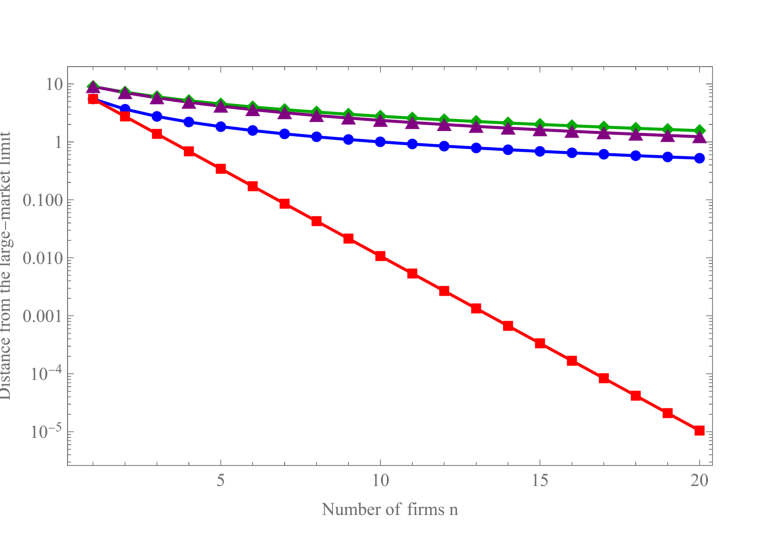}
	\caption{Distances from the corresponding large-market output. Blue: Cournot with linear costs. Red: hierarchical Stackelberg with linear costs. Green: Cournot with quadratic costs. Purple: hierarchical Stackelberg with quadratic costs. The vertical axis is logarithmic.}
	\label{fig:numerical-output-gaps}
\end{figure}

Under linear costs, the difference is pronounced. At \(n=10\), the
Cournot output remains one unit below its limit, whereas the
Stackelberg output is only approximately \(0.011\) below the same
limit. This reflects the geometric Stackelberg rate \(2^{-n}\), in
contrast to the algebraic Cournot rate \(1/(n+1)\).

Under quadratic costs, both gaps decrease at order \(1/n\), so their
difference is less pronounced. Nevertheless, the Stackelberg gap is
smaller for every displayed \(n\geq2\). For the selected value
\(\delta=1\), Theorem~\ref{thm:convergence-rate-comparison} gives
\(R_n^{C,Q}
\sim
\frac{3}{n}\), \(R_n^{S,Q}
\sim
\frac{2}{n}\), 
and therefore
\(\lim_{n\to\infty}
\frac{R_n^{S,Q}}{R_n^{C,Q}}
=
\frac{2}{3}\).

Consumer surplus is obtained from
\eqref{eq:consumer-surplus-definition}. Since \(B=1\),
\(CS(Q)=\frac12Q^2\).
Thus, the larger Stackelberg output generates greater consumer
surplus under both cost specifications for every \(n\geq2\).

For example, when \(n=3\), the linear-cost models give
\(CS\bigl(Q_3^{C,L}\bigr)
=
\frac12(8.25)^2
=
34.031\),
whereas
\(CS\bigl(Q_3^{S,L}\bigr)
=
\frac12(9.625)^2
=
46.320\).
Under quadratic costs,
\(CS\bigl(Q_3^{C,Q}\bigr)
=
\frac12(6)^2
=
18\),
whereas
\(CS\bigl(Q_3^{S,Q}\bigr)
\approx
\frac12(6.240)^2
\approx
19.466\).

These calculations concern consumer surplus, not total surplus.
Under linear costs, Proposition~
\ref{prop:linear-total-surplus-comparison} also guarantees that total
surplus is higher in the Stackelberg model. Under quadratic costs,
however, total surplus additionally depends on the distribution of
production through
\(c\sum_{i=1}^{n}x_i^2\).
Accordingly, the aggregate-output values in
Table~\ref{tab:numerical-output-comparison} are not by themselves
sufficient to compare total surplus under quadratic costs.

The numerical illustration confirms the theoretical conclusions
without introducing additional general claims. Sequential commitment
raises aggregate output and lowers price for every finite
\(n\geq2\). Under linear costs, it also changes the order of
convergence from algebraic to geometric. Under quadratic costs, both
models converge algebraically, but the Stackelberg model has the
smaller leading residual constant.

}

\section*{Author Contributions}

All authors contributed equally and substantially to
writing this article and are listed in alphabetical order. All authors read
and approved the final manuscript.

\section*{Funding}

Funding. This research was partially supported by the
Bulgarian National Science Fund (BNSF), Grant number KP-06-N92/1.

\section*{Data Availability}

Data sharing is not applicable
because no datasets were generated or analyzed during the study.

\section*{Declarations}

\subsection*{Ethics Approval}

The authors declare that they accept the principles of
ethical and professional conduct required by the journal.

\subsection*{Competing Interests}

The authors declare no competing interests.


\appendix
\section{Background on Tupled Fixed Points}

\subsection{Fixed Points for Ordered Triples of Maps}

The Banach fixed point concept \cite{Banach}, despite being over 100 years old, has numerous extensions and implementations. The generalizations can be classified in different ways. We are only interested in the recommended generalization \cite{GL}, which alters the concept of fixed points. Coupled fixed points, introduced in \cite{GL}, are significant in this context.

For a map of two variables $F:A\times A\to A$ (\cite{BL,GL}) a coupled fixed point $(x,y)\in A\times A$ of $F$ in $A$ is defined as $x=F(x,y)$ and $y=F(y,x)$.

The study of the existence and uniqueness of fixed
points for cyclic maps, i.e., \(T:A\to B\) and \(T:B\to A\), rather
than self-maps, was initiated by \cite{KSV}. 

Combining cyclic maps with coupled fixed points \cite{SK} broadened the concept to coupled fixed
points for cyclic maps of two variables by investigating maps $F:A\times A\to B$, $F:B\times B\to A$, and the search of sufficient conditions that will ensure the existence of an ordered pair $(x,y)$ so that $x=F(x,y)$ and $y=F(y,x)$.

However, the equations $x=F(x,y)$ and $y=F(y,x)$ that define the linked fixed points \cite{BL,GL} frequently result in an ordered pair $(x,y)$ such that $x=y$. This limitation was addressed by \cite{Zlatanov-FPT}, who
developed a modified coupled fixed-point concept for an ordered pair
of maps $(F,G)$. \cite{Zlatanov-FPT} suggests combining two maps,
\(F:A\times B\to A\) and \(G:A\times B\to B\), and defining an ordered
pair \((x,y)\) as a coupled fixed point for $(F,G)$ if $x=F(x,y)$ and $y=G(x,y)$. The definition from \cite{BL,GL} is obtained when $G(x,y)=F(y,x)$, and $A=B$.

Building on \cite{GL}, \cite{Borcut-Berinde} proposed seeking an ordered triple $(x,y,z)$ fulfilling $x=F(x,y,z)$, $y=F(y,x,y)$, and $z=F(z,y,x)$, . In \cite{IIKYZ}, generalizations based on the aforementioned have been presented, i.e. $x=F(x,y,z)$, $y=G(x,y,z)$, and $z=H(x,y,z)$ for $F,G,H:A^3\to A$. The concept of a tripled fixed point extends to tripled points for ordered triples of maps $(F,G,H)$ that meet the contractive type condition. We use a simplified version of the main results of \cite{IIKYZ}.

\subsection{Hardy-Rogers Fixed Point Theorem}

\begin{definition}(\cite{Hardy-Rogers,Reich})
	Let $(X, \rho)$ be a metric space. A map $T : X \to X$ is called a Hardy-Rogers map if there exist non-negative constants $a_i$, $i = 1, 2, 3, 4, 5$, such that $\sum_{i=1}^5 a_i < 1$ and for all $x, y \in X$ the following inequality holds:
	\begin{equation}\label{eq-1}
		\rho(Tx, Ty) \leq a_1 \rho(x, y) + a_2 \rho(x, Tx) + a_3 \rho(y, Ty) + a_4 \rho(x, Ty) + a_5 \rho(y, Tx).
	\end{equation}
	
	As pointed out in \cite{Hardy-Rogers}, by the symmetry of the metric $\rho(\cdot, \cdot)$, inequality(\ref{eq-1}) implies:
	$$
	\rho(Tx, Ty) \leq a_1 \rho(x, y) + \frac{a_2 + a_3}{2} \left( \rho(x, Tx) + \rho(y, Ty) \right) + \frac{a_4 + a_5}{2} \left( \rho(x, Ty) + \rho(y, Tx) \right).
	$$
	
	Therefore, without loss of generality, we may consider Hardy-Rogers maps to satisfy the simplified inequality:
	\begin{equation}\label{eq-2}
		\rho(Tx, Ty) \leq k_1 \rho(x, y) + k_2 \left( \rho(x, Tx) + \rho(y, Ty) \right) + k_3 \left( \rho(x, Ty) + \rho(y, Tx) \right),
	\end{equation}
	with $k_1 + 2k_2 + 2k_3 < 1$.
\end{definition}

Special cases include:
\begin{itemize}
	\item $k_2 = k_3 = 0$: Banach contraction~\cite{Banach},
	\item $k_1 = k_3 = 0$: Kannan contraction~\cite{Kannan},
	\item $k_1 = k_2 = 0$: Chatterjea contraction~\cite{Chatterjea}.
\end{itemize}

In the rest of this work, we will assume Hardy-Rogers maps satisfy inequality (\ref{eq-2}).

\begin{theorem}\label{th:1}(\cite{Hardy-Rogers})
	Let \((X,\rho)\) be a complete metric space, and let
	\(T:X\to X\) be a Hardy–Rogers map. Then:
	\begin{enumerate}
		\item there is a unique fixed point $\xi \in X$ of $T$ and, moreover, for any initial guess $x_0 \in X$, the iterated sequence $x_n = T x_{n-1}$ for $n = 1, 2, \ldots$ converges to the fixed point $\xi$;
		\item there holds a priori error estimate:
		$\rho(\xi, x_n) \leq \frac{k^n}{1 - k} \rho(x_0, x_1)$;
		\item there holds a posteriori error estimate: $\rho(\xi, x_n) \leq \frac{k}{1 - k} \rho(x_{n-1}, x_n)$;
		\item the rate of convergence is: $\rho(\xi, x_n) \leq k \rho(\xi, x_{n-1})$,
	\end{enumerate}
	where $k = \frac{k_1 + k_2 + k_3}{1 - k_2 - k_3}$ and $k_i$, $i = 1, 2, 3$ are the constants from (\ref{eq-2}).
\end{theorem}

\subsection{$N$-tupled Fixed Points}

Following a sequence of articles on coupled fixed
points, beginning with \cite{BL,GL}, and tripled fixed points, beginning with \cite{Borcut-Berinde}, these ideas were generalized to \(N\)-tuples of fixed points \cite{Samet-Vetro}. TThe proposed notions of coupled, tripled, and
\(N\)-tupled fixed points involved only one map, which led to a
symmetric system of defining equations and, consequently, to fixed
points $(\xi_1,\xi_2,\dots,\xi_n)$, $n\geq 2$, satisfying $\xi_i=\xi_j$. This drawback was eliminated by \cite{Zlatanov-FPT} for
coupled fixed points, by \cite{IIKYZ} for tripled fixed
points, and by \cite{Ali-Hristov} for \(N\)-tupled fixed points.

We will recall necessary definitions from \cite{Zlatanov-FPT,IIKYZ,Ali-Hristov}.

\begin{definition}\label{def:1}(\cite{Ali-Hristov})
	Let $I_i$, $i=1,2,\dots, n$ be nonempty sets. We will call the ordered $n$-tuple of maps $(F_1,F_2,\dots,F_n)$ a semi-cyclic map if $F_i:\prod_{j=1}^nI_j\to I_i$ for $i=1,2,\dots, n$ .
\end{definition}

\begin{definition}\label{def:2}(\cite{Ali-Hristov})
	Let $I_i$, $i=1,2,\dots, n$ be nonempty sets and the ordered $n$-tuple $(F_1,F_2,\dots,F_n)$ be a semi-cyclic map. 
	An ordered $n$-tuple $(\xi_1,\xi_2,\dots,\xi_n)\in \prod_{j=1}^nI_j$ is said to be an \(n\)-tupled fixed point of $(F_1,F_2,\dots,F_n)$ if $\xi_i=F_i(\xi_1,\xi_2,\dots,\xi_n)$ for $i=1,2,\dots, n$.
\end{definition}

If $n=2$ we get the definition for coupled fixed points from \cite{Zlatanov-FPT}. If moreover
$F_2(x,y)=F_1(y,x)$ we get the definition for coupled fixed points from \cite{GL,BL}.
If $n=3$ we get the definition for tripled fixed points from \cite{IIKYZ}. If
moreover $F_2(x,y,z)=F_1(y,x,y)$ and $F_3(x,y,z)=F_1(z,x,y)$ we get the definition for tripled fixed points from \cite{Borcut-Berinde}. Let us denote by $\pi_i$ the permutation 
$\pi_i(1,2,\dots,n)=(i,i+1,\dots,n,1,2,\dots,i-1)$. Then if $F_i(x_1,x_2,\dots, x_n)=F_1(\pi_i(x_1,x_2,\dots, x_n))$ we get the definition for $n$-tupled fixed point from \cite{Samet-Vetro}.

\begin{definition}\label{def:3}(\cite{Ali-Hristov})
	Let $I_i$, $i=1,2,\dots, n$ be nonempty sets and the ordered $n$-tuple $(F_1,F_2,\dots,F_n)$ be a semi-cyclic map.  
	For any $n$-tuple $\left(x^{(0)}_1,x^{(0)}_2,\dots,x^{(0)}_n\right)\in \prod_{j=1}^nI_j$ we define
	$x^{(k+1)}_i=F_i\left(x^{(k)}_1,x^{(k)}_2,\dots,x^{(k)}_n\right)$, for $i=1,2,\dots,n$ for all $k\geq 0$.
\end{definition}

A profound observation in \cite{Petrusel,Petrusel18b} connects fixed point results with coupled fixed point ones. Indeed, instead of considering the ordered pair of maps $(F_1,F_2)$, we can define the map $T:I_1\times I_2\to I_1\times I_2$ by $T(x,y)=(F_1(x,y),F_2(x,y))$, and $(\xi,\eta)$ 
will be a coupled fixed point for $(F_1,F_2)$ if and only if it is a fixed point for $T$. Using this notation. The same notation can be used and for $n$-tupled fixed points, by identifying them
as fixed points for the map $T$, defined by 
$$
T(\overline{x})=(F_1(\overline{x}),F_2(\overline{x}),\dots,F_n(\overline{x})),
$$
where $\overline{x}=(x_1,x_2,\dots,x_n)$.

{\section{Proofs for the Large-Market Results}
\label{app:large-market-proofs}

\subsection{Proofs for Aggregate Outputs under Quadratic Costs}
\label{app:quadratic-outputs}

This subsection provides the proofs of
Proposition~\ref{prop:quadratic-cournot-output},
Theorem~\ref{thm:quadratic-stackelberg-recursion}, and
Proposition~\ref{prop:quadratic-output-comparison}.

\begin{proof}[Proof of Proposition~\ref{prop:quadratic-cournot-output}]
	For Firm \(i\), the payoff function is
	\[
	\Pi_i(x_1,\ldots,x_n)
	=
	x_i\left(A-B\sum_{j=1}^n x_j\right)-cx_i^2.
	\]
	
	The first-order condition with respect to \(x_i\) is
	\(A-B\sum_{j\ne i}x_j-2(B+c)x_i=0\).
	At a symmetric Cournot equilibrium,
	\(x_{1,n}^{C,Q}
	=
	\cdots
	=
	x_{n,n}^{C,Q}
	=
	x_n^{C,Q}\).
	
	Consequently,
	\(A-B(n-1)x_n^{C,Q}-2(B+c)x_n^{C,Q}=0\)
	and hence
	\[
	x_n^{C,Q}
	=
	\frac{A}{B(n+1)+2c}
	=
	\frac{M}{n+1+2\delta}.
	\]
	
	Since
	\(\frac{\partial^2\Pi_i}{\partial x_i^2}
	=
	-2(B+c)<0\),
	the first-order condition determines the unique best response of each
	firm. Therefore,
	\[
	x_{i,n}^{C,Q}
	=
	\frac{A}{B(n+1)+2c}
	=
	\frac{M}{n+1+2\delta},
	\qquad i=1,\ldots,n.
	\]
	
	Summing the individual equilibrium outputs gives
	\[
	Q_n^{C,Q}
	=
	\sum_{i=1}^n x_{i,n}^{C,Q}
	=
	\frac{An}{B(n+1)+2c}
	=
	M\frac{n}{n+1+2\delta}.
	\]
	Finally,
	\[
		P_n^{C,Q}=
		A-BQ_n^{C,Q}=
		A-
		A\frac{n}{n+1+2\delta}=
		A\frac{1+2\delta}{n+1+2\delta}.
	\]
	
	This proves
	\eqref{eq:quadratic-cournot-individual-output}--%
	\eqref{eq:quadratic-cournot-price}.
\end{proof}

\begin{proof}[Proof of Theorem~\ref{thm:quadratic-stackelberg-recursion}]
	We use backward induction. Suppose that the firms preceding a subgame
	have already chosen an aggregate quantity \(S\), and let
	\(R=M-S\)
	denote the corresponding residual market capacity.
	
	For \(m\geq0\), let \(L_m(S)\) denote the aggregate output produced
	by a continuation subgame containing \(m\) firms. We prove that
	\begin{equation}\label{eq:app-quadratic-continuation-output}
		L_m(S)=\lambda_m(M-S),
	\end{equation}
	where
	\(\lambda_0=0\)
	and
	\begin{equation}\label{eq:app-quadratic-lambda-recursion}
		\lambda_m
		=
		\lambda_{m-1}
		+
		\frac{(1-\lambda_{m-1})^2}
		{2(1-\lambda_{m-1}+\delta)},
		\qquad m\geq1.
	\end{equation}
	
	For \(m=0\), no firm remains in the continuation subgame, so
	\(L_0(S)=0\). Thus,
	\eqref{eq:app-quadratic-continuation-output} holds with
	\(\lambda_0=0\).
	
	Assume that
	\eqref{eq:app-quadratic-continuation-output} holds for a continuation
	subgame containing \(m-1\) firms. Let the first firm in a continuation
	subgame containing \(m\) firms choose the quantity \(x\). The
	remaining \(m-1\) firms then produce
	\[
	L_{m-1}(S+x)
	=
	\lambda_{m-1}(M-S-x)
	=
	\lambda_{m-1}(R-x).
	\]
	
	Therefore, the reduced payoff of the first firm in this continuation
	subgame is
	\[
	\begin{aligned}
		\widetilde\Pi_m(x;S)
		&=
		x\left[
		A-B\left(
		S+x+\lambda_{m-1}(R-x)
		\right)
		\right]-cx^2\\
		&=
		B(1-\lambda_{m-1})x(R-x)-cx^2.
	\end{aligned}
	\]
	
	Differentiating with respect to \(x\), we obtain
	\[
	\frac{\partial\widetilde\Pi_m}{\partial x}(x;S)
	=
	B(1-\lambda_{m-1})R
	-
	2\left[
	B(1-\lambda_{m-1})+c
	\right]x.
	\]
	
	Moreover,
	\(\frac{\partial^2\widetilde\Pi_m}{\partial x^2}(x;S)
	=
	-2\left[
	B(1-\lambda_{m-1})+c
	\right]<0\).
	Thus, the first-order condition determines the unique maximizing
	quantity:
	\begin{equation}\label{eq:app-quadratic-current-output}
		x
		=
		\mu_m R,
		\qquad
		\mu_m
		=
		\frac{1-\lambda_{m-1}}
		{2(1-\lambda_{m-1}+\delta)}.
	\end{equation}
	
	The aggregate output produced by all \(m\) firms in the continuation
	subgame is therefore
	\[
	\begin{aligned}
		L_m(S)
		&=
		x+L_{m-1}(S+x)=
		\mu_mR+\lambda_{m-1}(R-\mu_mR)\\
		&=
		\left[
		\lambda_{m-1}
		+
		(1-\lambda_{m-1})\mu_m
		\right]R.
	\end{aligned}
	\]
	
	Hence
	\(\lambda_m
	=
	\lambda_{m-1}
	+
	(1-\lambda_{m-1})\mu_m\).
	Substituting \eqref{eq:app-quadratic-current-output}, we obtain
	\(\lambda_m
	=
	\lambda_{m-1}
	+
	\frac{(1-\lambda_{m-1})^2}
	{2(1-\lambda_{m-1}+\delta)}\).
	This proves \eqref{eq:app-quadratic-lambda-recursion}.
	
	For the entire hierarchy, no output has been chosen before Firm~1
	moves. Therefore, \(S=0\), and
	\(Q_n^{S,Q}
	=
	L_n(0)
	=
	M\lambda_n\).
	
	It follows from
	\eqref{eq:quadratic-stackelberg-normalization} that
	\(q_n=\lambda_n\).
	
	Consequently,
	\(q_0=0\)
	and
	\[q_n
	=
	q_{n-1}
	+
	\frac{(1-q_{n-1})^2}
	{2(1-q_{n-1}+\delta)},
	\qquad n\geq1.
	\]
	This proves \eqref{eq:q-initial-value},
	\eqref{eq:q-recursion}, and
	\eqref{eq:quadratic-stackelberg-aggregate-output}.
	
	Setting \(n=1\) in \eqref{eq:q-recursion} gives
	\(q_1
	=
	\frac{1}{2(1+\delta)}\).
	Therefore,
	\[
	Q_1^{S,Q}
	=
	Mq_1
	=
	\frac{A/B}{2(1+c/B)}
	=
	\frac{A}{2(B+c)}.
	\]
	
	This proves \eqref{eq:q-first-value} and
	\eqref{eq:quadratic-stackelberg-one-firm}.
\end{proof}

\begin{proof}[Proof of Proposition~\ref{prop:quadratic-output-comparison}]
	Define the normalized Stackelberg residual by
	\(\rho_n=1-q_n\) for \(n\geq0\).
	Since \(q_0=0\), we have \(\rho_0=1\). From
	\eqref{eq:q-recursion}, it follows that
	\[
		\rho_n
		=
		\rho_{n-1}
		-
		\frac{\rho_{n-1}^2}
		{2(\rho_{n-1}+\delta)}=
		\frac{
			\rho_{n-1}(\rho_{n-1}+2\delta)
		}{
			2(\rho_{n-1}+\delta)
		}.
	\]
	
	Thus,
	\begin{equation}\label{eq:app-quadratic-residual-recursion}
		\rho_n=h(\rho_{n-1}),
		\qquad
		h(t)=\frac{t(t+2\delta)}{2(t+\delta)}.
	\end{equation}
	
	The normalized Cournot residual is
	\[
		r_n=
		1-\frac{Q_n^{C,Q}}{M}=
		1-\frac{n}{n+1+2\delta}=
		\frac{1+2\delta}{n+1+2\delta}.
	\]
	
	We prove that
	\(\rho_n\leq r_n\) for \(n\geq1\).
	
	For \(n=1\), there holds
	\(\rho_1
	=
	1-\frac{1}{2(1+\delta)}
	=
	\frac{1+2\delta}{2(1+\delta)}
	=
	r_1\). 	
	The function \(h\) in
	\eqref{eq:app-quadratic-residual-recursion} is strictly increasing on
	\((0,\infty)\), because
	\(h'(t)
	=
	\frac{t^2+2\delta t+2\delta^2}
	{2(t+\delta)^2}>0\).
	
	Assume that
	\(\rho_{n-1}
	\leq
	r_{n-1}
	=
	\frac{1+2\delta}{n+2\delta}\).
	Using the monotonicity of \(h\), we obtain
	\[
	\rho_n
	=
	h(\rho_{n-1})
	\leq
	h(r_{n-1}).
	\]
	
	A direct calculation gives
	\[
	h(r_{n-1})
	=
	\frac{
		(1+2\delta)
		\bigl(1+2\delta+2\delta(n+2\delta)\bigr)
	}{
		2(n+2\delta)
		\bigl(1+2\delta+\delta(n+2\delta)\bigr)
	}.
	\]
	
	Moreover,
	\(h(r_{n-1})
	\leq
	\frac{1+2\delta}{n+1+2\delta}
	=
	r_n\),
	because the last inequality is equivalent to
	\(n-1\geq0\).
	Therefore, induction yields
	\(\rho_n\leq r_n\) for every \(n\geq1\).
	
	Since
	\(q_n=1-\rho_n\) and \(\frac{Q_n^{C,Q}}{M}=1-r_n\),
	we conclude that
	\(q_n
	\geq
	\frac{Q_n^{C,Q}}{M}\).
	
	Multiplying by \(M>0\), we obtain
	\(Q_n^{S,Q}\geq Q_n^{C,Q}\) for \(n\geq1\).
	
	Equality holds for \(n=1\). For \(n\geq2\), the inequality
	\(h(r_{n-1})<r_n\)
	is strict because it is equivalent to \(n-1>0\). Hence
	\(\rho_n<r_n\)
	and therefore
	\(Q_n^{S,Q}>Q_n^{C,Q}\) for \(n\geq2\).
	
	Finally, since \(P(Q)=A-BQ\) is strictly decreasing and \(B>0\),
	the inequality \(Q_n^{S,Q}>Q_n^{C,Q}\)
	implies
	\(P_n^{S,Q}<P_n^{C,Q}\) for \(n\geq2\).
	
	This completes the proof.
\end{proof}

\subsection{Proofs for Aggregate Outputs under Linear Costs}
\label{app:linear-outputs}

This subsection provides the proofs of
Proposition~\ref{prop:linear-cournot-output} and
Proposition~\ref{prop:linear-stackelberg-output}.

\begin{proof}[Proof of Proposition~\ref{prop:linear-cournot-output}]
	For Firm \(i\), the payoff function is
	\[
	\Pi_i(x_1,\ldots,x_n)
	=
	x_i\left(A-B\sum_{j=1}^{n}x_j\right)-cx_i.
	\]
	
	The first-order condition with respect to \(x_i\) is
	\(A-c-B\sum_{j\ne i}x_j-2Bx_i=0\).
	At a symmetric Cournot equilibrium,
	\(x_{1,n}^{C,L}
	=
	\cdots
	=
	x_{n,n}^{C,L}
	=
	x_n^{C,L}\).
	
	Consequently,
	\(A-c-B(n-1)x_n^{C,L}-2Bx_n^{C,L}=0\),
	and hence
	\[
	x_n^{C,L}
	=
	\frac{A-c}{B(n+1)}
	=
	\frac{\widetilde M}{n+1}.
	\]
	
	Moreover,
	\(\frac{\partial^2\Pi_i}{\partial x_i^2}
	=
	-2B<0\).
	Therefore, the first-order condition determines the unique best
	response of each firm, and
	\[
	x_{i,n}^{C,L}
	=
	\frac{A-c}{B(n+1)}
	=
	\frac{\widetilde M}{n+1},
	\qquad i=1,\ldots,n.
	\]
	
	Summing the individual equilibrium outputs gives
	\[
	Q_n^{C,L}
	=
	\sum_{i=1}^{n}x_{i,n}^{C,L}
	=
	\frac{n(A-c)}{B(n+1)}
	=
	\widetilde M\frac{n}{n+1}.
	\]
	Finally,
	\[
		P_n^{C,L}=
		A-BQ_n^{C,L}=
		A-\frac{n(A-c)}{n+1}=
		\frac{A+nc}{n+1}=
		c+\frac{A-c}{n+1}.
	\]
	$ $
	This proves
	\eqref{eq:linear-cournot-individual-output}--%
	\eqref{eq:linear-cournot-price}.
\end{proof}

\begin{proof}[Proof of Proposition~\ref{prop:linear-stackelberg-output}]
	We use backward induction. Suppose that the firms preceding a
	continuation subgame have already chosen the aggregate quantity \(S\).
	Let
	\(R=\widetilde M-S\)
	denote the residual output relative to
	\(\widetilde M=(A-c)/B\).
	
	For \(m\geq0\), let \(L_m(S)\) denote the aggregate output produced
	by a continuation subgame containing \(m\) firms. We prove that
	\begin{equation}\label{eq:app-linear-continuation-output}
		L_m(S)=\lambda_m(\widetilde M-S),
	\end{equation}
	where
	\begin{equation}\label{eq:app-linear-lambda}
		\lambda_m=1-2^{-m},
		\qquad m\geq0.
	\end{equation}
	
	For \(m=0\), no firm remains in the continuation subgame. Hence
	\(L_0(S)=0\), which agrees with
	\(\lambda_0=1-2^0=0\).
	
	Assume that
	\eqref{eq:app-linear-continuation-output} holds for a continuation
	subgame containing \(m-1\) firms. Let the first firm in a continuation
	subgame containing \(m\) firms choose the quantity \(x\). By the
	induction hypothesis, the remaining \(m-1\) firms produce
	\[
	L_{m-1}(S+x)
	=
	\lambda_{m-1}(\widetilde M-S-x)
	=
	\lambda_{m-1}(R-x).
	\]
	
	The reduced payoff of the first firm in this subgame is
	\[
	\begin{aligned}
		\widetilde\Pi_m(x;S)
		&=
		x\left[
		A-B\left(
		S+x+\lambda_{m-1}(R-x)
		\right)
		\right]-cx=
		Bx\left[
		\widetilde M-S-x-\lambda_{m-1}(R-x)
		\right]\\
		&=
		B(1-\lambda_{m-1})x(R-x).
	\end{aligned}
	\]
	
	Differentiating with respect to \(x\), we obtain
	\[
	\frac{\partial\widetilde\Pi_m}{\partial x}(x;S)
	=
	B(1-\lambda_{m-1})(R-2x).
	\]
	
	Moreover,
	\(\frac{\partial^2\widetilde\Pi_m}{\partial x^2}(x;S)
	=
	-2B(1-\lambda_{m-1})<0\).
	Therefore, the unique maximizing quantity is
	\begin{equation}\label{eq:app-linear-current-output}
		x=\frac{R}{2}
		=
		\frac{\widetilde M-S}{2}.
	\end{equation}
	
	The aggregate output of all \(m\) firms in the continuation subgame
	is consequently
	\[
		L_m(S)=
		x+L_{m-1}(S+x)=
		\frac{R}{2}
		+
		\lambda_{m-1}\left(R-\frac{R}{2}\right)=
		\frac{1+\lambda_{m-1}}{2}R.
	\]
	
	Thus,
	\(\lambda_m
	=
	\frac{1+\lambda_{m-1}}{2}\).
	Since \(\lambda_0=0\), induction gives
	\(\lambda_m=1-2^{-m}\).
	This proves \eqref{eq:app-linear-continuation-output} and
	\eqref{eq:app-linear-lambda}.
	
	We now determine the individual equilibrium outputs. Before Firm~1
	moves, the previously chosen aggregate quantity is \(S_0=0\). By
	\eqref{eq:app-linear-current-output}, there holds
	\(x_{1,n}^{S,L}
	=
	\frac{\widetilde M}{2}\).
	After Firm~1 has chosen its quantity, the residual output is
	\(\widetilde M-x_{1,n}^{S,L}
	=
	\frac{\widetilde M}{2}\).
	Each subsequent firm chooses one half of the residual output left by
	its predecessors. 
	
	Hence, if
	\(S_{k-1}
	=
	\sum_{j=1}^{k-1}x_{j,n}^{S,L}\),
	then
	\(x_{k,n}^{S,L}
	=
	\frac{\widetilde M-S_{k-1}}{2}\).
	
	It follows inductively that
	\(\widetilde M-S_{k-1}
	=
	\frac{\widetilde M}{2^{k-1}}\)
	and therefore
	\(x_{k,n}^{S,L}
	=
	\frac{\widetilde M}{2^k}\) for \(k=1,\ldots,n\).
	
	This proves
	\eqref{eq:linear-stackelberg-individual-output}.
	
	Summing these individual quantities, we obtain
	\[
		Q_n^{S,L}=
		\sum_{k=1}^{n}x_{k,n}^{S,L}=
		\widetilde M\sum_{k=1}^{n}2^{-k}=
		\widetilde M\left(1-2^{-n}\right),
	\]
	which proves
	\eqref{eq:linear-stackelberg-aggregate-output}.
	
	Finally,
	\[
		P_n^{S,L}=
		A-BQ_n^{S,L}=
		A-(A-c)\left(1-2^{-n}\right)=
		c+(A-c)2^{-n}.
	\]

This proves \eqref{eq:linear-stackelberg-price}.
\end{proof}

}

\section{Numerical Illustration with the Original Parameter Values}
\label{app:original-numerical-illustration}

This final appendix supplements the normalized examples in the main text by
returning to the original parameter values
\(A=30996\), \(B=\frac{1}{20}\), and \(c=\frac{1}{40}\).
Thus, inverse demand is \(P(Q)=30996-\frac{1}{20}Q\).  These values
change the numerical scale but not the qualitative comparisons established in
the paper.

\subsection{Cournot and Hierarchical Stackelberg Triopoly}

Suppose first that all three firms have quadratic costs
\(C_i(u)=\frac{1}{40}u^2\). Their payoff functions are
\[
 \Pi_i=x_i\left(30996-\frac{1}{20}(x_1+x_2+x_3)\right)
 -\frac{1}{40}x_i^2,\qquad i=1,2,3.
\]
The Cournot first-order conditions are
\[
\left\{
\begin{array}{rcl}
30996-\frac{3}{20}x-\frac{1}{20}y-\frac{1}{20}z&=&0,\\[2mm]
30996-\frac{1}{20}x-\frac{3}{20}y-\frac{1}{20}z&=&0,\\[2mm]
30996-\frac{1}{20}x-\frac{1}{20}y-\frac{3}{20}z&=&0.
\end{array}
\right.
\]
They give \(x_C=y_C=z_C=123984\), so \(Q_C=371952\) and
\(P_C=12398.4\). Each firm earns \(1152902419.20\), and aggregate
producer profit is \(3458707257.60\).

In the hierarchical Stackelberg model, Firm~3's continuation response is
\[
 b_3(x,y)=206640-\frac{x}{3}-\frac{y}{3}.
\]
After substitution, Firm~2 has the reduced payoff
\(\Phi_2(x,y)=y(20664-\frac{x}{30})-\frac{7}{120}y^2\), and hence
\[
 b_2(x)=177120-\frac{2}{7}x.
\]
The derivative of Firm~1's reduced payoff is
\[
 \frac{d}{dx}\Pi_1\bigl(x,b_2(x),b_3(x,b_2(x))\bigr)
 =14760-\frac{41}{420}x.
\]
Backward induction therefore gives
\[
 x_S=151200,\qquad y_S=133920,\qquad z_S=111600.
\]
Consequently, \(Q_S=396720\), \(P_S=11160\), and the three profits
are \(1115856000\), \(1046183040\), and \(934092000\), respectively.

For linear inverse demand, consumer surplus is
\(CS(Q)=\frac{B}{2}Q^2=\frac{1}{40}Q^2\). Hence
\[
 CS_C=3458707257.60,
 \qquad
 CS_S=3934668960.00.
\]
Adding aggregate producer profit gives
\[
 TS_C=6917414515.20,
 \qquad
 TS_S=7030800000.00.
\]

\begin{figure}[!htbp]
 \centering
 \includegraphics[width=0.82\textwidth]{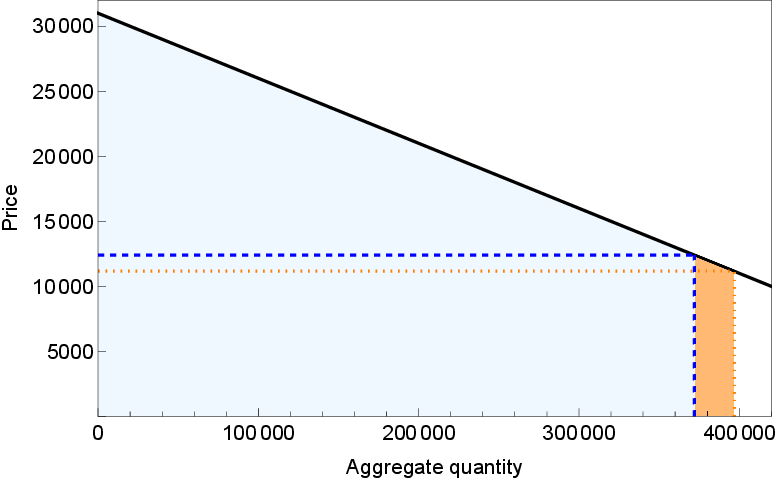}
 \caption{Consumer surplus under the Cournot and hierarchical Stackelberg
 equilibria for the original parameter values. The blue region is the
 Cournot consumer surplus, and the orange region is the additional consumer
 surplus generated by the Stackelberg outcome.}
 \label{fig:original-welfare}
\end{figure}

Figure~\ref{fig:original-welfare} visualizes the larger aggregate output and
lower market price under hierarchical competition. The comparison is
summarized in Table~\ref{tab:original-triopoly-comparison}.

\begin{table}[!htbp]
 \centering
 \caption{Triopoly outcomes for \(A=30996\), \(B=\frac{1}{20}\), and
 \(c=\frac{1}{40}\).}
 \label{tab:original-triopoly-comparison}
 \scriptsize
 \begin{tabular}{lrrrrr}
 \toprule
 Model and outcome & Firm 1 & Firm 2 & Firm 3 & Aggregate & Total surplus\\
 \midrule
 Stackelberg output & 151200 & 133920 & 111600 & 396720 & \\
 Stackelberg profit & 1115856000 & 1046183040 & 934092000 & 3096131040 & 7030800000.00\\
 Stackelberg price  & 11160 & 11160 & 11160 & & \\
 \midrule
 Cournot output & 123984 & 123984 & 123984 & 371952 & \\
 Cournot profit & 1152902419.20 & 1152902419.20 & 1152902419.20 & 3458707257.60 & 6917414515.20\\
 Cournot price & 12398.4 & 12398.4 & 12398.4 & & \\
 \bottomrule
 \end{tabular}
\end{table}

Relative to Cournot competition, the hierarchical Stackelberg structure
raises aggregate output by approximately \(6.66\%\), lowers price by
approximately \(9.99\%\), raises consumer surplus by approximately
\(13.76\%\), and raises total surplus by approximately \(1.64\%\).
Aggregate producer profit falls, while production becomes asymmetric: the
leader produces approximately \(35.48\%\) more than the final follower.

\subsection{Large-Market Numerical Comparison}

We next use the same parameter values in the large-market formulas from
Section~\ref{sec:large-market-comparison}. For linear costs,
\(\widetilde M=(A-c)/B=619919.5\), whereas for quadratic costs,
\(M=A/B=619920\) and \(\delta=c/B=\frac12\). The four aggregate-output
sequences are
\[
 Q_n^{C,L}=\widetilde M\frac{n}{n+1},\qquad
 Q_n^{S,L}=\widetilde M(1-2^{-n}),\qquad
 Q_n^{C,Q}=M\frac{n}{n+2},\qquad
 Q_n^{S,Q}=Mq_n,
\]
where \(q_0=0\) and
\[
 q_n=q_{n-1}+\frac{(1-q_{n-1})^2}{2(1-q_{n-1}+1/2)}.
\]

\begin{table}[!htbp]
 \centering
 \caption{Aggregate equilibrium outputs for the original parameter values.}
 \label{tab:original-output-comparison}
 \begin{tabular}{crrrr}
 \toprule
 \(n\) & \(Q_n^{C,L}\) & \(Q_n^{S,L}\) & \(Q_n^{C,Q}\) & \(Q_n^{S,Q}\)\\
 \midrule
 1  & 309959.750 & 309959.750 & 206640.000 & 206640.000\\
 2  & 413279.667 & 464939.625 & 309960.000 & 324720.000\\
 3  & 464939.625 & 542429.562 & 371952.000 & 396720.000\\
 4  & 495935.600 & 581174.531 & 413280.000 & 443439.784\\
 5  & 516599.583 & 600547.016 & 442800.000 & 475453.242\\
 6  & 531359.571 & 610233.258 & 464940.000 & 498416.947\\
 8  & 551039.556 & 617497.939 & 495936.000 & 528675.946\\
 10 & 563563.182 & 619314.110 & 516600.000 & 547418.025\\
 \bottomrule
 \end{tabular}
\end{table}

\begin{figure}[!htbp]
 \centering
 \includegraphics[width=0.9\textwidth]{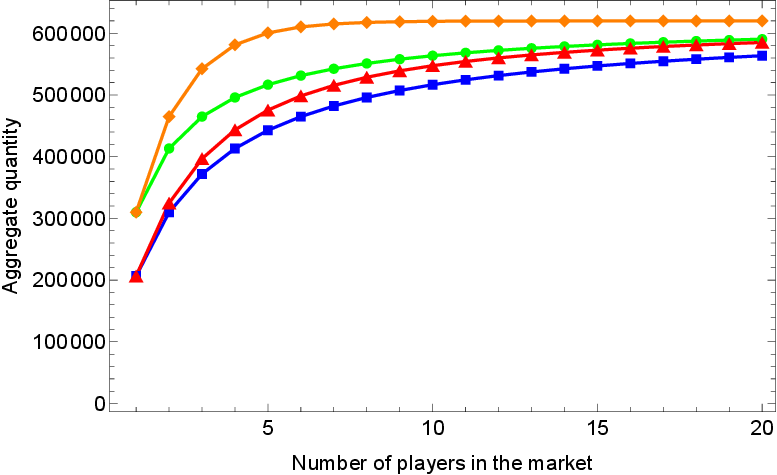}
 \caption{Aggregate quantities for the original parameter values. Green:
 Cournot with linear costs. Blue: Cournot with quadratic costs. Orange:
 Stackelberg with linear costs. Red: Stackelberg with quadratic costs.}
 \label{fig:original-quantities-all}
\end{figure}

Table~\ref{tab:original-output-comparison} and
Figure~\ref{fig:original-quantities-all} confirm that the Stackelberg and
Cournot outputs coincide for one firm, while the Stackelberg aggregate output
is larger for every displayed \(n\geq2\). The corresponding Stackelberg price
is therefore lower.

The asymmetric distribution of Stackelberg output is shown separately for
the two cost specifications in Figures~\ref{fig:original-first-last-linear}
and \ref{fig:original-first-last-quadratic}.

\begin{figure}[!htbp]
 \centering
 \includegraphics[width=0.9\textwidth]{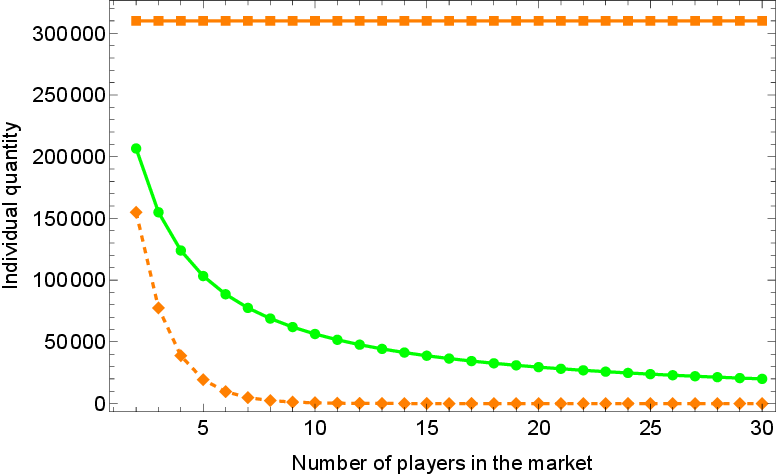}
 \caption{Individual quantities under linear costs. Green: the common
 Cournot quantity. Solid orange: the first Stackelberg firm. Dashed orange:
 the final Stackelberg firm.}
 \label{fig:original-first-last-linear}
\end{figure}

\begin{figure}[!htbp]
 \centering
 \includegraphics[width=0.9\textwidth]{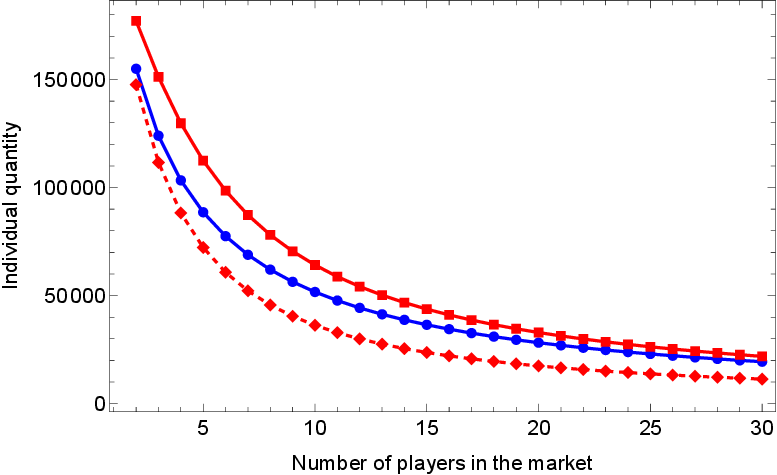}
 \caption{Individual quantities under quadratic costs. Blue: the common
 Cournot quantity. Solid red: the first Stackelberg firm. Dashed red: the
 final Stackelberg firm.}
 \label{fig:original-first-last-quadratic}
\end{figure}

In both panels, the first Stackelberg firm produces more than a Cournot firm,
whereas the final Stackelberg firm produces less. Thus, sequential commitment
affects not only aggregate output but also its distribution across firms.

\begin{table}[!htbp]
 \centering
 \caption{Distances from the corresponding large-market output for the
 original parameter values.}
 \label{tab:original-output-gaps}
 \begin{tabular}{crrrr}
 \toprule
 \(n\) & \(\widetilde M-Q_n^{C,L}\) & \(\widetilde M-Q_n^{S,L}\) & \(M-Q_n^{C,Q}\) & \(M-Q_n^{S,Q}\)\\
 \midrule
 1  & 309959.750 & 309959.750 & 413280.000 & 413280.000\\
 2  & 206639.833 & 154979.875 & 309960.000 & 295200.000\\
 3  & 154979.875 & 77489.938  & 247968.000 & 223200.000\\
 4  & 123983.900 & 38744.969  & 206640.000 & 176480.216\\
 5  & 103319.917 & 19372.484  & 177120.000 & 144466.758\\
 6  & 88559.929  & 9686.242   & 154980.000 & 121503.053\\
 8  & 68879.944  & 2421.561   & 123984.000 & 91244.054\\
 10 & 56356.318  & 605.390    & 103320.000 & 72501.975\\
 \bottomrule
 \end{tabular}
\end{table}

\begin{figure}[!htbp]
 \centering
 \includegraphics[width=0.9\textwidth]{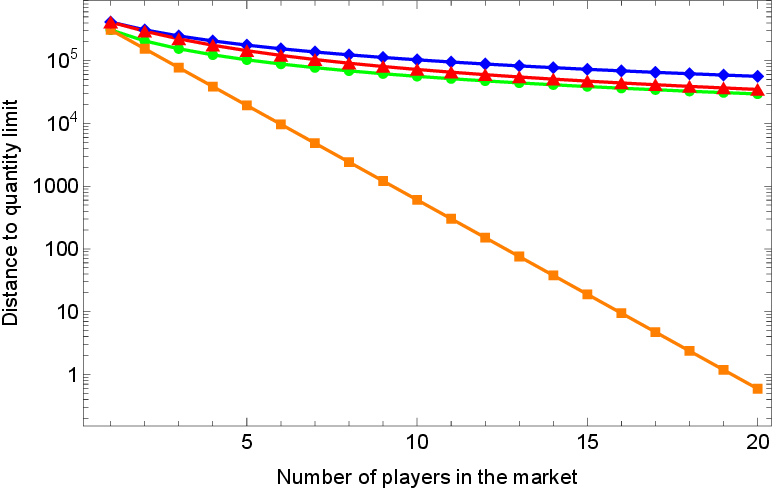}
 \caption{Distances of aggregate quantities from their large-market limits.
 Green and orange correspond to Cournot and Stackelberg with linear costs;
 blue and red correspond to Cournot and Stackelberg with quadratic costs.}
 \label{fig:original-quantity-gaps}
\end{figure}

Table~\ref{tab:original-output-gaps} and
Figure~\ref{fig:original-quantity-gaps} make the convergence-rate comparison
visible. With linear costs, the Stackelberg gap decreases geometrically as
\(\widetilde M2^{-n}\), whereas the Cournot gap is of order \(1/n\). With
quadratic costs, both gaps are of order \(1/n\), but the Stackelberg gap has
the smaller leading constant.

The equilibrium prices are
\[
 P_n^{C,L}=\frac{A+nc}{n+1},\qquad
 P_n^{S,L}=c+(A-c)2^{-n},\qquad
 P_n^{C,Q}=\frac{2A}{n+2},\qquad
 P_n^{S,Q}=A-BQ_n^{S,Q}.
\]

\begin{figure}[!htbp]
 \centering
 \includegraphics[width=0.9\textwidth]{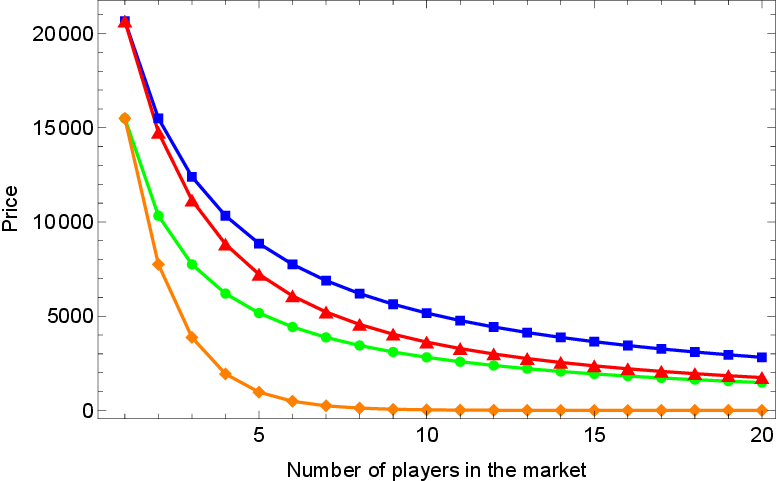}
 \caption{Equilibrium prices for the original parameter values. Green:
 Cournot with linear costs. Blue: Cournot with quadratic costs. Orange:
 Stackelberg with linear costs. Red: Stackelberg with quadratic costs.}
 \label{fig:original-prices-all}
\end{figure}

\begin{figure}[!htbp]
 \centering
 \includegraphics[width=0.9\textwidth]{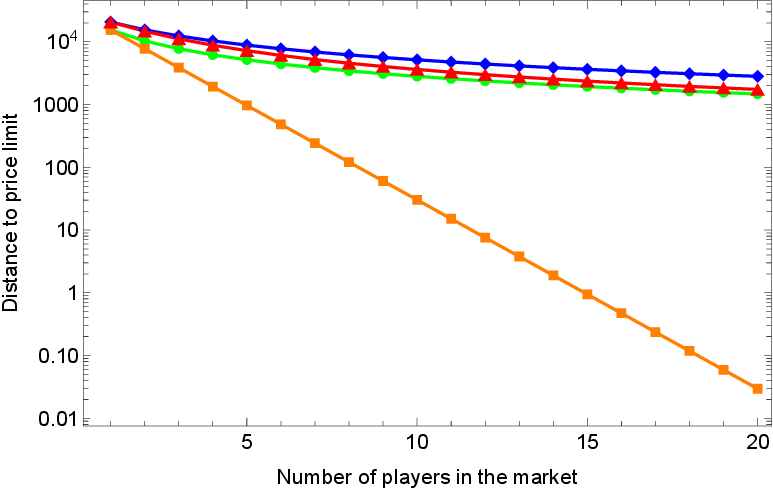}
 \caption{Distances of equilibrium prices from their large-market limits.
 Green and orange correspond to Cournot and Stackelberg with linear costs;
 blue and red correspond to Cournot and Stackelberg with quadratic costs.}
 \label{fig:original-price-gaps}
\end{figure}

Figures~\ref{fig:original-prices-all} and
\ref{fig:original-price-gaps} provide the price counterpart of the output
comparison. Under linear costs, both prices converge to \(c\), with geometric
convergence in the Stackelberg model. Under quadratic costs, both prices
converge to zero at order \(1/n\), and the Stackelberg price is lower for
every finite \(n\geq2\).

\bibliographystyle{plainnat}
\bibliography{martin-06}

\begin{thebibliography}{55}
\providecommand{\natexlab}[1]{#1}
\providecommand{\url}[1]{\texttt{#1}}
\expandafter\ifx\csname urlstyle\endcsname\relax
  \providecommand{\doi}[1]{doi: #1}\else
  \providecommand{\doi}{doi: \begingroup \urlstyle{rm}\Url}\fi

\bibitem[Adhikari et~al.(2025)Adhikari, Basu, and Avittathur]{Adhikari2025}
A.~Adhikari, S.~Basu, and B.~Avittathur.
\newblock Greening innovation, advertising, and pricing decisions under
  competition and market coverage.
\newblock \emph{Journal of Cleaner Production}, 494:\penalty0 144951, 2025.
\newblock \doi{10.1016/j.jclepro.2025.144951}.

\bibitem[Ali et~al.(2025)Ali, Hristov, Ilchev, Kulina, and
  Zlatanov]{Ali-Hristov}
A.~Ali, M.~Hristov, A.~Ilchev, H.~Kulina, and B.~Zlatanov.
\newblock Applications of n-tupled fixed points in partially ordered metric
  spaces for solving systems of nonlinear matrix equations.
\newblock \emph{Mathematics}, 13\penalty0 (13):\penalty0 2125, 2025.
\newblock \doi{10.3390/math13132125}.

\bibitem[Anderson and Engers(1992)]{Anderson_Engers_1992}
Simon~P. Anderson and Maxim Engers.
\newblock Stackelberg versus cournot oligopoly equilibrium.
\newblock \emph{International Journal of Industrial Organization}, 10\penalty0
  (1):\penalty0 127--135, 1992.
\newblock ISSN 0167-7187.
\newblock \doi{https://doi.org/10.1016/0167-7187(92)90052-Z}.
\newblock URL
  \url{https://www.sciencedirect.com/science/article/pii/016771879290052Z}.

\bibitem[Askar(2018)]{Askar2018}
S.S. Askar.
\newblock Tripoly stackelberg game model: One leader versus two followers.
\newblock \emph{Applied Mathematics and Computation}, 328:\penalty0 301--311,
  2018.
\newblock ISSN 0096-3003.
\newblock \doi{https://doi.org/10.1016/j.amc.2018.01.041}.
\newblock URL
  \url{https://www.sciencedirect.com/science/article/pii/S0096300318300626}.

\bibitem[Badev et~al.(2024)Badev, Kabaivanov, Kopanov, Zhelinski, and
  Zlatanov]{Badev_atal_longrun_mobilemarket_math12050724}
Anton Badev, Stanimir Kabaivanov, Petar Kopanov, Vasil Zhelinski, and Boyan
  Zlatanov.
\newblock Long-run equilibrium in the market of mobile services in the usa.
\newblock \emph{Mathematics}, 12\penalty0 (5), 2024.
\newblock ISSN 2227-7390.
\newblock \doi{10.3390/math12050724}.
\newblock URL \url{https://www.mdpi.com/2227-7390/12/5/724}.

\bibitem[Banach(1922)]{Banach}
S.~Banach.
\newblock Sur les opérations dans les ensembles abstraits et leurs
  applications aux intégrales.
\newblock \emph{Fundamenta Mathematicae}, 3\penalty0 (1):\penalty0 133--181,
  1922.
\newblock \doi{10.4064/fm-3-1-133-181}.

\bibitem[Basiri et~al.(2025)Basiri, Abedian, Aghasi, and
  Dashtaali]{Basiri2025258}
R.~Basiri, M.~Abedian, S.~Aghasi, and Z.~Dashtaali.
\newblock A dynamic analysis of the firms in oligopoly market structure: a case
  study.
\newblock \emph{Journal of Modelling in Management}, 20\penalty0 (1):\penalty0
  258--275, 2025.
\newblock \doi{10.1108/JM2-01-2024-0023}.

\bibitem[Berinde and Borcut(2011)]{Borcut-Berinde}
V.~Berinde and M.~Borcut.
\newblock Tripled fixed point theorems for contractive type mappings in
  partially ordered metric spaces.
\newblock \emph{Nonlinear Analysis: Theory, Methods \& Applications},
  74\penalty0 (15):\penalty0 4889--4897, 2011.
\newblock \doi{10.1016/j.na.2011.03.032}.

\bibitem[Bhaskar and Lakshmikantham(2006)]{BL}
T.~G. Bhaskar and V.~Lakshmikantham.
\newblock Fixed point theorems in partially ordered metric spaces and
  applications.
\newblock \emph{Nonlinear Analysis: Theory, Methods \& Applications},
  65\penalty0 (7):\penalty0 1379--1393, 2006.
\newblock \doi{10.1016/j.na.2005.10.017}.

\bibitem[Bischi et~al.(2010)Bischi, Chiarella, Kopel, and Szidarovszky]{BCKS}
G.~Bischi, C.~Chiarella, M.~Kopel, and F.~Szidarovszky.
\newblock \emph{Nonlinear Oligopolies: Stability and Bifurcations}.
\newblock Springer, Berlin--Heidelberg, 2010.
\newblock ISBN 978-3-642-02105-3.
\newblock \doi{10.1007/978-3-642-02106-0}.

\bibitem[Boyer and Moreaux(1986)]{boyer1986perfect}
Marcel Boyer and Michel Moreaux.
\newblock Perfect competition as the limit of a hierarchical market game.
\newblock \emph{Economics Letters}, 22\penalty0 (2-3):\penalty0 115--118, 1986.

\bibitem[Boyer and Moreaux(1987)]{BoyerMoreaux1987}
Marcel Boyer and Michel Moreaux.
\newblock On stackelberg equilibria with differentiated products: The critical
  role of the strategy space.
\newblock \emph{Journal of Industrial Economics}, 36\penalty0 (2):\penalty0
  217--230, 1987.
\newblock \doi{10.2307/2098414}.

\bibitem[Byrne et~al.(2024)Byrne, de~Roos, Lewis, Marx, and
  Wu]{ByrneDeRoos2024}
D.~P. Byrne, N.~de~Roos, M.~S. Lewis, L.~M. Marx, and X.~Wu.
\newblock Asymmetric information sharing in oligopoly: Evidence and
  implications.
\newblock Technical report, SSRN, 2024.

\bibitem[Cellini and Lambertini(2004)]{Cellini-Lambertini}
R.~Cellini and L.~Lambertini.
\newblock Dynamic oligopoly with sticky prices: Closed-loop, feedback and
  open-loop solutions.
\newblock \emph{Journal of Dynamical and Control Systems}, 10\penalty0
  (3):\penalty0 303--314, 2004.
\newblock \doi{10.1023/B:JODS.0000034432.46970.64}.

\bibitem[Chatterjea(1972)]{Chatterjea}
S.~K. Chatterjea.
\newblock Fixed-point theorems.
\newblock \emph{Comptes Rendus de l'Acad\'{e}mie Bulgare des Sciences},
  25:\penalty0 727--730, 1972.

\bibitem[Cournot(1960)]{Cournot2}
A.-A. Cournot.
\newblock \emph{Researches into the Mathematical Principles of the Theory of
  Wealth}.
\newblock Dover Publications, New York, 1960.
\newblock ISBN 978-0-486-66095-7.

\bibitem[Cumbul(2021)]{Cumbul2021}
Eray Cumbul.
\newblock Stackelberg versus cournot oligopoly with private information.
\newblock \emph{International Journal of Industrial Organization}, 74:\penalty0
  102674, 2021.
\newblock \doi{10.1016/j.ijindorg.2020.102674}.

\bibitem[Dianetti(2025)]{Dianetti}
J.~Dianetti.
\newblock Linear-quadratic-singular stochastic differential games and
  applications.
\newblock \emph{Decisions in Economics and Finance}, 48\penalty0 (1):\penalty0
  381–413, 2025.
\newblock \doi{10.1007/s10203-023-00422-0}.

\bibitem[Dzhabarova et~al.(2020)Dzhabarova, Kabaivanov, Ruseva, and
  Zlatanov]{DKRZ}
Y.~Dzhabarova, S.~Kabaivanov, M.~Ruseva, and B.~Zlatanov.
\newblock Existence, uniqueness and stability of market equilibrium in
  oligopoly markets.
\newblock \emph{Administrative Sciences}, 10\penalty0 (3):\penalty0 70, 2020.
\newblock \doi{10.3390/admsci10030070}.

\bibitem[Economides(1993)]{Economides1993}
Nicholas Economides.
\newblock Quantity leadership and social inefficiency.
\newblock \emph{International Journal of Industrial Organization}, 11\penalty0
  (2):\penalty0 219--237, 1993.
\newblock \doi{10.1016/0167-7187(93)90005-W}.

\bibitem[Escrihuela-Villar(2009)]{EscrihuelaVillar2009}
Marc Escrihuela-Villar.
\newblock A note on cartel stability and endogenous sequencing with tacit
  collusion.
\newblock \emph{Journal of Economics}, 96\penalty0 (2):\penalty0 137--147,
  2009.
\newblock \doi{10.1007/s00712-008-0053-8}.

\bibitem[Friedman(2009)]{Friedman}
J.~Friedman.
\newblock \emph{Oligopoly Theory}.
\newblock Cambridge University Press, Cambridge, 2009.
\newblock ISBN 978-0-521-11316-6.
\newblock \doi{10.1017/CBO9780511571893}.

\bibitem[Geraskin(2024)]{Geraskin2024627}
M.I. Geraskin.
\newblock Quantity conjectural variations in oligopoly games under different
  demand and cost functions and multilevel leadership.
\newblock \emph{Automation and Remote Control}, 85\penalty0 (7):\penalty0
  627--640, 2024.
\newblock \doi{10.1134/S0005117924700061}.

\bibitem[Guo and Lakshmikantham(1987)]{GL}
D.~Guo and V.~Lakshmikantham.
\newblock Coupled fixed points of nonlinear operators with applications.
\newblock \emph{Nonlinear Analysis: Theory, Methods \& Applications},
  11\penalty0 (5):\penalty0 623--632, 1987.
\newblock \doi{10.1016/0362-546X(87)90077-0}.

\bibitem[Hardy and Rogers(1973)]{Hardy-Rogers}
G.~Hardy and T.~Rogers.
\newblock A generalization of a fixed point theorem of reich.
\newblock \emph{Canadian Mathematical Bulletin}, 16\penalty0 (2):\penalty0
  201--206, 1973.
\newblock \doi{10.4153/CMB-1973-036-0}.

\bibitem[Ilchev et~al.(2024)Ilchev, Ivanova, Kulina, Yaneva, and
  Zlatanov]{IIKYZ}
A.~Ilchev, V.~Ivanova, H.~Kulina, P.~Yaneva, and B.~Zlatanov.
\newblock Investigation of equilibrium in oligopoly markets with the help of
  tripled fixed points in banach spaces.
\newblock \emph{Econometrics}, 12\penalty0 (2):\penalty0 18, 2024.
\newblock \doi{10.3390/econometrics12020018}.

\bibitem[Kabaivanov et~al.(2022)Kabaivanov, Zhelinski, and
  Zlatanov]{Kabaivanov}
S.~Kabaivanov, V.~Zhelinski, and B.~Zlatanov.
\newblock Coupled fixed points for hardy--rogers type of maps and their
  applications in the investigations of market equilibrium in duopoly markets
  for non-differentiable, nonlinear response functions.
\newblock \emph{Symmetry}, 14\penalty0 (3):\penalty0 605, 2022.
\newblock \doi{10.3390/sym14030605}.

\bibitem[Kannan(1968)]{Kannan}
R.~Kannan.
\newblock Some results on fixed points.
\newblock \emph{Bulletin of the Calcutta Mathematical Society}, 60:\penalty0
  71--76, 1968.

\bibitem[Kirk et~al.(2003)Kirk, Srinivasan, and Veeramani]{KSV}
W.~Kirk, P.~Srinivasan, and P.~Veeramani.
\newblock Fixed points for mappings satisfying cyclical contractive condition.
\newblock \emph{Fixed Point Theory}, 4\penalty0 (1):\penalty0 179–189, 2003.

\bibitem[Konrad and Leininger(2007)]{KonradLeininger2007}
Kai~A. Konrad and Wolfgang Leininger.
\newblock The generalized stackelberg equilibrium of the all-pay auction with
  complete information.
\newblock \emph{Review of Economic Design}, 11\penalty0 (2):\penalty0 165--174,
  2007.
\newblock \doi{10.1007/s10058-007-0029-0}.

\bibitem[Li(2025)]{Li2025918}
M.~Li.
\newblock Strategic decision-making in oligopoly markets based on stackelberg
  game and reinforcement learning algorithms.
\newblock pages 918--923, 2025.
\newblock \doi{10.1109/ICICR65456.2025.00162}.

\bibitem[Matsumoto and Szidarovszky(2018)]{MS}
A.~Matsumoto and F.~Szidarovszky.
\newblock \emph{Dynamic Oligopolies with Time Delays}.
\newblock Springer, Singapore, 2018.
\newblock ISBN 978-981-13-1785-9.
\newblock \doi{10.1007/978-981-13-1786-6}.

\bibitem[Ohnishi(2021)]{Ohnishi2021}
Kazuhiro Ohnishi.
\newblock A note on stackelberg mixed triopoly games with state-owned,
  labor-managed and capitalist firms.
\newblock \emph{International Game Theory Review (IGTR)}, 23\penalty0
  (01):\penalty0 1--10, March 2021.
\newblock \doi{10.1142/S021919892150002X}.
\newblock URL
  \url{https://ideas.repec.org/a/wsi/igtrxx/v23y2021i01ns021919892150002x.html}.

\bibitem[Pepall et~al.(2014)Pepall, Richards, and Norman]{Pepall}
L.~Pepall, D.~J. Richards, and G.~Norman.
\newblock \emph{Industrial Organization: Contemporary Theory and Empirical
  Applications}.
\newblock Wiley, Hoboken, NJ, 5 edition, 2014.
\newblock ISBN 978-1-118-84403-9.

\bibitem[Petru\c{s}el(2018)]{Petrusel}
A.~Petru\c{s}el.
\newblock Fixed points vs. coupled fixed points.
\newblock \emph{Journal of Fixed Point Theory and Applications}, 20\penalty0
  (4):\penalty0 150, 2018.
\newblock \doi{10.1007/s11784-018-0630-6}.

\bibitem[Petru\c{s}el et~al.(2018)Petru\c{s}el, Petru\c{s}el, Xiao, and
  Yao]{Petrusel18b}
A.~Petru\c{s}el, G.~Petru\c{s}el, Y.-B. Xiao, and J.-C. Yao.
\newblock Fixed point theorems for generalized contractions with applications
  to coupled fixed point theory.
\newblock \emph{Journal of Nonlinear and Convex Analysis}, 19\penalty0
  (1):\penalty0 71--88, 2018.

\bibitem[Reich(1971)]{Reich}
S.~Reich.
\newblock Kannan's fixed point theorem.
\newblock \emph{Bulletin of the Unione Matematica Italiana}, 4\penalty0
  (1):\penalty0 1--11, 1971.

\bibitem[Rhee(2025)]{Rhee2025}
K.~E. Rhee.
\newblock Competitive price discrimination in asymmetric oligopoly.
\newblock \emph{Review of Industrial Organization}, 67\penalty0 (1):\penalty0
  83--110, 2025.
\newblock \doi{10.1007/S11151-024-10004-Y}.

\bibitem[Robson(1990)]{Robson1990}
Arthur~J. Robson.
\newblock Stackelberg and marshall.
\newblock \emph{American Economic Review}, 80\penalty0 (1):\penalty0 69--82,
  1990.
\newblock URL \url{https://www.jstor.org/stable/2006734}.

\bibitem[Salahmanesh et~al.(2024)Salahmanesh, Farazmand, Anvari, and
  Ahmadi]{Salahmanesh20241203}
A.~Salahmanesh, H.~Farazmand, E.~Anvari, and A.~Ahmadi.
\newblock The impact of market penetration costs and rival countries' exports
  on iran's cement export profits in an oligopoly framework.
\newblock \emph{Iranian Economic Review}, 28\penalty0 (4):\penalty0 1203--1227,
  2024.
\newblock \doi{10.22059/ier.2024.346545.1007507}.

\bibitem[Samet and Vetro(2010)]{Samet-Vetro}
B.~Samet and C.~Vetro.
\newblock Coupled fixed point, f--invariant set and fixed point of n-order.
\newblock \emph{Annals of Functional Analysis}, 1\penalty0 (2):\penalty0
  46--56, 2010.
\newblock \doi{10.15352/afa/1399900586}.

\bibitem[Sherali(1984)]{Sherali1984}
Hanif~D. Sherali.
\newblock A multiple leader stackelberg model and analysis.
\newblock \emph{Operations Research}, 32\penalty0 (2):\penalty0 390--404, 1984.
\newblock ISSN 0030364X, 15265463.
\newblock URL \url{http://www.jstor.org/stable/172107}.

\bibitem[Sintunavarat and Kumam(2012)]{SK}
W.~Sintunavarat and P.~Kumam.
\newblock Coupled best proximity point theorem in metric spaces.
\newblock \emph{Fixed Point Theory and Applications}, 2012:\penalty0 93, 2012.
\newblock \doi{10.1186/1687-1812-2012-93}.

\bibitem[Smith(1987)]{Smith}
A.~Smith.
\newblock \emph{A Mathematical Introduction to Economics}.
\newblock Basil Blackwell, Oxford, 1987.
\newblock ISBN 978-0-631-14568-7.

\bibitem[Tesoriere(2017{\natexlab{a}})]{Tesoriere2017Setup}
Antonio Tesoriere.
\newblock Stackelberg equilibrium with multiple firms and setup costs.
\newblock \emph{Journal of Mathematical Economics}, 73:\penalty0 86--102,
  2017{\natexlab{a}}.
\newblock \doi{10.1016/j.jmateco.2017.09.002}.

\bibitem[Tesoriere(2017{\natexlab{b}})]{Tesoriere2017Zero}
Antonio Tesoriere.
\newblock Stackelberg equilibrium with many leaders and followers: The case of
  zero fixed costs.
\newblock \emph{Research in Economics}, 71\penalty0 (1):\penalty0 102--117,
  2017{\natexlab{b}}.
\newblock \doi{10.1016/j.rie.2016.11.004}.

\bibitem[Tirole(1988)]{Tirole}
J.~Tirole.
\newblock \emph{The Theory of Industrial Organization}.
\newblock MIT Press, Cambridge, MA, 1988.
\newblock ISBN 978-0-262-20071-4.

\bibitem[Tron and Th{\'e}ra(2026)]{TronThera2026}
Nguyen~Huu Tron and Michel Th{\'e}ra.
\newblock n-tupled regularity and cyclic fixed points of multivalued mappings
  in metric spaces.
\newblock \emph{Decisions in Economics and Finance}, 2026.
\newblock \doi{10.1007/s10203-026-00568-7}.

\bibitem[Varian(2014)]{Varian}
H.~R. Varian.
\newblock \emph{Intermediate Microeconomics: A Modern Approach}.
\newblock W. W. Norton \& Company, New York, 9 edition, 2014.
\newblock ISBN 978-0-393-93424-0.

\bibitem[Vives(1988)]{Vives}
X.~Vives.
\newblock Sequential entry, industry structure and welfare.
\newblock \emph{European Economic Review}, 32\penalty0 (8):\penalty0
  1671--1687, 1988.

\bibitem[von Stackelberg(2011)]{Stackelberg}
H.~von Stackelberg.
\newblock \emph{Market Structure and Equilibrium}.
\newblock Springer-Verlag, Berlin and Heidelberg, 2011.
\newblock ISBN 978-3-642-12585-0.

\bibitem[Wu et~al.(2025)Wu, Yu, and Yang]{Wu-Yu-Yang}
Y.~Wu, T.~Yu, and N.~Yang.
\newblock Price-based demand response in retail electricity trading with
  prosumers in homogeneous oligopoly market: A game-theoretical approach.
\newblock \emph{International Journal of Electrical Power and Energy Systems},
  172:\penalty0 111335, 2025.
\newblock \doi{10.1016/j.ijepes.2025.111335}.

\bibitem[Yan et~al.(2024)Yan, Zhang, Zhang, Deng, and Gao]{Yan2024}
J.~Yan, J.~Zhang, L.~Zhang, C.~Deng, and T.~Gao.
\newblock Individual and cluster demand response in retail electricity trading
  with end-users in differentiated oligopoly market: A game-theoretical
  approach.
\newblock \emph{International Journal of Electrical Power and Energy Systems},
  161:\penalty0 110118, 2024.
\newblock \doi{10.1016/j.ijepes.2024.110118}.

\bibitem[Yousefimanesh et~al.(2023)Yousefimanesh, Bos, and
  Vermeulen]{YousefimaneshBosVermeulen2023}
Niloofar Yousefimanesh, Iwan Bos, and Dries Vermeulen.
\newblock Strategic rationing in stackelberg games.
\newblock \emph{Games and Economic Behavior}, 140:\penalty0 529--555, 2023.
\newblock \doi{10.1016/j.geb.2023.05.001}.

\bibitem[Zlatanov(2021)]{Zlatanov-FPT}
B.~Zlatanov.
\newblock Coupled best proximity points for cyclic contractive maps and their
  applications.
\newblock \emph{Fixed Point Theory}, 22:\penalty0 431--452, 2021.
\newblock \doi{10.24193/fpt-ro.2021.1.29}.

\end{thebibliography}

\end{document}